\documentclass[12pt, reqno]{amsart}
\usepackage{amsmath}
\usepackage{amssymb}
\usepackage{epsfig}
\usepackage{graphicx}
\usepackage{color}
\usepackage{fullpage}
\definecolor{shadecolor}{gray}{0.875}
\usepackage{amscd}

\numberwithin{equation}{section}

\input xy
\xyoption{all}

\calclayout
\allowdisplaybreaks[3]

\theoremstyle{plain}
\newtheorem{prop}{Proposition}[section]

\newtheorem{theo}[prop]{Theorem}
\newtheorem{coro}[prop]{Corollary}

\newtheorem{lemm}[prop]{Lemma}

\theoremstyle{definition}
\newtheorem{defi}[prop]{Definition}

\newtheorem{conj}[prop]{Conjecture}

\newtheorem{assum}[prop]{Assumption}

\newtheorem{rema}[prop]{Remark}

\newtheorem{exam}[prop]{Example}

\def\bR{{\mathbb R}}

\def\Pic{\mathrm{Pic}}

\def\et{\mathrm{et}}

\def\vol{\mathrm{vol}}

\def\Nef{\mathrm{Nef}}
\def\Hilb{\mathrm{Hilb}}

\def\Pic{\mathrm{Pic}}
\def\Gal{\mathrm{Gal}}
\def\Sing{\mathrm{Sing}}
\def\Span{\mathrm{Span}}
\def\Eff{\overline{\mathrm{Eff}}}
\def\et{\textrm{\'et}}
\def\Spec{\mathrm{Spec}}

\makeatother
\makeatletter

\author{Brian Lehmann}
\address{Department of Mathematics \\
Boston College  \\
Chestnut Hill, MA \, \, 02467}
\email{lehmannb@bc.edu}

\author{Akash Kumar Sengupta}
\address{Department of Mathematics, Columbia University, New York, NY 10027}
\email{akashs@math.columbia.edu}

\author{Sho Tanimoto}
\address{Graduate School of Mathematics, Nagoya University, Furocho Chikusa-ku, Nagoya, 464-8602, Japan}
\email{sho.tanimoto@math.nagoya-u.ac.jp}

\title[Geometric consistency of Manin's conjecture]{Geometric consistency of Manin's Conjecture}

\begin{document}

\begin{abstract}

We conjecture that the exceptional set in Manin's Conjecture has an explicit geometric description. Our proposal includes the rational point contributions from any generically finite map with larger geometric invariants.  We prove that this set is contained in a thin subset of rational points, verifying there is no counterexample to Manin's Conjecture which arises from an incompatibility of geometric invariants.

\end{abstract}

\maketitle

\date{\today}

\tableofcontents

\section{Introduction} 
\label{secct:intro}

Let $X$ be a geometrically integral smooth projective Fano variety over a number field $F$ and let $\mathcal{L} = \mathcal{O}_{X}(L)$ be an adelically metrized big and nef line bundle on $X$. Manin's Conjecture, first formulated in \cite{FMT89} and \cite{BM}, predicts that the growth in the number of rational points on $X$ of bounded $\mathcal{L}$-height is controlled by two geometric constants $a(X,L)$ and $b(F,X,L)$.  These constants are defined for any smooth projective variety $X$ and any big and nef divisor $L$ on $X$ as
\begin{equation*}
a(X,L) = \min \{ t\in \bR \mid  K_X + tL \in \Eff^{1}(X) \}
\end{equation*}
and
\begin{align*}
b(F, X,L) = & \textrm{ the codimension of the minimal supported face} \\
& \textrm{ of }  \Eff^{1}(X) \textrm{ containing } K_{X} + a(X, L)L
\end{align*}
where $\Eff^{1}(X)$ is the pseudo-effective cone of divisors of $X$.  If $L$ is nef but not big, we set $a(X,L) =b(F, X, L) =  \infty$.

\begin{rema}
One can define the $a$ and $b$ invariants analogously for any big divisor $L$ on $X$, and it is natural to ask whether they still control the behavior of asymptotic point counts for the associated height function.  However, in this situation the invariants can exhibit pathological behavior so that we cannot expect Manin's Conjecture to hold.  In particular this pathological behavior can occur for the anticanonical divisor on almost Fano varieties as defined in \cite[Section 3]{Peyre03}.  See Section \ref{sect: nonbigdiv} for more details.
\end{rema}

For any subset $Q \subset X(F)$, we let $N(Q, \mathcal L, T)$ denote the number of rational points on $Q$ whose height associated to $\mathcal L$ is bounded above by $T$. Manin's Conjecture predicts the asymptotic formula for the counting function $N(Q, \mathcal L, T)$ for a suitable choice of $Q$, and it has been formulated mainly by Batyrev, Manin, Peyre, and Tschinkel in a series of papers \cite{FMT89}, \cite{BM}, \cite{Peyre}, \cite{BT}, and \cite{Peyre03}. The conjecture relies upon the notion of a thin set: a subset of $X(F)$ which is a finite union of sets of the form $f(Y(F))$ where $f : Y \to X$ is a generically finite morphism from a variety $Y$ such that $f$ does not admit a rational section. (See Definition~\ref{defi:thinmaps} for more details.)

\begin{conj}[Manin's Conjecture] \label{conj: maninsconjecture_intro}
Let $F$ be a number field.  Let $X$ be a geometrically rationally connected and geometrically integral smooth projective variety defined over $F$ and let $\mathcal{L}$ be a big and nef line bundle with an adelic metrization on $X$.

Suppose that $X(F)$ is not a thin set. Then there exists a thin set $Z \subset X(F)$ such that we have
\[
N(X(F) \setminus Z, \mathcal L, T) \sim c(F, Z, \mathcal L)T^{a(X, L)} \log (T)^{b(F, X, L)-1}
\]
as $T \rightarrow \infty$ where $c(F, Z, \mathcal L)$ is Peyre-Batyrev-Tschinkel's constant introduced in \cite{Peyre} and \cite{BT}.
\end{conj}

In the statement of Manin's Conjecture an exceptional set $Z$ of rational points must be removed in order to obtain the expected growth rate.  
For example, it is possible for points to grow more quickly than predicted along certain subvarieties of $X$ and such points should not be counted.  More precisely, the following definition identifies the possible geometric obstructions to Manin's Conjecture.

\begin{defi}
Let $X$ be a smooth projective variety over a field $F$ of characteristic $0$ and let $L$ be a big and nef divisor on $X$.  A morphism of  smooth projective varieties $f: Y \to X$ is called a breaking thin map if it satisfies the following two conditions:
\begin{enumerate}
\item $f$ is generically finite onto its image, and
\item $(a(Y,f^{*}L),b(F, Y,f^{*}L)) > (a(X,L),b(F, X,L))$ in the lexicographic order.
\end{enumerate}
Note that this definition implicitly depends on the choice of $L$. 
\end{defi}

If Manin's Conjecture is self-consistent then the exceptional set should include all subsets of the form $f(Y(F))$ where $f: Y \to X$ is a breaking thin map.  However, the point contributions from breaking thin maps need not lie on a Zariski closed proper subset of $X$.  \cite{BT-cubic} used this idea to show that the exceptional set in Manin's Conjecture can be Zariski dense, contradicting the original formulation of the conjecture. 
\cite{Peyre03} conjectured that Manin's Conjecture should be revised by allowing the exceptional set to be a thin set of points, and this version was subsequently verified in a few examples (\cite{LeRudulier}, \cite{BHB18}). 

Our main theorem shows that point contributions from breaking thin maps will always be contained in a thin set.  This gives strong support to the conjecture of \cite{Peyre03}: one can never construct a counterexample to the thin set version of Manin's Conjecture using breaking thin maps. 

\begin{theo} \label{theo: maintheorem}
Let $X$ be a geometrically uniruled smooth projective variety over a field $F$ of characteristic $0$ and let $L$ be a big and nef divisor on $X$.  As we vary over all breaking thin $F$-maps $f: Y \to X$, the points
\begin{equation*}
\bigcup_{f} f(Y(F))
\end{equation*}
are contained in a thin subset of $X(F)$.
\end{theo}

\begin{rema}
In the proof of Theorem \ref{theo: maintheorem} we explicitly identify the finite set of maps which contribute to the thin set.  Thus our results are relevant for all fields $F$ of characteristic $0$, even those for which the notion of a thin set is not meaningful.  For example, over an algebraically closed field the geometric result Theorem \ref{theo: mainfiniteness} can be seen as an alternative to Theorem \ref{theo: maintheorem}.
\end{rema}

Theorem \ref{theo: maintheorem} generalizes earlier partial results in \cite{BT}, \cite{HTT15}, \cite{LTT14}, \cite{HJ16},  \cite{LTDuke}, and \cite{Sen17}.  These papers also establish some practical techniques for computing this thin set. See \cite{LTRMS} for some computational aspects of the above exceptional sets.

In fact, we prove a more precise statement (Theorem \ref{theo: precisetheorem}) which also addresses generically finite morphisms $f: Y \to X$ such that $Y$ has the same $a$ and $b$ invariants as $X$.  In this situation the rational point contributions from $Y$ can affect the leading constant in Manin's Conjecture (as in \cite{LeRudulier}, \cite{BHB18}), and thus one must decide whether or not $f(Y(F))$ should be included in the exceptional set in order to obtain the constant predicted by \cite{Peyre} and \cite{BT}.
We distinguish the two possibilities using the geometric notion of a face contracting morphism (Definition \ref{defi: facecontraction}).

Finally, we conjecture that the exceptional set in Manin's Conjecture will actually coincide with the geometrically defined subset of $X(F)$ constructed in Theorem \ref{theo: precisetheorem}.  Conjecture \ref{conj: maninsconjecture} gives a precise formulation of the exceptional set using geometric invariants.  In Section \ref{sect: conjecturaldescription} we verify this in many examples where Manin's Conjecture is known to hold.

\subsection{A summary of the proof}
To prove Theorem \ref{theo: maintheorem}, it would suffice to show that there is a finite set of breaking thin maps $\{ f_{i}: Y_{i} \to X \}_{i=1}^{r}$ such that
\begin{equation*}
\bigcup_{f} f(Y(F)) = \bigcup_{i=1}^{r} f_{i}(Y_{i}(F)).
\end{equation*}
In particular, it would suffice to show that there is a finite set of breaking thin maps $\{ f_{i} \}$ such that every breaking thin map $f: Y \to X$ factors through some $f_{i}$.  Our proof is built on this idea.

The first step is to prove a factoring result for breaking thin maps over an algebraically closed field of characteristic $0$.  However, to obtain a factoring for the map $f: Y \to X$ we will need to allow ourselves to alter the variety $Y$ in the following way.

\begin{defi} 
Let $Y$ be a smooth projective variety over a field of characteristic $0$ and let $L$ be a big and nef $\mathbb{Q}$-divisor on $X$.  Let $\pi: Y \dashrightarrow W$ be the canonical model associated to $K_{Y} + a(Y,L)L$ and let $U$ be the maximal open subset of $Y$ where $\pi$ is defined. Suppose that $T \to W$ is a dominant morphism of normal projective varieties.  Then there is a unique component of $T \times_W U$ mapping dominantly to $T$.   
Denote by $\widetilde{Y}$ the normalization of the Zariski closure of this component in $T \times Y$.

We call such an $\widetilde{Y}$ an Iitaka base change of $Y$; it is naturally equipped with maps $\widetilde{Y} \to T$ and $\widetilde{Y} \to Y$.  
\end{defi}

The following theorem extends earlier results of \cite{HJ16}, \cite{LTDuke}, and \cite{Sen17} to the most general setting. 

\begin{theo} \label{theo: mainfiniteness}
Let $X$ be a uniruled smooth projective variety over an algebraically closed field of characteristic $0$. Let $L$ be a big and nef $\mathbb{Q}$-divisor on $X$.  There is a finite set of breaking thin maps $\{ f_{i}: Y_{i} \to X\}$ such that for any breaking thin map $f: Y \to X$ either the image of $f : Y \to X$ is contained in the augmented base locus $\mathbf B_+(L)$ or there is an Iitaka base change $\widetilde{Y}$ of $Y$ with respect to $f^{*}L$ such that the induced morphism $\widetilde{f}: \widetilde{Y} \to X$ factors rationally through one of the $f_{i}$.
\end{theo}

The key input is Birkar's solution of the Borisov-Alexeev-Borisov Conjecture (\cite{birkar16} and \cite{birkar16b}).  Although the varieties $Y$ in Theorem \ref{theo: mainfiniteness} need not form a bounded family, using Birkar's result we show that their images in $X$ are covered by a set of adjoint rigid subvarieties (Definition \ref{defi: adjointrigid}) which do form a bounded family.  Furthermore, \cite{Sen17} controls the possible ramification loci of such morphisms $f$ and their degrees are again bounded by the Borisov-Alexeev-Borisov Conjecture.  Thus we can construct the $\{f_{i}\}$ in Theorem \ref{theo: mainfiniteness} using suitably chosen covers over the universal family of adjoint rigid subvarieties. Then the corresponding factoring property follows from the homotopy lifting property of covering spaces.

The second step is to ``descend'' Theorem \ref{theo: mainfiniteness} from the algebraic closure $\overline{F}$ to the ground field $F$.  There are two main obstacles.  First, infinitely many twists over $F$ can be identified with a single map over $\overline{F}$.  Thus it is more natural to allow ourselves to work with all twists of a finite set of maps $\{ f_{i}: Y_{i} \to X \}$ when proving our factoring result.  Second, it is quite difficult to determine when the maps constructed by Theorem \ref{theo: mainfiniteness} descend to the ground field.  Even when they do descend, the corresponding homotopy lifting property may not be available over the ground field, and thus it is unclear whether the desired factoring holds. Fortunately, we only care about the situation when our varieties are equipped with a rational point.  Using a delicate construction involving the arithmetic \'etale fundamental group and the homotopy lifting property for a rational basepoint, we prove a factoring result over an arbitrary field of characteristic $0$ (Lemma~\ref{lemm: finitelymanycoversovernf}) in the situation when $Y$ admits a rational point.
Here we state a shorter version of Lemma~\ref{lemm: finitelymanycoversovernf} to give readers a sense of this lemma. We construct universal families of generically finite maps which satisfy a certain universal property up to twisting:

\begin{lemm} \label{lemm: finitelymanycoversovernf_intro}
Let $X$ be a geometrically uniruled geometrically integral smooth projective variety defined over a field $F$ of characteristic $0$ and let $L$ be a big and nef $\mathbb Q$-divisor on $X$. Let $p : X \rightarrow W$ be a surjective morphism between projective varieties.  Suppose that any general fiber over $W$ has the same $a$-invariant with respect to $L$ as $X$ does with respect to $L$. Then there exist a non-empty open subset $W^\circ \subset W$ (with preimage $X^\circ \subset X$), a proper closed subset $C \subsetneq X$, and a finite set of dominant generically finite morphisms  $\{ f_{j}: \mathcal{Y}_{j} \to X \}$ defined over $F$ that fit into commutative diagrams
\begin{equation*}
\xymatrix{ \mathcal{Y}_{j} \ar[r]^{f_{j}} \ar[d]_{q_{j}} &  X \ar[d]_{p} \\
T_{j} \ar[r] & W}
\end{equation*}
such that the following holds.
(1) $p : X^\circ \to W^\circ$ is a good family of adjoint rigid varieties (See Definition~\ref{defi:goodfamily} for the definition for good families of adjoint rigid varieties); (2) Suppose that $q: \mathcal{Y} \to T$ is a projective surjective morphism of varieties over $F$ where $\mathcal{Y}$ is smooth and geometrically integral and that we have a diagram
\begin{equation*}
\xymatrix{ \mathcal{Y} \ar[r]^{f} \ar[d]_{q} &  X \ar[d]_{p} \\
T \ar[r]^{g} & W}
\end{equation*}
 satisfying the following properties:
\begin{enumerate}
\item There is some open subset $T' \subset T$ such that $\mathcal{Y}$ is a good family of adjoint rigid varieties over $T'$ and the map $f: q^{-1}(T') \to X$ has image in $X^\circ$ and is a good morphism of good families. (See Definition~\ref{defi:goodfamily} for the definition for good morphisms.)
\item There is a rational point $y \in \mathcal{Y}(F)$ such that $f(y) \not \in C$.
\end{enumerate}
Then for some index $j$ there will be a twist $f_{j}^\sigma : \mathcal Y_j^\sigma \rightarrow X$ such that $f(y) \in f_{j}^{\sigma}(\mathcal{Y}_{j}^{\sigma}(F))$.  
Furthermore, there is a dominant generically finite map $\widetilde{T} \to T$ such that the main component $\widetilde{q}: \widetilde{\mathcal{Y}} \to \widetilde{T}$ of the base change of $q$ by $\widetilde{T} \to T$ satisfies that the induced map $\widetilde{f}: \widetilde{\mathcal{Y}} \to \mathcal{U}$ will factor rationally through $f_{j}^{\sigma}$ and a general geometric fiber of $\widetilde{q}$ will map birationally to a geometric fiber of the map $q_{j}^{\sigma}: \mathcal{Y}_{j}^{\sigma} \to T_{j}^{\sigma}$.
\end{lemm}

Using the above lemma we have now established that $\cup_{f} f(Y(F))$ is contained in the rational points coming from the twists of a finite set of thin maps.  The final step is to show that if we fix a morphism $f_{i}: Y_{i} \to X$, all of its twists which are breaking thin maps will together only contribute a thin set of rational points. The essential ingredient of the following theorem is the Hilbert Irreducibility Theorem proved by Serre.

\begin{theo} \label{theo: mainthinness}
Let $X$ be a geometrically uniruled smooth projective variety over a field $F$ of characteristic $0$. Suppose that $f: Y \to X$ is a generically finite morphism from a smooth projective variety $Y$.  As $\sigma$ varies over all $\sigma \in H^1(F, \mathrm{Aut}(Y/X))$ such that
$Y^{\sigma}$ is irreducible and
\begin{equation*}
(a(X, L), b(F, X, L)) < (a(Y^{\sigma}, (f^{\sigma})^{*}L), b(F, Y^\sigma, (f^\sigma)^*L))
\end{equation*}
the set 
\begin{equation*}
 Z= \bigcup_{\sigma} f^\sigma(Y^\sigma (F)) \subset X(F)
 \end{equation*}
is contained in a thin subset of $X(F)$.
\end{theo}

Recall that our main Theorem \ref{theo: precisetheorem} also addresses thin maps $f: Y \to X$ such that the $a$ and $b$ invariants of $Y$ and $X$ are the same.  Thus we will actually prove stronger versions of Theorem \ref{theo: mainfiniteness} and Theorem \ref{theo: mainthinness} which address this more general situation.

\subsection{Structure of the paper}
Section~\ref{Preliminaries} is devoted to preliminaries.  In Section \ref{sect: mmp}, we briefly review some foundational results of the minimal model program from \cite{BCHM}, \cite{birkar16}, and \cite{birkar16b} and derive some consequences in preparation for the rest of the paper.  

Section~\ref{sec: geoinv} develops the theory of the geometric invariants $a(X, L), b(F, X, L)$ over an arbitrary field of characteristic $0$.  In particular, we construct universal families of adjoint rigid subvarieties in Section \ref{subsec: BAB} and we introduce face contracting morphisms in Section \ref{subsec: facecontracting}. 

In Section~\ref{sect: conjecturaldescription}, we propose a conjectural geometric description of the exceptional set in Manin's Conjecture (Definition~\ref{defi:exceptionalsets}). The main result of this paper (Theorem~\ref{theo: precisetheorem}) is that this proposed set is contained in a thin subset of rational points in accordance with the prediction made by Peyre.  We then compare our construction with the exceptional set in various examples and discuss a few counterexamples to possible extensions.

The rest of the paper is devoted to proofs of our main theorems.  Section~\ref{sec: twists} is devoted to the study of twists and the relationship with thin sets, proving a stronger version of Theorem~\ref{theo: mainthinness}.  In Section~\ref{sect: boundedness} we study the boundedness of breaking thin maps over an algebraically closed field of characteristic $0$ and prove a stronger version of Theorem~\ref{theo: mainfiniteness}.  In Section~\ref{sec: thinset}, we work over an arbitrary field $F$ of characteristic $0$ and prove our main Theorem \ref{theo: precisetheorem}.  The key technical result is Lemma~\ref{lemm: finitelymanycoversovernf}, which constructs universal families for breaking thin maps by combining the results of Section \ref{sect: boundedness} with a careful analysis of arithmetic \'etale fundamental groups.

\bigskip

\noindent
{\bf Acknowledgments.}  
 The second author would like to thank his advisor J\'{a}nos Koll\'{a}r for constant support and encouragement. The authors would like to thank Yuri Tschinkel and Anthony V\'arilly-Alvarado for constant supports of this research. The authors would also like to thank Yuri Tschinkel and Anthony V\'arilly-Alvarado for answering our questions about toric varieties and the lifting property and Yoshinori Gongyo for his explanation of Theorem \ref{theo: BAB2}. The authors would like to thank Tim Browning, Brendan Hassett, and Marta Pieropan for useful comments.  Finally, the authors would like to thank an anonymous reviewer for constructive criticisms and detailed suggestions which significantly improved the exposition of the paper. The authors would also like to thank multiple referees for careful reading of the paper and many detailed suggestions to improve the exposition of the paper.
 
Brian Lehmann was supported by NSF grant 1600875.  Sho Tanimoto was partially supported by MEXT Japan, Leading Initiative for Excellent Young Researchers (LEADER), by Inamori Foundation, by JSPS KAKENHI Early-Career Scientists Grant number 19K14512, by JSPS Bilateral Joint Research Projects Grant number JPJSBP120219935, and by JST FOREST program Grant number JPMJFR212Z.

\section{Preliminaries}
\label{Preliminaries}

Let $F$ be a field of characteristic $0$.  A variety $X$ defined over $F$ is an integral separated scheme of finite type over $F$.  For an extension of fields $F'/F$, we denote the base change of $X$ to $F'$ by $X_{F'}$ and denote the pullback of a $\mathbb{Q}$-Cartier divisor $L$ from $X$ to $X_{F'}$ by $L_{F'}$. For an algebraic closure $\overline{F}/F$, we will also sometimes denote the base change by $\overline{X}$ and the pullback by $\overline{L}$, particularly when $X$ is geometrically integral. 

\begin{defi}
\label{defi:thinmaps}
Let $X$ and $Y$ be varieties.  A morphism $f: Y \to X$ is thin if it is generically finite onto its image and admits no rational section.

A thin subset of $X(F)$ is a finite union $\cup_{j} f_{j}(Y_{j}(F))$ where $f_{j}: Y_{j} \to X$ are thin maps over $F$.
\end{defi}

\bigskip

Consider a commutative diagram of dominant morphisms of varieties
\begin{equation*}
\xymatrix{ &  U  \ar[d]\\
T  \ar[r] & W}
\end{equation*}
If there is a unique component of $T \times_{W} U$ which dominates $T$ and $U$ under the projection maps, then we call it the ``main component'' of the product.

\bigskip

Let $X$ be a projective variety.  A family of subvarieties of $X$ is a diagram
\begin{equation*}
X \times W \supset \xymatrix{  \mathcal{U} \ar[r]^{s} \ar[d]^{p} &  X \\
W & }
\end{equation*}
such that $W$ is a variety, $p$ is projective and flat with irreducible fibers, $\mathcal U$ is a closed subscheme of $X \times W$, and $s$ is the restriction of the projection to $\mathcal{U}$.  A curve $C$ is said to be a movable curve if it is a member of a family of $1$-dimensional subvarieties such that $s$ is dominant.

\bigskip

We will denote the Hilbert scheme of $X$ by $\Hilb(X)$.  A bounded subset of $\Hilb(X)$ is any set that is contained in a finite union of irreducible components.  

\bigskip

We next recall several definitions from birational geometry.  Suppose that $X$ is a normal  projective variety defined over $F$.
We denote the N\'eron-Severi space of $\mathbb{R}$-Cartier divisors up to numerical equivalence by $N^1(X)$ and the space of $\mathbb{R}$-$1$-cycles modulo numerical equivalence by $N_1(X)$. We denote the pseudo-effective cone and the nef cone of divisors by
\[
\overline{\mathrm{Eff}}^1(X), \quad \mathrm{Nef}^1(X)
\]
respectively, and the pseudo-effective cone and the nef cone of curves by
\[
\overline{\mathrm{Eff}}_1(X), \quad \mathrm{Nef}_1(X)
\]
respectively.  These are pointed closed convex cones in $N^1(X)$ and $N_1(X)$.

\begin{defi}[\cite{Nakamaye00}, \cite{ELMNP06}]
Let $X$ be a smooth projective variety defined over a field $F$ of characteristic $0$.  Let $D$ be a $\mathbb Q$-divisor on $X$.  The asymptotic base locus of $D$ is
\begin{equation*}
\mathbf{B}(D) = \bigcap_{m \in \mathbb{Z}_{>0}} \mathrm{Bs}(|mdD|)
\end{equation*}
where $d$ is any positive integer such that $dD$ is Cartier.  The definition is independent of the choice of $d$.
The augmented base locus of $D$ is
\begin{equation*}
\mathbf{B}_{+}(D) =  \bigcap_{\textrm{ample }\mathbb{Q}\textrm{-div }A} \mathbf{B}(D-A).
\end{equation*}
The augmented base locus is always a closed subset of $X$ by \cite[Proposition 1.5]{ELMNP06}.
\end{defi}

Suppose that $L$ is a big and nef $\mathbb{Q}$-divisor on a smooth variety over an algebraically closed field of characteristic $0$.  By \cite[Theorem 0.3]{Nakamaye00} the augmented base locus coincides with the Zariski closure of the set of subvarieties $V \subset X$ such that $L|_{V}$ is not big.  In fact, the restriction of $L$ to any component of the augmented base locus fails to be big. 
We claim that the same is true for any ground field of characteristic $0$.  Indeed, if $X$ is defined over the ground field $F$, then by \cite[Proposition 1.5]{ELMNP06} $\mathbf{B}_{+}(L_{\overline{F}}) = \mathbf{B}(L_{\overline{F}} - A_{\overline{F}})$ for some ample $\mathbb{Q}$-divisor $A$ defined over the ground field.  Since the formation of a base locus is compatible with change of base field, we conclude that both sides are defined over the ground field.  The fact that $\mathbf{B}_{+}(L)$ is the Zariski closure of the subvarieties $V$ such that $L|_{V}$ is not big can be deduced from the corresponding statement over $\overline{F}$.

Note that if $X$ is defined over a number field and $L$ is a big and nef $\mathbb{Q}$-divisor, the Northcott property for rational points is only guaranteed to hold after removing the points contained in $\mathbf{B}_{+}(L)$.  In particular, we will always include $\mathbf{B}_{+}(L)$ in the exceptional set for Manin's Conjecture.

\section{The minimal model program} \label{sect: mmp}

We will use the standard notations of the minimal model program regarding singularities of pairs. We refer to \cite[Definition 2.34]{KM98} for their definitions. We will frequently use the following well-known lemma (see for example \cite[Theorem 2.3]{LTT14}): 

\begin{lemm} \label{lemm: terminalpair}
Let $X$ be a smooth projective variety over a field $F$ of characteristic $0$ and let $L$ be a big and nef $\mathbb{Q}$-divisor on $X$.  Then there is an effective $\mathbb{Q}$-divisor $\Delta$ and an effective ample $\mathbb{Q}$-divisor $A$ such that $(X,\Delta + A)$ is terminal and $L$ is $\mathbb{Q}$-linearly equivalent to $\Delta + A$.
\end{lemm}

\subsection{Canonical models}
Suppose that $X$ is a smooth projective variety over a field $F$ of characteristic $0$ and that $L$ is a big and nef $\mathbb{Q}$-divisor on $X$.  Fix a positive integer $d$ such that $dL$ is Cartier.  Note that the base change of the section ring of $dL$ to $\overline{F}$ is isomorphic to the product of the section rings of the pullbacks of $dL$ to all geometric components of $X_{\overline{F}}$.  
Thus \cite[Theorem 1.2]{BCHM} (combined with Lemma \ref{lemm: terminalpair}) shows that the section ring
\begin{equation*}
\bigoplus_{m \geq 0} H^{0}(X,\mathcal{O}_{X}(md(K_{X} + L)))
\end{equation*}
is finitely generated. When this ring is non-zero, via the Proj construction we obtain a rational map $\pi: X \dashrightarrow W$ such that $\dim(W) = \kappa(X,K_{X} + L)$.  The map $\pi: X \dashrightarrow W$ is known as the canonical model for $(X,L)$, or equivalently, the canonical model for $K_{X} + L$.

\begin{lemm} \label{lemm:birationaltocanonical}
Let $X$ be a geometrically uniruled  smooth projective variety and let $L$ be a big and nef $\mathbb{Q}$-divisor on $X$ such that $K_{X} + L$ is pseudo-effective.  Suppose that $\psi: X \to W$ is a surjective morphism of projective varieties such that:
\begin{enumerate}
\item the base change of $\psi$ to $\overline{F}$ has connected fibers,
\item $\kappa(X_{w},K_{X_{w}} + L|_{X_{w}}) = 0$ for a general fiber $X_{w}$ over a closed point $w \in W$, and
\item there is an ample $\mathbb{Q}$-divisor $H$ on $W$ such that $K_{X} + L - \psi^{*}H$ is $\mathbb{Q}$-linearly equivalent to an effective divisor.
\end{enumerate}
Then $\psi$ is birationally equivalent to the canonical map for $K_{X} + L$.  If the canonical map is a morphism on $X$, then there is a non-empty open subset $W^{\circ} \subset W$ such that on $\psi^{-1}(W^{\circ})$ the canonical map coincides with $\psi$.
\end{lemm}

\begin{proof}
To prove the first statement, we may replace $X$ by a birational model for which the canonical map for $K_{X} + L$ is a morphism and replace $L$ with its pullback to this birational model.  Thus it suffices to prove that the last statement holds.

Let $\pi: X \to W'$ denote the canonical map for $K_{X} + L$. By condition (3), there is a rational map $g: W' \dashrightarrow W$ such that $\psi = g \circ \pi$ as rational maps.  By condition (2), a general fiber of $\psi$ is contracted by $\pi$.  Together with condition (1) these show that $g$ is birational, yielding the desired claim.  
\end{proof}

\subsection{Boundedness of singular Fano varieties}
In \cite{birkar16} and \cite{birkar16b} Birkar establishes the Borisov-Alexeev-Borisov Conjecture concerning the boundedness of mildly singular Fano varieties.  We will use the following special cases of Birkar's results.

\begin{theo}[\cite{birkar16b} Theorem 1.1]
\label{theo: BAB}
Let $\overline{F}$ be an algebraically closed field of characteristic $0$.  Let $d$ be a positive integer and fix $\epsilon > 0$. Then there exists a constant $C_{1} = C_{1}(d,\epsilon) > 0$ such that for any $\epsilon$-lc pair $(X,\Delta)$ such that $X$ has dimension $\leq d$ and $K_{X}+\Delta$ is antiample, we have
\[
(-K_{X}-\Delta)^{\dim X} \leq C_{1}.
\]
\end{theo}

\begin{rema}
This form of \cite[Theorem 1.1]{birkar16b} is not explicitly stated.  However, as in the proof of \cite[Theorem 1.1]{birkar16b}, we can write $-\delta K_X  \sim_{\mathbb{R}} K_{X} + (1 +\delta)(A + \Delta)$ where $A$ is any ample divisor that is $\mathbb{R}$-linearly equivalent to $-(K_{X} + \Delta)$.  
For any non-negative $\epsilon'$ that is less than $\epsilon$, we can ensure that the pair $(X,(1+\delta)(A+\Delta))$ is $\epsilon'$-lc by choosing $\delta$ and $A$ appropriately.  Thus by running the $(-\delta K_X)$-minimal model program we obtain a rational map $\phi: X \dashrightarrow X'$ where $X'$ is an $\epsilon'$-lc weak Fano variety.  In fact, since $X'$ does not depend on $\delta$ and $A$, by taking a limit as $\epsilon' \to \epsilon$ we see that $X'$ is $\epsilon$-lc.  Note that
\begin{equation*}
\vol(-K_{X}-\Delta) \leq \vol(-K_{X'} - \phi_{*}\Delta) \leq \vol(-K_{X'}).
\end{equation*}
As $X'$ varies over all $\epsilon$-lc weak Fano varieties \cite[Theorem 2.11]{birkar16b} shows that $\vol(-K_{X'})$ has a universal upper bound depending only on $d$ and $\epsilon$, yielding the desired statement. 
\end{rema}

\begin{theo}[\cite{birkar16b}]
\label{theo: BAB2}
Let $\overline{F}$ be an algebraically closed field of characteristic $0$.  Let $d$ be a positive integer and fix an $\epsilon > 0$ and a finite set of rational numbers $I \subset [0,1)$.  There is a constant $C_{2} = C_{2}(d,\epsilon,I)$ such that the following holds.  Suppose that
\begin{itemize}
\item $X$ is a projective variety of dimension $\leq d$,
\item $(X,\Delta)$ is an $\epsilon$-lc pair such that the coefficients of $\Delta$ lie in $I$, and
\item $K_{X} + \Delta$ is an antiample $\mathbb{Q}$-Cartier divisor.
\end{itemize}
Then $C_{2}(K_{X} + \Delta)$ is Cartier for all such pairs $(X,\Delta)$.
\end{theo}

\begin{proof}
\cite[Theorem 1.1]{birkar16b} shows that as we vary over all such pairs the set of underlying varieties $X$ is bounded.  We show that in fact the set of pairs $(X,\Delta)$ is log bounded.  Using the boundedness of the underlying $X$, we can find a family of very ample divisors $A$ on these varieties $X$ such that the space of sections of $A$ is bounded in dimension.  Since the coefficient set of $\Delta$ is finite, it suffices to prove that the degree of $\Delta$ against $A$ is bounded.  This follows from
\begin{equation*}
\Delta \cdot A^{\dim X-1} < -K_{X} \cdot A^{\dim X-1}
\end{equation*}
and the boundedness of the varieties $X$.

Since the pairs $(X,\Delta)$ are log bounded, by \cite[Lemma 2.24]{birkar16} and the fact that the coefficient set of $\Delta$ is finite we deduce the desired statement on the Cartier index.
\end{proof}

\cite{Araujo10} shows how the Borisov-Alexeev-Borisov Conjecture can be used to deduce a structure theorem for the cone of nef curves.   We will give a quick explanation of Araujo's arguments since \cite{Araujo10} only explicitly addresses the case when $\dim(X) = 3$.

\begin{lemm} \label{lemm:intbound}
Let $\overline{F}$ be an algebraically closed field of characteristic $0$.  Let $d$ be a positive integer and fix an $\epsilon > 0$ and a finite set of rational numbers $I \subset [0,1)$.  There is a constant $C_{3} = C_{3}(d,\epsilon,I)$ such that the following holds. Suppose that
\begin{itemize}
\item $X$ is a projective variety of dimension $\leq d$;
\item $(X,\Delta)$ is a $\epsilon$-lc pair such that the coefficients of $\Delta$ lie in $I$, and;
\item $K_{X} + \Delta$ is an antiample $\mathbb{Q}$-Cartier divisor.
\end{itemize}
Then any codimension $2$ set $B \subset X$ there is a movable curve $C$ avoiding $B$ and satisfying
\begin{equation*}
-(K_{X} + \Delta) \cdot C \leq C_{3}.
\end{equation*}
\end{lemm}

\begin{proof}
Theorem \ref{theo: BAB} yields an upper bound $C_{1}$ on $(-K_{X} - \Delta)^{\dim X}$ that only depends on $d$ and $\epsilon$. Choose a positive integer $C_{2}$ as in Theorem \ref{theo: BAB2} so that $-C_{2}(K_{X} + \Delta)$ is a Cartier ample divisor.  By \cite[1.1 Theorem and 1.2 Lemma]{Kollar93} there is a positive integer $M = M(d)$ such that $-MC_{2}(K_{X} + \Delta)$ is very ample.  Set $C_{3} = C_{1}(C_{2}M)^{d-1}$.  Then one can find a suitable curve $C$ by taking intersections of general elements in $|-MC_{2}(K_{X} + \Delta)|$.
\end{proof}

Given a cone $\mathcal{C}$ in $N_{1}(X)$ and an element $\ell \in N^{1}(X)$, we will let $\mathcal{C}_{\ell \geq 0}$ denote the intersection of $\mathcal{C}$ with the half-space of curve classes with non-negative intersection against $\ell$.

\begin{lemm}[\cite{Araujo10}] \label{lemm:conetheorem}
Let $X$ be a smooth projective variety and let $\Delta$ be an effective $\mathbb{Q}$-Cartier divisor on $X$ such that $(X,\Delta)$ is an $\epsilon$-lc pair for some $\epsilon > 0$.  Fix an ample $\mathbb{Q}$-Cartier divisor $A$.  Then the cone $\Eff_{1}(X)_{K_{X} + \Delta \geq 0} + \Nef_{1}(X)$ has only finitely many extremal rays with negative intersection against $K_{X} + \Delta + A$.  Furthermore, these rays are generated by the classes of movable curves.
\end{lemm}

\begin{proof}
We first prove the result when the base field is algebraically closed.  Let $I$ be the coefficient set of $\Delta$.  Suppose that $\phi: X \dashrightarrow X'$ is a run of the $(K_{X} + \Delta)$-minimal model program with scaling resulting in a Mori fiber space $\pi: X' \to Z'$. 
Let $Y$ denote a general fiber of $\pi$.  Then $\phi_{*}\Delta|_{Y}$ has coefficients in $I$ and $(Y,\phi_{*}\Delta|_{Y})$ is an $\epsilon$-lc Fano pair.  Altogether we see that there is a universal bound $C_{3}$ as in Lemma \ref{lemm:intbound} for all such pairs $(Y,\Delta|_{Y})$ obtained in this way.

By \cite[Theorem 1.1]{Araujo10} we have
\begin{equation*}
\Eff_{1}(X)_{K_{X} + \Delta \geq 0} + \Nef_{1}(X) = \Eff_{1}(X)_{K_{X} + \Delta \geq 0} + \overline{\sum_{\alpha \in \Sigma} \mathbb{R}_{\geq 0}\alpha}
\end{equation*}
where each $\alpha$ is obtained by running a $(K_{X} + \Delta)$-minimal model program with scaling to obtain $\phi: X \dashrightarrow X'$ where $X'$ carries a Mori fiber space structure $\pi: X' \to Z'$ and setting $\alpha$ to be the numerical pullback (as in \cite[Section 4]{Araujo10}) of the class of a curve in a general fiber $Y$ of $\pi$. 
As discussed earlier, we know that for any such $Y$ there is a movable curve $C'$ in $Y$ avoiding the $\phi^{-1}$-indeterminacy locus and satisfying $-(K_{X'} + \phi_{*}\Delta) \cdot C' \leq C_{3}$.  Since $C'$ avoids the $\phi^{-1}$-indeterminacy locus, its strict transform $C$ in $X$ satisfies $-(K_{X} + \Delta) \cdot C \leq C_{3}$.  Note that the numerical pullback of the corresponding class $\alpha$ is represented by the movable curve $C$ because $\rho(X'/Z') =1$.

There are only finitely many classes of curves $C$ which satisfy both $(K_{X} + \Delta + A) \cdot C \leq 0$ and $-(K_{X} + \Delta) \cdot C \leq C_{3}$. This proves that the set of such $\alpha$ as above satisfying $(K_{X} + \Delta + A) \cdot \alpha < 0$ is finite, finishing the proof when the ground field is algebraically closed.

For a general ground field $F$ of characteristic $0$, let $\{\overline{X}_{j}\}_{j=1}^{s}$ denote the components of $X_{\overline{F}/F}$.  Note that
\begin{equation*}
N_{1}(X) = (\oplus_{j} N_{1}(\overline{X}_{j}))^{\Gal(\overline{F}/F)}
\end{equation*}
and that the action of $\Gal(\overline{F}/F)$ on $\oplus_{j} N_{1}(\overline{X}_{j})$ factors through a finite group $G$.  We can apply the statement of the theorem to each component $\overline{X}_{j}$ and take direct sums to see that $\Eff_{1}(X_{\overline{F}})_{K_{X} + \Delta \geq 0} + \Nef_{1}(X_{\overline{F}})$ admits only finitely many $(K_{X} + \Delta + A)$-negative extremal rays generated by numerical classes $\{ \overline{\alpha}_{i} \}_{i=1}^{r}$.

Define $\alpha_{i} = \frac{1}{|G|} \sum_{g \in G} g\overline{\alpha}_{i}$.  Since the pseudo-effective and nef cones of curves are preserved by the $G$-action, it is clear that $\Eff_{1}(X)_{K_{X} + \Delta \geq 0} + \Nef_{1}(X)$ has only finitely many extremal rays with negative intersection against $K_{X} + \Delta + A$ and these rays are generated by a subset of $\{\alpha_{i}\}_{i=1}^{r}$.
Furthermore, we can construct an irreducible movable curve over the ground field representing some multiple of $\alpha_{i}$ by taking the Galois orbit of an irreducible curve representing $\overline{\alpha}_{i}$ over $\overline{F}$ and descending to $F$.
\end{proof}

\begin{rema}  
If $(X,\Delta)$ is a terminal pair, then it is also $\epsilon$-lc for some $\epsilon > 0$.  Thus we can apply Lemma \ref{lemm:conetheorem} to terminal pairs $(X,\Delta)$. 
\end{rema}

\section{Geometric invariants in Manin's Conjecture}
\label{sec: geoinv}

In this section we study the geometric behavior of the $a$ and $b$ invariants over an arbitrary field $F$ of characteristic $0$.

\subsection{$a$-invariant}
\label{subsec: a-inv}

\begin{defi}
Let $X$ be a smooth projective variety defined over a field $F$ of characteristic $0$. Let $L$ be a big and nef $\mathbb Q$-divisor on $X$. Then we define the {\it Fujita invariant} (or $a$-invariant) by
\[
a(X, L) = \min \{ t \in \mathbb R \mid K_{X} + tL \in \overline{\mathrm{Eff}}^1(X)\}.
\]
When $L$ is nef but not big, we formally set $a(X, L) = +\infty$.
When $X$ is singular, we define the Fujita invariant as the Fujita invariant of the pullback of $L$ to any smooth birational model of $X$. This is well-defined by \cite[Proposition 2.7]{HTT15}.
\end{defi}

We will frequently use the following fundamental properties of the $a$-invariant:
\begin{itemize}
\item If $\phi: X' \to X$ is a birational map then $a(X',\phi^{*}L) = a(X,L)$.  (\cite[Proposition 2.7]{HTT15})
\item $a(X, L) > 0$ if and only if $X$ is geometrically uniruled.  (\cite[0.3 Corollary]{BDPP})
\item When $a(X, L) > 0$, $a(X, L)$ is always a rational number.  (See \cite[Corollary 1.1.7]{BCHM} when $L$ is ample and \cite[Theorem 2.16]{HTT15} when $L$ is big and nef.)  
\end{itemize}
It is convenient to concentrate on a certain class of pairs:

\begin{defi} \label{defi: adjointrigid}
Let $X$ be a projective variety and let $L$ be a big and nef divisor on $X$.  Let $\phi: X' \to X$ be a resolution of singularities.  We say that  $(X, L)$ is adjoint rigid if $a(X,L) > 0$ and $\kappa(X',K_{X'} + a(X',\phi^{*}L)\phi^{*}L) = 0$.  Note that this definition does not depend on the choice of resolution. 
\end{defi}

One should think of adjoint rigid pairs as birational analogues of mildly singular Fano varieties. (See [LTT18, Proposition 2.5] for a precise statement in this direction.)  In particular, one can analyze their structure using the recent results of Birkar on the boundedness of singular Fano varieties.

On the other hand, the geometry of arbitrary polarized pairs is controlled by adjoint rigid pairs.  Indeed, if $a(X,L)>0$, the following lemma 
shows that the closure of a general fiber of the canonical map for $(X,a(X,L)L)$ has the same $a$-value as $X$ and is adjoint rigid.
We will frequently leverage this fact by replacing an arbitrary pair with the fibers of its canonical model map.

\begin{lemm} \label{lemm:ainvandcanonicalfibers}
Let $X$ be a smooth geometrically uniruled projective variety defined over $F$ and $L$ a big and nef $\mathbb Q$-divisor on $X$. Let $\pi: X \dashrightarrow W$ denote the canonical map for $(X,a(X,L)L)$.  Let $Y$ be the closure in $X$ of a general fiber of $\pi$.  Then $a(Y,L|_{Y}) = a(X,L)$ and $(Y, L|_Y)$ is adjoint rigid.
\end{lemm}

\begin{proof}
By the birational invariance of the $a$-invariant, it suffices to prove this when $\pi$ is a morphism and $Y$ is a smooth fiber of $\pi$.  The equality $a(Y,L|_{Y}) = a(X,L)$ follows from the fact that the restriction of $K_{X} + a(X,L)L$ to a general fiber $Y$ has Iitaka dimension $0$, hence lies on the boundary of the pseudo-effective cone.
\end{proof}

We next consider how the $a$-invariant behaves under change of ground field. 

\begin{prop} \label{prop: galinvofa}
Let $X$ be a smooth projective variety over a field $F$ of characteristic $0$ and let $L$ be a big and nef $\mathbb Q$-divisor on $X$.
We fix an extension $F'/F$ and let $Y \subset X_{F'}$ be an integral subvariety defined over $F'$. Let $\sigma \in \mathrm{Aut}(F'/F)$.
Then we have
\[
a(Y, L_{F'}|_{Y})= a(\sigma(Y), L_{F'}|_{\sigma(Y)})
\]
and the adjoint divisors with respect to $L$ on smooth models of $Y$ and $\sigma (Y)$ have the same Iitaka dimension.
\end{prop}

\begin{proof}
After applying an embedded resolution of singularities, we have a birational model $X'$ over $F'$ such that the strict transform $Y'$ of $Y$ is smooth. Let $\sigma_*X'$ be the pullback of $X' \to \Spec (F')$ by $\sigma^{-1} : \Spec (F') \to \Spec (F')$. We have a natural isomorphism as schemes $\sigma : X' \to \sigma_*X'$ and we denote the image of $Y'$ by $\sigma(Y')$. Then $\sigma(Y')$ is a smooth model of $\sigma(Y)$. Let $K_{Y'}$ be the canonical divisor on $Y'$. Then we have $\sigma(K_{Y'}) = K_{\sigma(Y')}$. Moreover for any integer $m$ such that $ma(Y,L|_{Y})L$ is a Cartier divisor we have
\[
m(K_{Y'} + a(Y,L_{F'}|_{Y})L_{F'}|_{Y'}) \sim D \geq 0 \iff m(K_{\sigma(Y')} + a(Y, L_{F'}|_{Y})L_{F'}|_{\sigma(Y')}) \sim \sigma(D) \geq 0.
\]
This shows that $K_{Y} + a(Y,L_{F'}|_{Y})L_{F'}|_{Y}$ and $K_{\sigma(Y')} + a(Y, L_{F'}|_{Y})L_{F'}|_{\sigma(Y')}$ have the same Iitaka dimension, proving the claims.
\end{proof}

\begin{coro}
\label{coro: flatbasechange}
Let $X$ be a smooth projective variety defined over a field of characteristic $0$ and let $L$ be a big and nef $\mathbb Q$-divisor on $X$.  Let $F'/F$ denote a field extension and let $X'$ denote any irreducible component of $X_{F'}$.  Then $a(X,L) = a(X',L_{F'}|_{X'})$ and $\kappa(X,K_{X} + a(X,L)L) = \kappa(X', K_{X'} + a(X', L_{F'}|_{X'})L_{F'}|_{X'})$.  
\end{coro}

\begin{proof}
It suffices to prove the statement when $F'/F$ is either an algebraic extension or a purely transcendental extension since any extension of fields is a composition of these two types.

Suppose first that $F'/F$ is algebraic.  Let $\widetilde{F}/F'$ be an extension such that $\widetilde{F}/F$ is Galois.  Then we can prove the desired statement for $F'/F$ by verifying it for both $\widetilde{F}/F$ and $\widetilde{F}/F'$.  Thus it suffices to prove the statement under the additional condition that $F'/F$ is Galois.

When the extension is Galois then $\Gal(F'/F)$ acts transitively on the components of $X_{F'}$.  By Proposition \ref{prop: galinvofa} we see that every component of $X_{F'}$ has the same $a$-invariant and thus this $a$-invariant is equal to $a(X,L)$. 
Furthermore, since base change commutes with taking spaces of sections, we see that for every integer $m$ such that $m(K_{X} + a(X,L)L)$ is a Cartier divisor we have
\begin{equation*}
\dim H^{0}(X,m(K_{X} + a(X,L)L)) = \dim H^{0}(X_{F'},m(K_{X_{F'}} + a(X,L)L_{F'})).
\end{equation*}
Since $X_{F'}$ is a disjoint union of Galois translates of $X'$, we have
\begin{equation*}
\dim H^{0}(X_{F'},m(K_{X_{F'}} + a(X,L)L_{F'})) = r \cdot \dim H^{0}(X',m(K_{X'} + a(X,L)L_{F'}|_{X'})).
\end{equation*}
where $r$ denotes the number of components of $X_{F'}$.  Since $r$ does not change as $m$ increases, the Iitaka dimensions coincide.

Now suppose that $F'/F$ is purely transcendental.  This implies that $X_{F'}$ is smooth and irreducible.  Since sections are preserved by base change we deduce that for any positive rational number $a$ we have $\kappa(X,K_{X} + aL) = \kappa(X_{F'},K_{X_{F'}} + aL_{F'})$.  This implies that the $a$-invariants  over $F$ and $F'$ also coincide.
\end{proof}

The $a$-invariant is constant for general fibers of a smooth morphism:

\begin{theo} \label{theo: aconstancy}
Let $\pi : \mathcal X \rightarrow W$ be a smooth projective morphism defined over a field $F$ of characteristic $0$ such that every fiber of $\pi$ is integral and geometrically uniruled and let $L$ be a relatively big and nef $\mathbb Q$-divisor on $\mathcal X$. Then there exists a Zariski open subset $W^{\circ} \subset W$ such that the $a$-invariant $a(X_w,L|_{X_w})$ and the Iitaka dimension of $a(X_w, L|_{X_{w}})L|_{X_w} + K_{X_w}$ remain constant as we vary over all closed points $w \in W^{\circ}$.
\end{theo}

\begin{proof}
Over an algebraically closed field this is \cite[Theorem 4.3]{LTDuke}. Over a general $F$, we can take the base change to $\overline{F}$, replace $\mathcal{X}$ with any component of $\mathcal{X}_{\overline{F}}$, and replace the morphism to $W_{\overline{F}}$ with its Stein factorization.  Since the space of sections associated to a divisor is stable under flat base change, the result over an arbitrary field of characteristic $0$ follows from the result over its algebraic closure combined with Corollary~\ref{coro: flatbasechange}. 
\end{proof}

The following lemmas describe some important geometric properties of the $a$-invariant.

\begin{lemm} \label{lemm: genfinite}
Let $X$ be a projective variety and let $L$ be a big and nef $\mathbb Q$-Cartier divisor on $X$. Suppose that $f : Y \to X$ is a generically finite dominant morphism from a projective variety.
Then we have $a(Y, f^*L) \leq a(X, L)$.
\end{lemm}

\begin{proof}
After birational modifications to $X$ and $Y$ we may assume that $X$ and $Y$ are smooth.
By the ramification formula, we have
\[
a(X, L)f^*L + K_Y = a(X, L)f^*L + f^*K_X + R
\]
where $R \geq 0$ is the ramification divisor. Thus $K_{Y} + a(X, L)f^*L$ is pseudo-effective and our assertion follows.
\end{proof}

\begin{lemm} \label{lemm:ainvdominantfamily}
Let $X$ be a smooth projective variety defined over $F$ and $L$ a big and nef $\mathbb Q$-divisor on $X$. Let $p : \mathcal U \to W$ be a family of subvarieties on $X$ with the evaluation map $s : \mathcal U \rightarrow X$. Suppose that $s$ is dominant. Then for a general member $Y$ of $p$ we have $a(Y, L|_{Y}) \leq a(X, L)$.
\end{lemm}

\begin{proof}
This is stated in \cite[Proposition 4.1]{LTT14} assuming the base field is algebraically closed.
Its generalization to arbitrary fields of characteristic $0$ follows from Corollary~\ref{coro: flatbasechange}. 
\end{proof}

\begin{lemm} \label{lemm:dominantequalitycase}
Let $f: Y \to X$ be a dominant generically finite morphism of smooth projective varieties and let $L$ be a big and nef $\mathbb{Q}$-Cartier divisor on $X$.  Suppose that $a(Y,f^{*}L) = a(X,L)$.  Let $\Gamma$ denote the closure of a general fiber of the canonical map for $(Y,a(Y,f^{*}L)f^{*}L)$ and let $S$ denote its image in $X$.  Then $a(S,L|_{S}) = a(X,L)$ and $S$ is adjoint rigid with respect to $L|_{S}$.
\end{lemm}

\begin{proof}
By Lemma \ref{lemm:ainvandcanonicalfibers} we have $a(\Gamma,f^{*}L|_{\Gamma}) = a(Y,f^*L)$ and $\Gamma$ is adjoint rigid with respect to $f^{*}L|_{\Gamma}$.  Since $f|_{\Gamma}$ is generically finite for a general $\Gamma$, Lemma \ref{lemm: genfinite} shows that $a(S,L|_{S}) \geq a(\Gamma,f^{*}L|_{\Gamma})$.   However as we vary $\Gamma$ the images $S$ form a dominant family of subvarieties of $X$, so by Lemma \ref{lemm:ainvdominantfamily} we must have $a(S,L|_{S}) \leq a(X,L)$.  By combining these inequalities with the fact that $a(Y,f^{*}L) = a(X,L)$ we deduce that we have equalities of $a$-values everywhere.

Let $g: \Gamma' \to S'$ be a birational model of $f|_{\Gamma}: \Gamma \to S$ such that $\Gamma'$ and $S'$ are smooth and let $L'$ denote the pullback of $L$ to $S'$.  By the ramification formula we have
\[
a(S',L')g^*L' + K_{\Gamma'} = a(S', L')g^{*}L' + g^*K_{S'} + R
\]
for some effective divisor $R$.  Since the $a$-values of $\Gamma'$ and $S'$ are the same, we deduce that
\begin{equation*}
\kappa(S',K_{S'}+a(S',L')L') \leq \kappa(\Gamma',K_{\Gamma'}+a(\Gamma',g^{*}L')g^{*}L') = 0.
\end{equation*}
Thus $S$ is adjoint rigid with respect to $L|_{S}$.
\end{proof}

\subsection{$b$-invariant}
\label{subsec: b-inv}

If $\mathcal{C}$ is a closed convex pointed cone in a finite-dimensional real vector space, then a face of $\mathcal{C}$ is any closed convex subcone $\mathcal{F} \subset \mathcal{C}$ such that if $\alpha_{1},\alpha_{2} \in \mathcal{C}$ and $\alpha_{1} + \alpha_{2} \in \mathcal{F}$ then $\alpha_{1},\alpha_{2} \in \mathcal{F}$.

\begin{defi} \label{defi:facedefinition}
Let $X$ be a smooth projective variety defined over a field $F$ of characteristic $0$ and let $L$ be a big and nef $\mathbb{Q}$-divisor on $X$.  
We let $\mathcal{F}_{X,L}$ denote the face of $\Nef_{1}(X)$ consisting of classes with vanishing intersection against $K_{X} + a(X,L)L$.  We also let $\mathcal{F}^{X,L} = \mathcal{F}_{X,L}^{\vee}$ denote the dual face of $\Eff^{1}(X)$. 
\end{defi}

\begin{defi}
Let $X$ be a smooth projective variety defined over a field $F$ of characteristic $0$ and let $L$ be a big and nef $\mathbb Q$-divisor on $X$.  We define the $b$-invariant by
\begin{align*}
b(F, X, L) = \dim \, \mathcal{F}_{X,L} = \mathrm{codim} \, \mathcal{F}^{X,L}.
\end{align*}
When $X$ is singular, we define the $b$-invariant by pulling $L$ back to any smooth birational model of $X$; the result is independent of the choice of model by \cite[Proposition 2.10]{HTT15}.  When $L$ is nef but not big, we formally set $b(F, X,L) = \infty$.
\end{defi}

The faces $\mathcal{F}_{X,L}$ satisfy a natural compatibility under birational transforms.

\begin{lemm} \label{lemm:birfaceinv}  
Let $f : X' \to X$ be a generically finite morphism of smooth projective varieties such that every geometric component of $X'$ maps birationally to a geometric component of $X$ via $\overline{f}$. Let $L$ be a big and nef $\mathbb Q$-divisor on $X$.  Then the pushforward map $f_{*}: N_{1}(X') \to N_{1}(X)$ induces an isomorphism $f_{*}: \mathcal{F}_{X',f^{*}L} \cong \mathcal{F}_{X,L}$.
\end{lemm}

\begin{proof}
For any geometric component of $X'$ the N\'eron-Severi space is spanned by the $\overline{f}$-exceptional divisors and the restriction of $\overline{f}^{*}N^{1}(\overline{X})$ to the component.  Since $X'$ is integral we know that $\Gal(\overline{F}/F)$ acts transitively on the geometric components of $X'$.  We conclude that $N^{1}(X')$ is spanned by $f$-exceptional divisors and $f^{*}N^{1}(X)$.

We next argue that $f_{*}$ actually maps $\mathcal{F}_{X',f^{*}L}$ to $\mathcal{F}_{X,L}$.  Note that there is an effective $\phi$-exceptional divisor $E$ on $X'$ such that $K_{X'} + a(X',f^{*}L)f^{*}L = f^{*}(K_{X} + a(X,L)L) + E$.  Thus any nef curve class with vanishing intersection against $K_{X'} + a(X',f^{*}L)f^{*}L$ will also have vanishing intersection against $f^{*}(K_{X} + a(X,L)L)$ and vanishing intersection against $E$.  The former property shows that $f_{*}$ maps $\mathcal{F}_{X',f^{*}L}$ to $\mathcal{F}_{X,L}$.

Next note that $E$ contains every $f$-exceptional divisor with a positive coefficient. Hence any class in $\mathcal{F}_{X',f^{*}L}$ has zero intersection against every $f$-exceptional divisor. Therefore every curve class in $\mathcal{F}_{X',f^{*}L}$ is pulled back from $\mathcal{F}_{X,L}$. 
Thus the pullback for curve classes induces a map $\frac{1}{\deg(f)} f^* : \mathcal{F}_{X,L} \to \mathcal{F}_{X',f^{*}L}$ inverse to $f_{*}$ and our assertion follows.  
\end{proof}

In contrast to the $a$-value, the $b$-value can change upon field extension, but it can only increase.

\begin{prop}
Let $X$ be a geometrically integral smooth projective variety defined over a field $F$ of characteristic $0$ and let $L$ be a big and nef $\mathbb Q$-divisor on $X$. Let $F'/F$ be a finite extension. Then we have
\[
b(F, X, L) \leq b(F', X_{F'}, L_{F'}).
\]
\end{prop}

\begin{proof}
This follows from the fact that $b(F, X, L)$ is given by
\[
\dim \mathrm{Nef}_1(\overline{X})^{\mathrm{Gal}(\overline{F}/F)} \cap (K_{X} + a(X,L)L)^{\perp}.
\]
\end{proof}

Recall that a face $\mathcal{F}$ of a cone $\mathcal{C}$ is said to be supported if there is a linear functional which is non-negative on $\mathcal{C}$ and vanishes precisely along $\mathcal{F}$.  It follows from the theory of dual cones that $\mathcal{F}^{X,L}$ is the minimal supported face of $\Eff^{1}(X)$ containing the class of $K_{X} + a(X,L)L$.  The following result shows that $\mathcal{F}^{X,L}$ satisfies a stronger property; a related statement is given in \cite[Theorem 2.16]{HTT15}.

\begin{lemm} \label{lemm:alternativedescription}
Let $X$ be a geometrically uniruled geometrically integral smooth projective variety and let $L$ be a big and nef $\mathbb{Q}$-divisor on $X$. Then:
\begin{enumerate}
\item $\mathcal{F}^{X,L}$ is the minimal face of $\Eff^{1}(X)$ containing $K_{X} + a(X,L)L$.
\item Let $\mathcal{E}$ denote the set of irreducible divisors $E$ such that for some $c>0$ the divisor $K_{X} + a(X,L)L - cE$ is $\mathbb{Q}$-linearly equivalent to an effective divisor.   Then the cone generated by the classes of the divisors in $\mathcal{E}$ contains a relatively open neighborhood of $K_{X} + a(X,L)L$ in $\mathcal{F}^{X,L}$. In particular, the classes of the divisors in $\mathcal{E}$ span $\Span(\mathcal{F}^{X,L})$.
\end{enumerate}
\end{lemm}

\begin{proof}
By Lemma \ref{lemm: terminalpair} $a(X,L)L$ is $\mathbb{Q}$-linearly equivalent to a sum $\Delta + A$ where $(X,\Delta+A)$ is terminal and $A$ is ample. 
Let $C$ be a compact polyhedral convex hull of a finite set of ample $\mathbb{Q}$-divisors such that the interior of $C$ contains $A$.
Let $U$ denote the set of classes in $N^{1}(X)$ of the form
\begin{equation*}
U = \left\{ \left. K_{X} + \Delta + \frac{1}{2}A + \lambda D \right| D \in C, \lambda \geq 0 \right\}.
\end{equation*}
In particular $U$ is a closed rational polyhedral set containing $K_{X} + a(X,L)L$ in its interior. 

(1) We apply Lemma \ref{lemm:conetheorem} to the pair $(X,\Delta)$ and the ample $\mathbb{Q}$-divisor $\frac{1}{4}A$ to obtain a finite set of $(K_{X} + \Delta+ \frac{1}{4}A)$-negative extremal rays.  Note that these are the only extremal rays of $\Nef_{1}(X)$ which can have vanishing intersection against an element of $U$ that lies on the boundary of the pseudo-effective cone.  Thus $\Eff^1(X) \cap U$ is a rational polyhedral set, since it is cut out in $U$ by finitely many rational linear inequalities.  In particular, every face of $\Eff^{1}(X)$ intersecting the interior of $U$ must be a supported face, proving (1).

(2) By \cite[Theorem D]{BCHM} (and the fact that $\mathbb{R}$-linear equivalence can be detected after change of base field), every pseudo-effective class contained in $U$ is represented by an effective divisor.  By (1) we know that $K_{X} + a(X,L)L$ is in the relative interior of $\mathcal{F}^{X,L}$. Thus, if $E$ is an effective divisor representing a class in $\mathcal{F}^{X,L} \cap U$, there is some positive constant $c$ such that $K_{X} + a(X,L)L - cE$ also has class in $\mathcal{F}^{X,L} \cap U$ and is thus represented by an effective divisor.  The statement follows.
\end{proof}

\cite[Theorem 1.2]{Sengupta17} shows that over an algebraically closed field of characteristic $0$ the $b$-invariant is constant for an open set of fibers in a smooth family of uniruled projective varieties.  However, the behavior over an arbitrary field of characteristic $0$ is more subtle since the $b$-invariant depends on the splitting behavior of the fibers.

The following useful criterion gives a geometric characterization of the $b$-invariant over an arbitrary field of characteristic $0$.  An analogous statement over $\mathbb{C}$ is proved in \cite[Lemma 2.11]{Sengupta17}.

\begin{lemm} \label{lemm: monodromyandbvalue}
Let $X$ be a geometrically uniruled geometrically integral smooth projective variety defined over a field $F$ of characteristic $0$, and let $L$ be a big and nef $\mathbb Q$-divisor on $X$.  Let $\pi: X \dashrightarrow W$ be the canonical model for $K_{X} + a(X,L)L$.  Suppose that there is a non-empty open set $W^{\circ} \subset W$ such that (i) $W^\circ$ is smooth, (ii) $\pi$ is well-defined and smooth over $W^{\circ}$, and (iii) every fiber $X_{w}$ over a closed point $w$ satisfies $a(X_{w},L|_{X_{w}}) = a(X,L)$.  Let $w \in W^{\circ}$ be a closed point and fix a geometric point $\overline{w}$ above $w$.
\begin{enumerate}
\item We have
\begin{align*}
b(F,X,L) & = \dim \, \left((N^{1}(X_w)\cap N^1(\overline{X}_{\overline{w}})^{\pi^{\et}_1(\overline{W^{\circ}},\overline{w})}) / \Span(\{E_{i}\}_{i = 1}^r) \right)
\end{align*}
where $\{ E_{i} \}_{i=1}^r$ is the finite set of irreducible divisors which dominate $W$ under $\pi$ and which satisfy $K_{X} + a(X,L)L - c_iE_{i} \in \Eff^{1}(X)$ for some $c_i > 0$.   (Here we are identifying $N^1(X_w)$ with its image under the natural embedding into $N^1(\overline{X}_{\overline{w}})$.) 
\item Let $i: X_w \hookrightarrow X$ denote the inclusion.  Then $i_{*}(\mathcal{F}_{X_w,L|_{X_{w}}}) \subset \mathcal{F}_{X,L}$.
\item If $X_{w}$ is a general fiber of $\pi$ then $i_{*}: \mathcal{F}_{X_w,L|_{X_{w}}} \to \mathcal{F}_{X,L}$ is a surjection.
\end{enumerate}
\end{lemm}

\begin{proof}
We may resolve $\pi$ to be a morphism without affecting $\pi^{-1}(W^{\circ})$, and since by Lemma \ref{lemm:birfaceinv} blowing-up induces an isomorphism of the faces $\mathcal{F}_{X,L}$ the statement for the blow-up is equivalent to the statement for the original variety.  

(1) We first show that there are only finitely many divisors $\{ E_{i} \}_{i=1}^r$ which dominate $W$ and satisfy $K_{X} + a(X,L)L - c_iE_{i} \in \Eff^{1}(X)$ for some $c_i > 0$.  The geometric generic fiber $\overline{X}_{\overline{\eta}}$ is adjoint rigid with respect to $L$.  If we restrict $E_{i}$ to $\overline{X}_{\overline{\eta}}$, then the support must lie in the unique effective divisor numerically equivalent to $K_{\overline{X}_{\overline{\eta}}} + a(X,L)L|_{\overline{X}_{\overline{\eta}}}$.  Thus there are only finitely many divisors $E_{i}$ of this type.

By \cite[Th\'eor\`eme 5.2]{andre96} and \cite[Theorem 1.1]{MP12} we know that for some geometric point $\overline{s} \in \overline{W}^{\circ}$  the specialization map $N^{1}(\overline{X}_{\overline{\eta}}) \to N^{1}(\overline{X}_{\overline{s}})$ is an isomorphism where $\overline{\eta}$ denotes the geometric generic point of $W$.  Furthermore by \cite[Proposition 3.3]{MP12} this isomorphism is compatible with restriction from $N^{1}(\overline{X})$ and is a map of $\pi^{\et}_1(\overline{W^{\circ}},\overline{s})$-modules.
Note that the image of the map
\begin{equation*}
N^1(\overline{X}) \twoheadrightarrow N^{1}(\overline{X}_{\eta})  \to N^{1}(\overline{X}_{\overline{\eta}})
\end{equation*}
is exactly the $\Gal(\overline{\overline{F}(\overline{W})}/\overline{F}(\overline{W}))$-invariant part, or equivalently, the $\pi^{\et}_1(\overline{W^{\circ}},\overline{s})$-invariant part.  Altogether we conclude that under the restriction map $N^{1}(\overline{X})$ will surject onto $N^{1}(\overline{X}_{\overline{s}})^{\pi^{\et}_1(\overline{W^{\circ}},\overline{s})}$ for a single fiber $\overline{X}_{\overline{s}}$.

We claim that $N^{1}(\overline{X}) \to N^{1}(\overline{X}_{\overline{t}})^{\pi^{\et}_1(\overline{W^{\circ}},\overline{t})}$ is surjective for every $\overline{t} \in \overline{W}^{\circ}$.  Indeed, consider the sequence of maps
\begin{equation*}
N^1(\overline{X}) \to N^{1}(\overline{X}_{\eta})  \to N^{1}(\overline{X}_{\overline{\eta}}) \to N^{1}(\overline{X}_{\overline{t}})
\end{equation*}
where \cite[Proposition 3.6]{MP12} defines the last map and shows it is injective.  The map in the middle is also injective.  Since $\overline{X}_{\eta}$ is rationally connected, the kernel of the leftmost map is generated by the kernel of $\Pic(\overline{X}) \to \Pic(\overline{X}_{\eta})$.  Altogether we see that for every $\overline{t}$ the kernel of $N^{1}(\overline{X}) \to N^{1}(\overline{X}_{\overline{t}})$ is the subspace of $N^{1}(\overline{X})$ spanned by all $\pi$-vertical divisors.  To prove our claim it suffices to show that the rank of $N^{1}(\overline{X}_{\overline{t}})^{\pi^{\et}_1(\overline{W^{\circ}},\overline{t})}$ is constant. Since the generic fiber of $\overline{\pi}$ is rationally connected, the Picard rank of the geometric fibers of $\pi$ is constant.  By taking monodromy invariants we obtain the claim.

In particular, we have a surjection $N^{1}(\overline{X}) \to N^{1}(\overline{X}_{\overline{w}})^{\pi^{\et}_1(\overline{W^{\circ}},\overline{w})}$.  Note that the action of $\Gal(\overline{F}/F)$ on $N^{1}(\overline{X})$ and on $N^{1}(\overline{X}_{w})$ factors through a finite group.  Thus we see that the restriction map
\begin{equation}\label{equation:monodromy}
N^{1}(\overline{X})^{\Gal(\overline{F}/F)} \to N^{1}(X_w)\cap N^1(\overline{X}_{\overline{w}})^{\pi^{\et}_1(\overline{W^{\circ}},\overline{w})}
\end{equation}
is surjective.

By Lemma \ref{lemm:alternativedescription}, $b(F,X,L)$ is the dimension of the quotient of $N^{1}(X)$ by all effective irreducible divisors $E$ satisfying $K_{X} + a(X,L)L - c E \in \Eff^{1}(X)$ for some $c>0$.  Note that $E$ lies in the kernel of the restriction map to $N^{1}(X_{w})$ if and only if $\pi(E) \subsetneq W$ and $E$ satisfies $K_{X} + a(X,L)L - cE \in \Eff^{1}(X)$ for some $c > 0$.  Thus if we further quotient by the finitely many $E_i$'s which dominate $W$ and satisfy $K_{X} + a(X,L)L - c E \in \Eff^{1}(X)$ for some $c>0$, the dimension of the resulting space is $b(F,X,L)$.

(2) First we claim that for every fiber $X_{w}$ over a closed point of $W^{\circ}$ an effective divisor on $X$ will restrict to a divisor on $X_{w}$ linearly equivalent to an effective divisor. Indeed, since the invertible sheaf associated to any effective divisor $D$ on $X$ is flat over $W^\circ$, the function $w \mapsto h^0(X_{w}, \mathcal O_X(D)|_{X_w})$ satisfies upper semicontinuity. Since this space of sections has positive dimension for a general fiber, it must be positive for all points on $W^\circ $. Then dually, nef curve classes on $X_w$ push forward to nef curve classes on $X$.  Using adjunction and the assumption that $a(X,L) = a(X_w,L|_{X_w})$ we see that $f_*$ maps $\mathcal{F}_{X_w,L|_{X_{w}}}$ to $\mathcal{F}_{X,L}$.

(3) Finally, we show that when $X_{w}$ is general then $i_{*}$ induces a surjection of faces.  By Lemma \ref{lemm: terminalpair} $L$ is $\mathbb{Q}$-linearly equivalent to a sum $\Delta + A$ where $A$ is ample and $(X,\Delta + A)$ is terminal.  By Lemma \ref{lemm:conetheorem} the face $\mathcal{F}_{X,L}$ only contains finitely many extremal rays of $\Nef_{1}(X)$ and each such ray is generated by the class of a movable curve $C_{j}$ on $X$.   Recall that we have modified $X$ so that the canonical map $\pi$ is a morphism on $X$.  Thus there is an ample divisor $H$ on $W$ such that $K_{X} + a(X,L) L - \pi^{*}H$ is $\mathbb{Q}$-linearly equivalent to an effective divisor.   
This implies that each $C_{j}$ must be $\pi$-vertical. In particular, for each $C_{j}$, a general fiber of $\pi$ will contain an irreducible deformation of $C_{j}$ which is a movable curve on that fiber.  Thus the class of each such curve is nef on $X_{w}$, and in particular, is contained in $\mathcal{F}_{X_{w},L|_{X_{w}}}$. This proves that $i_* : \mathcal{F}_{X_w,L|_{X_{w}}} \to \mathcal{F}_{X,L}$ is surjective. 
\end{proof}

\begin{coro} \label{coro: bvalequality}
Let $X$ be a geometrically uniruled geometrically integral smooth projective variety defined over a field $F$ of characteristic $0$ and let $L$ be a big and nef $\mathbb Q$-divisor on $X$.  Suppose that $\pi: X \to W$ is a surjective morphism to a projective variety $W$ which is birationally equivalent to the canonical model map for $K_{X} + a(X,L)L$.

Suppose that there is a non-empty smooth open set $W^{\circ} \subset W$ such that $\pi$ is smooth over $W^{\circ}$ and the geometric monodromy action of $\pi^{\et}_1(\overline{W^{\circ}},\overline{w})$ on $N^1(\overline{X}_{\overline{w}})$ is trivial for closed points $w \in W^{\circ}$.  
Then for a general point $w \in W^{\circ}$ the inclusion $i: X_{w} \to X$ induces an isomorphism $i_{*}: \mathcal{F}_{X_{w},L|_{X_{w}}} \to \mathcal{F}_{X,L}$.  In particular we have $b(F,X,L) = b(F,X_{w},L|_{X_{w}})$.
\end{coro}

\begin{proof}
Let $\phi: X' \to X$ be a smooth birational model such that the canonical map for $K_{X'} + a(X',\phi^{*}L)\phi^{*}L$ is a morphism on $X'$.  Let $\pi': X' \to W$ be the composition $\pi \circ \phi$.  Note that $\pi'$ agrees with the canonical map over an open subset of $W$.  We first shrink $W^{\circ}$ to ensure that $\pi'$ coincides with the canonical map over $W^{\circ}$.  After shrinking $W^{\circ}$ further we may ensure that $W^{\circ}$ is smooth
and that $\pi'$ is smooth over $W^{\circ}$, and by Lemma \ref{lemm:ainvandcanonicalfibers} we may also ensure that every fiber $X'_{w}$ over a closed point $w \in W^{\circ}$ satisfies $a(X'_{w},\phi^{*}L|_{X'_{w}}) = a(X',\phi^{*}L)$.  We may now apply Lemma \ref{lemm: monodromyandbvalue} (3) to deduce that for a general $w \in W^{\circ}$ the map $i_{*}: \mathcal{F}_{X'_{w},\phi^{*}L|_{X'_{w}}} \to \mathcal{F}_{X',\phi^{*}L}$ is a surjection.  By Lemma~\ref{lemm:birfaceinv} the map $i_{*}: \mathcal{F}_{X_{w},L|_{X_{w}}} \to \mathcal{F}_{X,L}$ is also surjective.  On the other hand, the monodromy assumption shows that $N^1(X) \rightarrow N^1(X_w)$ is surjective by \eqref{equation:monodromy}. Dually $i_{*}: N_1(X_w) \to N_1(X)$ is an injection, proving the claim.
\end{proof}

\subsection{Boundedness results}
\label{subsec: BAB}
Using the boundedness of singular Fano varieties proved by Birkar in \cite{birkar16} and \cite{birkar16b}, the papers \cite{LTT14}, \cite{HJ16}, \cite{LTDuke}, \cite{Sen17} prove certain types of boundedness results for the $a,b$-invariants over an algebraically closed field.  In this section we revisit these results in the setting of an arbitrary field $F$ of characteristic $0$.  
The following theorem proves the boundedness of subvarieties with $a$-value at least as large as $X$.

\begin{theo} \label{theo: aconstruction}
Let $X$ be a geometrically uniruled smooth projective variety defined over $F$ and let $L$ be a big and nef $\mathbb Q$-divisor on $X$.  There exists a constructible bounded subset $T \subset \Hilb(X)$ of the Hilbert scheme of $X$, a decomposition $T = \cup T_{i}$ of $T$ into locally closed subvarieties, and smooth projective morphisms $p_{i}: \mathcal U_{i} \to T_{i}$ and morphisms $s_{i}: \mathcal U_{i} \to X$ such that
\begin{itemize}
\item over $\overline{F}$, each fiber of $\overline{p}_i : \overline{\mathcal U}_i \rightarrow \overline{T}_i$ is an integral uniruled variety which is mapped birationally by $\overline{s}_{i}$ onto the subvariety of $\overline{X}$ parametrized by the corresponding point of $\Hilb(\overline{X})$;
\item every fiber $Y$ of $p_{i}$ is a smooth variety satisfying $a(Y,s_{i}^{*}L|_{Y}) \geq a(X,L)$ and is adjoint rigid with respect to $s_{i}^{*}L|_{Y}$;
and
\item for every subvariety $Y \subset X$ not contained in $\mathbf{B}_{+}(L)$ which satisfies $a(Y,L|_{Y}) \geq a(X,L)$ and which is adjoint rigid with respect to $L$, there is some index $i$ such that $Y$ is birational to a fiber of $p_{i}$ under the map $s_{i}$.
\end{itemize}
Furthermore, if $s_{i}: \mathcal U_{i} \to X$ is dominant then $s_i$ must be generically finite.  
\end{theo}

Related results have appeared in \cite{HJ16} and \cite{LTDuke}.  We will give a quick proof by appealing to \cite{birkar16b}.  

\begin{proof}

Suppose that we have a subvariety $\overline{Y} \subset \overline{X}$ which is adjoint rigid with respect to $L$, not contained in $\mathbf{B}_{+}(L_{\overline{X}})$, and has $a$-value at least as large as $a(X,L)$.  Let $\psi: \widetilde{\overline{Y}} \to \overline{Y}$ be a resolution.  By \cite[Theorem 3.5]{LTDuke}, there is a birational contraction $\phi: \widetilde{\overline{Y}} \dashrightarrow \overline{Y}'$ where $\overline{Y}'$ is a $\mathbb{Q}$-factorial terminal weak Fano variety such that $a(\overline{Y},L|_{\overline{Y}})\phi_{*}\psi^{*}L|_{\overline{Y}} \equiv -K_{\overline{Y}'}$.  By \cite[Theorem 2.11]{birkar16b} (whose validity is established in all dimensions by induction), we see that there is some constant $C$ depending only on the dimension of $X$ such that
\begin{align*}
C \geq (-K_{\overline{Y}'})^{\dim \overline{Y}'} & = a(\overline{Y},L|_{\overline{Y}})^{\dim \overline{Y}} \left(\phi_{*}\psi^{*}L|_{\overline{Y}}\right)^{\dim \overline{Y}} \\
& \geq a(\overline{Y},L|_{\overline{Y}})^{\dim \overline{Y}} (\psi^{*}L|_{\overline{Y}})^{\dim \overline{Y}} \\
& = a(\overline{Y},L|_{\overline{Y}})^{\dim \overline{Y}} L|_{\overline{Y}}^{\dim \overline{Y}}.
\end{align*}
Thus we have
\begin{equation*}
L|_{\overline{Y}}^{\dim \overline{Y}} \leq C/a(\overline{Y},L|_{\overline{Y}})^{\dim \overline{Y}} \leq C/a(\overline{X},L)^{\dim \overline{Y}}.
\end{equation*}
By applying \cite[Lemma 4.7]{LTT14} to each irreducible component of $\overline{X}$ we conclude that such subvarieties $\overline{Y}$ are parametrized by a bounded subset of $\Hilb(\overline{X})$.  
We let $\overline{M} \subset \Hilb(\overline{X})$ denote a finite union of components which contains this locus (equipped with the reduced structure). 

Let $\overline{N} \subset \overline{M}$ denote the constructible subset over which the universal family has irreducible and reduced fibers.  By repeatedly resolving singularities in the fibers and stratifying the base, we find a finite union of locally closed subsets $\overline{N}_{i} \subset \overline{N}$ and smooth morphisms $\overline{p}_{i}: \overline{\mathcal U}_{i} \to \overline{N}_{i}$ whose fibers are smooth projective varieties birational to the subvarieties of $\overline{X}$ parametrized by points of $\overline{N}$.  We next replace each $\overline{N}_{i}$ by the open subset parametrizing subvarieties not contained in $\mathbf{B}_{+}(L)$; this ensures that $L$ is relatively big and nef on the universal family over $\overline{N}_{i}$.  Applying Theorem \ref{theo: aconstancy} and further stratifying the $\overline{N}_{i}$, we may suppose that the $a$-invariant and Iitaka dimension of the adjoint pair is constant in each family.  In particular, there is a constructible sublocus $\overline{T} \subset \overline{M}$ parametrizing the subvarieties which are not contained in $\mathbf{B}_{+}(L)$, have $a$-invariant at least $a(X,L)$ and are adjoint rigid with respect to $L$.

Now we note that $\overline{T}$ is defined over the ground field and thus descends to a subset $T$ of $\Hilb(X)$.  Indeed, given any subvariety $Y$ parametrized by a point of $\overline{T}$ and any $\sigma \in \Gal(\overline{F}/F)$, by Proposition \ref{prop: galinvofa} $\sigma(Y)$ is also parametrized by $\overline{T}$.
 Furthermore $\mathbf{B}_{+}(L_{\overline{X}})$ descends to $\mathbf{B}_{+}(L)$.  By repeatedly resolving singularities of fibers of the universal family over $T$ and taking finer stratifications of the base, we construct smooth families $p_{i}: \mathcal U_{i} \to T_{i}$ whose fibers are birational to the subvarieties parametrized by $T$.  We can further stratify $T$ so that after base-change to $\overline{F}$ we obtain a substratification of the original stratification of $\overline{T}$; since the $a$-invariant and Iitaka dimension are not affected by base change (see Corollary \ref{coro: flatbasechange}) these families will now have the desired properties. Altogether we established all claims in the statement except the last claim.

Assume that $s_i : \mathcal U_i \to X$ is dominant. We claim that a general member $Y$ of $p_i : \mathcal U_i \rightarrow T_i$ is contracted by $\pi$ where $\pi : X \dashrightarrow W$ is the canonical map associated to $a(X, L)L + K_X$.
Indeed, take a general complete intersection $T \subset T_i$ of dimension $\dim(X) - \dim(Y)$.  Then the base change $\mathcal U_{i, T}$ has an evaluation map $s_i' : \mathcal U_{i, T} \to X$ which is dominant and generically finite. The ramification formula implies that there exists an effective divisor $R \geq0$ on $\mathcal U_{i, T}$ such that for a general member $Y$ of $\mathcal U_{i, T}$, we have
 \begin{align*}
 a(X, L)s_i^*L|_Y + K_Y &= (a(X, L)(s_i')^*L + K_{\mathcal U_{i, T}})|_Y\\
 &= (a(X, L)(s_i')^*L + (s_i')^*K_X + R)|_Y.
 \end{align*}
 Since this divisor has Iitaka dimension $0$, we conclude that $(a(X, L)(s_i')^*L + (s_i')^*K_X)|_{Y}$ also has Iitaka dimension $0$, proving our claim.

 Since general fibers of $\pi : X \dashrightarrow W$ are adjoint rigid, \cite[Proposition 4.14]{LTDuke} applied to a general fiber of $\pi$ shows that the dominant family of adjoint rigid subvarieties contained in that fiber will have a generically finite evaluation map.  We deduce that the evaluation map $s_{i}$ will also be generically finite.
\end{proof}

It will be convenient to work with two variants of Theorem \ref{theo: aconstruction}.
The first focuses on subvarieties with strictly larger $a$-invariant.  \cite[Theorem 1.3]{HJ16} proves a similar result when $L$ is big and semiample. 

\begin{theo}
\label{theo: HJ}
Let $X$ be a geometrically uniruled smooth projective variety and let $L$ be a big and nef $\mathbb Q$-divisor on $X$.
\begin{enumerate}
\item As we vary over all projective subvarieties $Y$ not contained in $\mathbf{B}_{+}(L)$, the set
\begin{equation*}
\{ \, a(Y,L|_{Y}) \, | \, Y \not \subset \mathbf{B}_{+}(L) \} \cap [a(X,L),\infty)
\end{equation*}
is finite.
\item For any fixed $a \geq a(X,L)$, let $V^{a}$ denote the union of all subvarieties $Y$ such that $a(Y, L|_Y) > a$.
Then $V^{a} \subsetneq X$ is a proper closed subset and each component $V^{a}_{d} \subset V^{a}$ satisfies $a(V^{a}_{d}, L|_{V^{a}_{d}}) > a$.
\end{enumerate}
\end{theo}

\begin{proof}
(1) follows immediately from the finiteness of the families in Theorem \ref{theo: aconstruction}.  To see (2), we first combine the finiteness of the families in Theorem \ref{theo: aconstruction} with Lemma \ref{lemm:ainvdominantfamily} to see that $V^{a}$ is not Zariski dense in $X$.  This argument also shows that each component of the Zariski closure of $V^{a}$ which is not contained in $\mathbf{B}_{+}(L)$ admits a dominant family of subvarieties $Y$ such that $a(Y,L|_{Y}) > a$.  By again applying Lemma \ref{lemm:ainvdominantfamily} we see that each component of the Zariski closure of $V^{a}$ which is not contained in $\mathbf{B}_{+}(L)$ has $a$-value larger than $a$.  Furthermore \cite{Nakamaye00} shows that each component $B_{d}$ of $\mathbf{B}_{+}(L)$ satisfies $a(B_{d},L|_{B_{d}}) = \infty$.  Together these imply that $V^{a}$ is closed and that $a(V^{a}_{d}, L|_{V^{a}_{d}}) > a$ for every component $V^{a}_{d}$.
\end{proof}

The second variant replaces the families in Theorem \ref{theo: aconstruction} with projective closures and replaces the smooth families with the corresponding universal families in $\Hilb(X)$.

\begin{theo}
\label{theo: rigidfamilies}
Let $X$ be a geometrically uniruled smooth projective variety and let $L$ be a big and nef $\mathbb Q$-divisor on $X$.
Then there exist a proper closed subset $V$ of $X$ and finitely many families $p_i : \mathcal U_i \rightarrow W_i$ of closed subschemes of $X$ where $W_i$ is a projective subscheme of $\Hilb(X)$ such that
\begin{itemize}
\item over $\overline{F}$, $\overline{p}_i : \overline{\mathcal U}_i \rightarrow \overline{W}_i$ generically parametrizes integral uniruled subvarieties of $\overline{X}$;  
\item for each $i$, the evaluation map $s_i : \mathcal U_i \rightarrow X$ is generically finite and dominant;
\item for each $i$, a general member $Y$ of $p_i$ is a subvariety of $X$ such that $(Y,L|_{Y})$ is adjoint rigid and $a(X,L) = a(Y,L|_{Y})$; and
\item for any subvariety $Y$ such that $(Y,L|_{Y})$ is adjoint rigid and $a(Y, L|_Y) \geq a(X, L)$, either $Y$ is contained in $V$ or $Y$ is a member of a family $p_i : \mathcal U_i \rightarrow W_i$ for some $i$.
\end{itemize}
\end{theo}

\begin{proof}
First, we replace the loci $T_{i} \subset \Hilb(X)$ in Theorem \ref{theo: aconstruction} with their projective closures $W_{i}$ (equipped with the reduced structure). Let $p_{i}: \mathcal{U}_{i} \to W_{i}$ be the universal family equipped with evaluation maps $s_{i}: \mathcal{U}_{i} \to X$.  To start with we set $V = \mathbf{B}_{+}(L)$.  Then for each $i$ such that $s_{i}: \mathcal{U}_{i} \to X$ is not dominant, we add the image of $s_{i}$ to $V$ and remove $\mathcal{U}_{i}$ from our set of families.  By Theorem \ref{theo: HJ}, any subvariety with $a$-invariant larger than $X$ will be contained in $V$. The desired properties of the $\mathcal{U}_{i}$ then follow directly from Theorem \ref{theo: aconstruction}.
\end{proof}

Finally, we will need two results useful for understanding dominant breaking thin maps.  The cited statements are proved over an algebraically closed field but the extension to arbitrary fields of characteristic $0$ is immediate via Corollary \ref{coro: flatbasechange}.

\begin{theo}{\cite[Corollary 2.20]{Sen17}}
\label{theo: akash} 
Let $X$ be a geometrically uniruled smooth projective variety and let $L$ be a big and nef $\mathbb{Q}$-divisor on $X$.
Suppose that $f: Y \rightarrow X$ is a dominant generically finite morphism with $Y$ smooth projective and with $a(Y, f^*L)=a(X, L)$.  Suppose $R_{i}$ is a component of the ramification divisor $R$ on $Y$ which dominates the base of the canonical model map for $K_{Y} + a(Y,f^{*}L)f^{*}L$ and whose image $B_i$ is a component of the branch divisor $B$ on $X$. Then
\[
a(B_i, L|_{B_i}) > a(X, L).
\]
\end{theo}

\begin{prop}{\cite[Proposition 2.17]{Sen17}}
\label{prop: degree}
Let $X$ be a geometrically uniruled smooth projective variety and let $L$ be a big and nef $\mathbb Q$-divisor on $X$. Then there exists a constant $M$ only depending on $\dim X$, $L^{\dim X}$, and $a(X, L)$ such that for any dominant thin map $f: Y \rightarrow X$ such that $Y$ is geometrically integral, $a(Y, f^{*}L) = a(X, L)$, and $(Y, f^*L)$ is adjoint rigid, we have
\[
\deg (f: Y \rightarrow f(Y)) \leq M.
\]
\end{prop}

\subsection{Fiber dimension}
\label{subsec: fiberdim}

As discussed earlier, we will frequently study the geometry of an arbitrary polarized pair $(X,L)$ by replacing it with the fibers of its canonical model.  In this section we define and study an invariant which helps us understand this replacement operation.

\begin{defi}
Let $X$ be a smooth projective variety defined over a field $F$ of characteristic $0$. Let $L$ be a big and nef $\mathbb Q$-divisor on $X$. We define
\[
d(X, L) = \dim(X) - \kappa(X,K_{X} + a(X,L)L).
\]
When $X$ is singular, we define $d(X,L)$ by pulling $L$ back to a smooth birational model of $X$.  Note that $d(X,L)$ is invariant under extension of the ground field.
\end{defi}

The following lemma shows that any adjoint rigid subvarieties of $X$ with dimension larger than $d(X,L)$ must be contained in a closed subset.  We will use it to show that adjoint rigid subvarieties only contribute a thin subset of rational points.

\begin{lemm}
\label{lemm:d(Y)>d(X)}
Let $X$ be a geometrically uniruled smooth projective variety and let $L$ be a big and nef $\mathbb{Q}$-divisor on $X$.  There is a proper closed subset $V \subsetneq X$ such that if $f: Y \to X$ is any thin map satisfying $a(Y,f^{*}L) \geq a(X,L)$ and $d(Y,L) > d(X,L)$ then $f(Y) \subset V$.
\end{lemm}

\begin{proof}
Let $V \subset X$ denote the closed subset defined by Theorem \ref{theo: rigidfamilies}.  Assume for a contradiction that $f(Y) \not \subset V$.  In particular this implies that $a(Y,f^{*}L) = a(f(Y),L|_{f(Y)}) = a(X,L)$.  Let $\Gamma$ denote the closure of a general fiber of the canonical model for $(Y,a(Y,f^{*}L)f^{*}L)$ and let $S$ denote its image in $X$.  By Lemma \ref{lemm:dominantequalitycase} we see that $S$  has the same $a$-invariant as $X$ and is adjoint rigid with respect to the restriction of $L$.  Since the general such $S$ is not contained in $V$, as we vary $\Gamma$ the images $S$ are parametrized by a family $p_{i}: \mathcal{U}_{i} \to W_{i}$ as in Theorem \ref{theo: rigidfamilies} such that the evaluation map $s_{i}$ is dominant (even if $f$ is not dominant).  
To summarize this discussion, we will obtain the desired contradiction if we can prove that there is no dominant family of subvarieties $S$ such that $(S,L|_{S})$ is adjoint rigid, $a(S,L|_{S}) = a(X,L)$ and $\dim(S) > d(X,L)$.

To prove this, we may replace $X$ by any birational model.  In particular we may suppose that the canonical model map for $K_{X} + a(X,L)L$ is a morphism $\pi: X \to W$.  Thus there is an ample $\mathbb{Q}$-divisor $H$ on $W$ and an effective $\mathbb{Q}$-divisor $E$ such that $K_{X} + a(X,L)L$ is $\mathbb{Q}$-linearly equivalent to $\pi^{*}H + E$.  Suppose we have a diagram of smooth varieties
\begin{equation*}
\xymatrix{ \mathcal{S} \ar[r]^{g} \ar[d]_{q} &  X \\
T & }
\end{equation*}
such that $g$ is generically finite and dominant and the fibers of $q$ are smooth varieties $S_t$ satisfying $a(S_t,g^{*}L|_{S_t}) = a(X,L)$ and $\dim(S_t) > d(X,L)$.  We can write 
\begin{equation*}
K_{\mathcal{S}} + a(X,L)g^{*}L = g^{*}(K_{X} + a(X,L)L) + R \sim_{\mathbb{Q}} g^{*}\pi^{*}H + (g^{*}E + R)
\end{equation*}
for some effective divisor $R$. Note that the restriction of $g^{*}\pi^{*}H$ to a general fiber of $q$ has Iitaka dimension at least $1$, so that the general fiber of $q$ can not possibly be adjoint rigid with respect to the restriction of $L$.  This proves the desired contradiction.
\end{proof}

\subsection{Face contraction}
\label{subsec: facecontracting}

The notion of face contraction refines the $b$-invariant.  We will use it to help classify when the point contributions of a dominant map with equal geometric invariants should lie in the exceptional set for Manin's Conjecture.  

We will be interested in the following situation:

\begin{assum} \label{assum:fc}
Let $X$ be a geometrically uniruled geometrically integral smooth projective variety defined over a field $F$ of characteristic $0$ and let $L$ be a big and nef $\mathbb{Q}$-divisor on $X$.  Let $f: Y \to X$ be a morphism of smooth projective varieties that is generically finite onto its image.  Suppose 
that either
\begin{enumerate}
\item $f$ is dominant and $a(Y,f^{*}L) = a(X,L)$, or
\item $a(Y,f^{*}L) = a(X,L)$, $d(Y,f^{*}L) = d(X,L)$, and there is a commuting diagram
\begin{equation*}
\xymatrix{ Y \ar[r]^{f} \ar[d]_{\pi_{Y}} &  X \ar[d]_{\pi_{X}} \\
T \ar[r] & W}
\end{equation*}
where $\pi_{Y}$ and $\pi_{X}$ are the canonical models for the adjoint pairs and the general fiber of $\pi_{Y}$ maps onto a fiber of $\pi_{X}$ which has the same $a$-value as $X$ and is contained in the locus where $\pi_{X}$ is smooth.
\end{enumerate}
\end{assum}

\begin{lemm} \label{lemm:pushforwardpreservesfaces}
Suppose we are in the situation of Assumption \ref{assum:fc}. Then the pushforward map $f_{*}: N_{1}(Y) \to N_{1}(X)$ satisfies $f_{*}\mathcal{F}_{Y,f^{*}L} \subset \mathcal{F}_{X,L}$.
\end{lemm}

\begin{proof}
In case (1) $f$ is a dominant map, so by the Riemann-Hurwitz formula we have $K_{Y} \geq f^{*}K_{X}$.  Since the $a$-invariants are the same we have $K_{Y} + a(Y,f^{*}L)f^{*}L \geq f^{*}(K_{X} + a(X,L)L)$.  Thus any nef curve class on $Y$ with vanishing intersection against $K_{Y} + a(Y,f^{*}L)L$ also has vanishing intersection against $f^{*}(K_{X} + a(X,L)L)$, showing that $f_{*}\mathcal{F}_{Y,f^{*}L} \subset \mathcal{F}_{X,L}$. 

Suppose that we are in case (2).  Fix a general closed point $w$ in the $(\pi_{X}\circ f)$-image of $Y$.  By assumption $X_{w}$ is smooth.  Let $Y_{t}$ denote the fiber over a closed point $t \in T$ mapping to $w$. Since the fiber $Y_{t}$ is general, Lemma \ref{lemm: monodromyandbvalue} (3) shows that pushforward induces a surjection $\mathcal{F}_{Y_{t},f^{*}L|_{Y_{t}}} \to \mathcal{F}_{Y,f^{*}L}$.  By Lemma \ref{lemm: monodromyandbvalue} (2) pushforward induces a map $\mathcal{F}_{X_w,L|_{X_{w}}} \to \mathcal{F}_{X,L}$.  Thus it suffices to prove that $f|_{Y_{t}*} \mathcal{F}_{Y_{t},f^{*}L|_{Y_{t}}} \subset \mathcal{F}_{X_{w},L|_{X_{w}}}$.

Since $f|_{Y_{t}}: Y_{t} \to X_{w}$ is a dominant generically finite map of smooth varieties, by the Riemann-Hurwitz formula we have that $K_{Y_{t}} - f|_{Y_{t}}^{*}K_{X_{w}}$ is effective.  Since by Lemma \ref{lemm:ainvandcanonicalfibers} the $a$-invariants are the same this implies that $f|_{Y_{t}*}\mathcal{F}_{Y_{t},f^{*}L|_{Y_{t}}} \subset \mathcal{F}_{X_{w},L|_{X_{w}}}$ by the same argument as before. 
\end{proof}

Note that by the birational invariance of faces proved in Lemma \ref{lemm:birfaceinv}, any morphism $f: Y \to X$ with a birational model that satisfies Assumption \ref{assum:fc} will still satisfy $f_{*}\mathcal{F}_{Y,f^{*}L} \subset \mathcal{F}_{X,L}$.  This leads us to the following definition.

\begin{defi}{\cite[Definition 3.6]{LT17}} \label{defi: facecontraction}
Let $f: Y \to X$ be a morphism of geometrically integral projective varieties that is generically finite onto its image.  Furthermore assume that there is a commuting diagram
\begin{equation*}
\xymatrix{ Y' \ar[r]^{f'} \ar[d]_{\phi_{Y}} &  X' \ar[d]^{\phi_{X}} \\
Y \ar[r]^{f} & X}
\end{equation*}
such that $\phi_{Y}$ and $\phi_{X}$ are birational and $f': Y' \to X'$ satisfies Assumption \ref{assum:fc}.

We say that such a morphism $f$ is face contracting if the induced map $f_{*}: \mathcal{F}_{Y,f^{*}L} \to \mathcal{F}_{X,L}$ is not injective.
\end{defi}

Since the dimensions of $\mathcal{F}_{Y,f^{*}L}$ and $\mathcal{F}_{X,L}$ are $b(F,Y,f^{*}L)$ and $b(F,X,L)$ respectively, a dominant breaking thin map is automatically face contracting.  However, the converse is not true (see \cite[Example 3.7]{LT17}).

The following lemma shows that for certain families $Y$ of adjoint rigid subvarieties on $X$ the map $f: Y \to X$ must be a face contracting map. We will use it to show that subvarieties of this type can only contribute a thin set of rational points.

\begin{lemm} \label{lemm: facecontractingcondition}
Let $f: Y \to X$ be a dominant generically finite morphism of geometrically integral projective varieties with $X$ geometrically uniruled and fix a big and nef $\mathbb{Q}$-divisor $L$ on $X$.  Suppose there is a birational model $f': Y' \to X'$ of $f$ with birational maps $\phi_X: X' \to X$, $\phi_Y : Y' \to Y$ and a diagram 
\begin{equation*}
\xymatrix{ Y' \ar[r]^{f'} \ar[d]_{q} &  X' \ar[d]_{p} \\
T \ar[r]^{g} & W}
\end{equation*}
satisfying the following conditions:
\begin{enumerate}
\item $X'$ and $Y'$ are smooth and projective,
\item $q$ and $p$ are projective and surjective,
\item $g$ is generically finite and dominant,
\item $a(Y,f^{*}L) = a(X,L)$ and $b(F,Y,f^{*}L) = b(F,X,L)$,
\item $q$ is birationally equivalent to the canonical model for $K_{Y'} + a(Y',f'^{*}\phi_X^{*}L) f'^{*}\phi_X^{*}L$, 
\item $\dim(W) > \kappa(X,K_{X} + a(X,L)L)$.\end{enumerate}
Then $f$ is face contracting.
\end{lemm}

\begin{proof}
Note that for any diagram of smooth varieties
\begin{equation*}
\xymatrix{ Y'' \ar[r]^{f''} \ar[d]_{\psi_{Y}} &  X'' \ar[d]_{\psi_{X}} \\
Y' \ar[r]^{f'} &  X' }
\end{equation*}
with $\psi_{Y}, \psi_{X}$ birational the hypotheses of the theorem still hold for $f''$.  Thus by passing to birational models (which for convenience we absorb into the notation) we may assume that $Y' = Y$, $X' = X$, and the canonical model for $K_{X} + a(X,L)L$ is a morphism.

Let $\mathcal{F}_{Y,f^{*}L}$ and $\mathcal{F}_{X,L}$ be the faces as in Definition \ref{defi:facedefinition} with respect to $f^{*}L$ and $L$ respectively.  Fix an ample divisor $H$ on $W$.  Note that $p^{*}H$ vanishes on every element of $f_{*}\mathcal{F}_{Y,f^{*}L}$, since there is some $\epsilon > 0$ such that $K_{Y} + a(Y,f^{*}L)f^{*}L - \epsilon f^{*}p^{*}H$ is pseudo-effective.  However, $p^{*}H$ does not vanish on every element of $\mathcal{F}_{X,L}$.  Indeed, Lemma \ref{lemm:alternativedescription} shows that for every divisor $D$ with class in $\mathcal{F}^{X,L}$ there is some sufficiently small $\epsilon > 0$ such that $K_{X} + a(X,L)L - \epsilon D$ is $\mathbb{Q}$-linearly equivalent to an effective divisor.  Since the Iitaka dimension of $p^{*}H$ is greater than $\kappa(X,K_{X} + a(X,L)L)$, we deduce that $p^{*}H \not \in \mathcal{F}^{X,L}$.  Thus $f_{*}\mathcal{F}_{Y,f^{*}L} \subsetneq \mathcal{F}_{X,L}$, and since the $b$-values are equal $f$ must be face contracting.
\end{proof}

\section{A conjectural description of exceptional sets in Manin's Conjecture}
\label{sect: conjecturaldescription}

Let $F$ be a number field and suppose that we have a geometrically rationally connected and geometrically integral smooth projective variety $X$ defined over $F$ carrying a big and nef line bundle $\mathcal L = \mathcal{O}_{X}(L)$ with an adelic metrization on $X$.  The adelic metrization defines a height function on the rational points of $X$.  Manin's conjecture predicts the asymptotic growth rate of the counting function for rational points of bounded height after removing an exceptional thin set.   

Originally \cite{BM} and its refinement \cite{Peyre} predicted that the exceptional set for Manin's Conjecture consisted of points on a proper closed subset.  However there are now many counterexamples to these two versions of Manin's Conjecture (\cite{BT-cubic}, \cite{EJ06}, \cite{Els11}, \cite{BL16}, and \cite{LeRudulier}). These counterexamples arise from geometric obstructions; for example, it is possible that as we vary over breaking thin maps $f: Y \to X$ the union of the sets $f(Y(F))$ is Zariski dense.  \cite{Peyre03} was the first to modify the conjecture by proposing that the exceptional set in Manin's Conjecture is contained in a thin set. Here we give a conjectural geometric description of the exceptional set in Manin's Conjecture.  In fact, we will construct this set over an arbitrary field $F$ of characteristic $0$:

\begin{defi}
\label{defi:exceptionalsets}
Let $F$ be a field of characteristic $0$.
Suppose that $X$ is a geometrically uniruled and geometrically integral smooth projective variety over $F$ equipped with a big and nef $\mathbb{Q}$-divisor $L$.
When constructing the exceptional set it is harmless to replace $X$ by a birational model, thus we will assume that the canonical model $\pi : X\rightarrow W$ for $K_{X} + a(X, L)L$ is a morphism.

\

\noindent
{\bf The definition of $Z_0$}:
Let $Z_0$ be the set of rational points contained in the union of $\mathbf B_+(L)$ and a proper closed subset $\pi^{-1}V$ where $V \subset W$ is a proper closed subset such that over $W^\circ = W \setminus V$, $\pi$ is smooth and for any closed point $w \in W \setminus V$, we have $a(X_w, L|_{X_w}) = a(X, L)$.
Note that $Z_0$ consists of points on a proper closed subset of $X$.

\

\noindent
{\bf The definition of $Z_1$}:
As $f: Y \rightarrow X$ varies over all $F$-thin maps such that $Y$ is geometrically integral and smooth, $d(Y,f^{*}L) < d(X,L)$ and 
\[
(a(X, L), b(F, X, L)) \leq (a(Y, f^*L), b(F, Y, f^*L)),
\]
we define the set $Z_1 \subset X(F)$ by
\[
Z_1 = \bigcup_f f(Y(F)) \subset X(F).
\]

\

\noindent
{\bf The definition of $Z_2$}:
As $f: Y \rightarrow X$ varies over all $F$-thin maps such that $Y$ is geometrically integral and smooth, $d(Y,f^{*}L) = d(X,L)$, $f(Y)$ meets with $\pi^{-1}(W^\circ)$, and either
\[
(a(X, L), b(F, X, L)) < (a(Y, f^*L), b(F, Y, f^*L)),
\]
or the $a$ and $b$ values are equal and $f$ is face contracting, we define the set $Z_2 \subset X(F)$ by
\[
Z_2 = \bigcup_f f(Y(F)) \subset X(F).
\]

\

\noindent
{\bf The definition of $Z_3$}:
As $f: Y \rightarrow X$ varies over all $F$-thin maps such that $Y$ is geometrically integral and smooth, $d(Y,f^{*}L) > d(X,L)$ and 
\[
(a(X, L), b(F, X, L)) \leq (a(Y, f^*L), b(F, Y, f^*L)),
\]
we define the set $Z_3 \subset X(F)$ by
\[
Z_3 = \bigcup_f f(Y(F)) \subset X(F).
\]
By Lemma~\ref{lemm:d(Y)>d(X)}, $Z_{3}$ is contained in a proper closed subset of $X$.

\end{defi}

We propose the following refinement of Manin's Conjecture which includes a description of the exceptional thin set.  A similar but weaker statement was predicted in \cite{LTDuke}.  For any subset $Q \subset X(F)$, we let $N(Q, \mathcal L, T)$ denote the number of rational points on $Q$ whose height associated to $\mathcal L$ is bounded above by $T$.

\begin{conj}[Manin's Conjecture] \label{conj: maninsconjecture}
Let $F$ be a number field.  Let $X$ be a geometrically rationally connected and geometrically integral smooth projective variety defined over $F$ and let $\mathcal{L}$ be a big and nef line bundle with an adelic metrization on $X$.

Let $Z$ be the union of $Z_0$, $Z_1$, $Z_2$, and $Z_3$.  
Suppose that $X(F)$ is not a thin set.
Then we have
\[
N(X(F) \setminus Z, \mathcal L, T) \sim c(F, Z, \mathcal L)T^{a(X, L)} \log (T)^{b(F, X, L)-1}
\]
as $T \rightarrow \infty$ where $c(F, Z, \mathcal L)$ is Peyre-Batyrev-Tschinkel's constant introduced in \cite{Peyre} and \cite{BT}.
\end{conj}

\begin{rema}
Assuming the conjecture of Colliot-Th\'el\`ene that the Brauer-Manin obstructions are the only obstructions to weak approximation for  geometrically rationally connected smooth projective varieties, it follows that $X(F)$ is not thin as soon as there is a rational point when $F$ is a number field. See the remark after Conjecture 1.4 in \cite{BL16}.
\end{rema}

\begin{rema}
When $(X, L)$ is adjoint rigid, the constant $c(F,Z,\mathcal L)$ does not depend on $Z$.  But if $(X,L)$ is not adjoint rigid, then the definition of $c(F,Z,\mathcal L)$ involves a summation of Tamagawa numbers over the base of the canonical map so that we must keep track of which fibers are removed by $Z$ when defining the constant.  
\end{rema}

\begin{rema}
By \cite{HM07} a smooth projective variety $X$ is geometrically rationally connected whenever it carries a big and nef $\mathbb{Q}$-divisor $L$ such that $(X,L)$ is adjoint rigid. 
(For a careful explanation see \cite[Proof of Theorem 4.5]{LTT14}.)
\end{rema}

\begin{rema}
\cite{Peyre16} formulates an appealing version of Manin's Conjecture using the notion of freeness of a rational point.  Peyre's conjecture has some similarities with Conjecture \ref{conj: maninsconjecture}.  Let $Z^{f}$ denote the exceptional set as in \cite[Formule empirique 6.13]{Peyre16}.   \cite[Proposition 5.8]{Peyre16} shows that $Z^{f}$ includes most points on non-free curves; comparing against \cite[Theorem 1.1 and Proposition 6.14]{LT17}  we should expect these points to account for subvarieties $Y$ with $a(Y,L) > a(X,L)$.

Nevertheless, the two proposals for the exceptional set are different.  The set $Z^{f}$ may fail to be contained in the union of the $Z_{i}$: a general cubic fourfold has empty $Z_{i}$ but admits non-free lines so that $Z^{f}$ is non-empty by \cite[Proposition 5.8]{Peyre16}.  Conversely, the union of the $Z_{i}$ may fail to be contained in $Z^{f}$:  in the example of \cite{BT-cubic} the $Z_{i}$ contains every point on a cubic surface fiber with Picard rank $> 1$ while $Z^{f}$ does not (see \cite[Section 8.3]{Peyre16} and particularly \cite[Remarque 8.9]{Peyre16}). 
\end{rema}

The main theorem of this paper is the following result:

\begin{theo} \label{theo: precisetheorem}
Let $F$ be a field of characteristic $0$. 
Let $X$ be a geometrically uniruled geometrically integral smooth projective variety defined over $F$ and let $L$ be a big and nef $\mathbb Q$-divisor on $X$.  The subsets $Z_{0}$, $Z_1$, $Z_2$, and $Z_{3}$ in Definition~\ref{defi:exceptionalsets} are contained in a thin subset of $X(F)$.
\end{theo}

\begin{rema} \label{rema: computability}
To establish Manin's Conjecture for specific examples Conjecture \ref{conj: maninsconjecture} predicts that one should first calculate the sets $Z_{0}, Z_{1}, Z_{2}, Z_{3}$.  Our proof of Theorem \ref{theo: precisetheorem} will show that in principle by the Borisov-Alexeev-Borisov Conjecture (\cite{birkar16}, \cite{birkar16b}) this computation only involves checking the behavior of subvarieties and covers in a finite degree range.  However, currently Birkar's result is ineffective and in practice there is room for vast improvement of current computational techniques.  For low dimensional examples the framework established by \cite{LTT14}, \cite{LTDuke}, and \cite{Sen17} is often sufficient for calculating these sets.
\end{rema}

Note that in Conjecture \ref{conj: maninsconjecture} we remove point contributions for some thin maps with $a$ and $b$ values equal to $X$.  The following examples show that sometimes, but not always, we must discount contributions from such maps in order to obtain the correct leading constant.  The face contraction condition is the key criterion for distinguishing the two cases.

\begin{exam}[Peyre's constant]
Let $F$ be a number field.
The papers \cite{EJ06}, \cite{Els11}, \cite{BL16}, \cite{BHB18} give many examples of Fano varieties $X$ admitting a Zariski dense set of subvarieties with the same $a$ and $b$ values as $X$ with respect to $-K_{X}$.  Suppose that the rational points on these subvarieties grow at the expected rate.  If we include these points, \cite[Theorem 1.2]{BL16} shows that Manin's Conjecture with Peyre's constant will be violated for an appropriate choice of anticanonical height function.  In order to obtain the correct Peyre's constant we must remove point contributions from all such subvarieties.  Theorem \ref{theo: precisetheorem} shows that such points always lie in a thin set, generalizing the examples proved in \cite{BL16}.
\end{exam}

\begin{exam}
As we vary over all dominant generically finite maps $f: Y \rightarrow X$ of degree $\geq 2$ such that 
\[
(a(X, L), b(F, X, L)) = (a(Y, f^*L), b(F, Y, f^*L)),
\]
the set $\cup f(Y(F))$ need not lie in a thin set of rational points (see \cite[Example 8.7]{LTDuke}).  Thus, in the definition of $Z_{2}$ it is important to only consider contributions from maps which are face contracting.
\end{exam}

Let us compute these exceptional sets for some examples.  In all of the following examples $(X,L)$ will be adjoint rigid, so $Z_{0} = \mathbf B_+(L)$ and $Z_{3}$ will be empty because $d(X, L)$ is equal to the dimension of $X$. Thus they need not be discussed.

\begin{exam}[Surfaces]
Let $F$ be a field of characteristic $0$.  Let $S$ be a geometrically rational geometrically integral smooth projective surface defined over $F$ and let $L$ be a big and nef $\mathbb Q$-divisor on $S$ such that $(S,L)$ is adjoint rigid.  For simplicity let us suppose that the Picard rank of $S$ and the geometric Picard rank of $S$ coincide. Then by \cite[Proposition 5.9]{LTT14} and \cite[Theorem 1.8]{LTDuke} $Z_1$ is contained in a proper closed subset and $Z_2$ is empty. Thus when $F$ is a number field, we expect that Manin's Conjecture should hold after removing points on a closed set. This version of Manin's Conjecture for geometrically rational surfaces has been confirmed for many examples, see e.g.~\cite{dBBD07}, \cite{Bro09}, \cite{Bro10}, and \cite{dBBP12}.
\end{exam}

\begin{exam}[Flag varieties]
Let $F$ be a number field.
Let $X$ be a geometrically integral generalized flag variety defined over $F$ with a rational point
and let $L = -K_X$.
Manin's conjecture for flag varieties has been established in \cite{FMT89} with empty exceptional set.
By \cite{Bor96}, the Brauer-Manin obstructions are the only obstructions to weak approximation, so in particular $X(F)$ is not thin. 
Hence, $Z_1$ does not cover $X(F)$. Since $X$ is homogeneous, this implies that $Z_1$ must be empty. On the other hand, since there are no subvarieties with higher $a$-value and $X$ is simply connected, there is no dominant morphism $f: Y \to X$ such that $a(Y,-f^{*}K_{X}) = a(X,-K_{X})$ and $(Y,-f^{*}K_{X})$ is adjoint rigid by \cite[Corollary 2.20]{Sen17}. Thus we conclude that $Z_2$ is also empty.
\end{exam}

\begin{exam}[Toric varieties]
Let $F$ be a number field.
Let $X$ be a geometrically integral smooth toric variety defined over a number field $F$ and let $L$ be a big and nef divisor on $X$. Manin's Conjecture for such a variety was proved in \cite{BT-general}, \cite{BT-0}, and \cite{Sal98} after removing rational points on the boundary. Suppose that $(X,L)$ is adjoint rigid.
Since any $F$-torus satisfies the weak weak approximation property, by \cite[Theorem 3.5.7]{Serre} we see that $X(F)$ is not thin.  Since the torus part is a homogeneous space, we conclude that $Z_1$ is contained in the boundary. By the same reasoning $Z_2$ is also contained in the boundary.  
So our refinement is compatible with the above results. A similar proof works for smooth equivariant compactifications of other algebraic groups and Manin's Conjecture for such varieties has been established in many cases, see e.g.~\cite{CLT02}, \cite{STBT07} and \cite{ST16}.
\end{exam}

\begin{exam}[Le Rudulier's example]
Let $S$ be the surface $\mathbb P^1 \times \mathbb P^1$ over $\mathbb{Q}$ and set $X = \mathrm{Hilb}^{[2]}(S)$.
\cite{LeRudulier} proved Manin's Conjecture for $(X,-K_X)$.
We briefly explain why her result is compatible with our refinement.
We freely use the notations from \cite[Section 9.3]{LTDuke}.
Let $L = H_1[2] + H_2[2]$. Le Rudulier proved Manin's Conjecture for $L$ after removing rational points on $D_1, D_2, E$ and $f(W(\mathbb{Q}))$. We denote this exceptional set by $Z'$.  The analysis in \cite[Section 9.3]{LTDuke} shows that (i) all subvarieties with higher $a$ values are contained in $D_1, D_2$, or $E$; (ii) the only thin maps $g: Y \rightarrow X$ such that the image is not contained in $D_1\cup D_2\cup E$, $(Y, g^*L)$ is adjoint rigid, $\dim Y < \dim X$, and $(a(X, L), b(\mathbb{Q}, X, L)) \leq (a(Y, g^*L), b(\mathbb{Q}, Y, g^*L))$ are the images of the fibers of one of the projections $\pi_i : W \rightarrow \mathbb P^1$. These imply that $Z_1$ is contained in $Z'$. To analyze $Z_{2}$, we first note that the geometric fundamental group of $X \setminus (D_1 \cup D_2 \cup E)$ is $\mathbb Z/2\mathbb Z$.  Thus, over $\overline{\mathbb Q}$ there is only one possible cover $f: W \rightarrow X$ such that $a(W,-f^{*}K_{X}) = a(X,-K_{X})$ and $(W,-f^{*}K_{X})$ is adjoint rigid.  On the other hand, by copying the argument of \cite[Example 8.6]{LTDuke} in this setting we see that all nontrivial twists of $f: W \rightarrow X$ have $a, b$ values less than $a, b$ values of $X$.  Thus $Z_2 = f(W(\mathbb{Q}))$ is also contained in $Z'$.

Another interesting example of \cite{LeRudulier} is $\Hilb^{[2]}(\mathbb{P}^{2})$ over $\mathbb{Q}$. To obtain the expected growth rate of points of a bounded height, one must remove points from a dominant map $f: W \to \Hilb^{[2]}(\mathbb{P}^{2})$ but not its twists.  It turns out that $f$ is face contracting but its twists are not, giving a geometric explanation of this phenomenon; see \cite[Example 8.6]{LTDuke}.
\end{exam}

The circle method has been successfully used to prove Manin's Conjecture for low degree complete intersections, e.g., \cite{Bir61} and \cite{BHB17}. Verifying our refinement for this class of varieties is out of reach at this moment.  However, based on the properties of rational curves on low degree hypersurfaces proved by \cite{HRS04}, \cite{BK13}, \cite{RY16}, \cite{BV17} and the connection with $a$ and $b$ invariants proved in \cite{LT17}, we expect that $Z_{1}$ and $Z_{2}$ are empty for general smooth hypersurfaces in $\mathbb{P}^{n}$ of degree $\leq n-2$ and for every smooth hypersurface in $\mathbb{P}^{n}$ of degree $\ll \log_2(n)$.

\subsection{Counterexamples to extensions} \label{sect: nonbigdiv}

The Weil height formalism associates a height function to any adelically metrized big line bundle, and it is interesting to study the asymptotic growth rate of points of bounded height in this context as well.  It is natural to ask whether the $a,b$-invariants (defined in the analogous way) will still predict the asymptotic growth rate of points.  There are a few important classes of varieties for which the $a,b$-invariants for big line bundles do indeed agree with point growth rates: for example, when the variety is a toric variety \cite{BT-general} or an equivariant compactification of a vector group \cite{CLT02}.  

In this section we will give a couple of examples which demonstrate the pathological behavior of the $a,b$-invariants when the polarization $L$ is  big but not nef. Some of these are almost Fano varieties in the sense of Peyre (\cite[D\'efinition 3.1]{Peyre03}). In particular we will concentrate on examples for which $-K_{X}$ is big.  One can find many examples of such varieties using the following construction.  Let $W$ be any smooth projective variety and choose an ample divisor $H$ on $W$ such that $H-K_{W}$ is ample.  Then the projective bundle $X = \mathbb{P}_{W}(\mathcal{O} \oplus \mathcal{O}(-H))$ has big anticanonical divisor.  In this situation we have $a(X,-K_{X})=1$ regardless of the choices of $W$ and $H$, while the behavior of rational points depends very heavily on these choices.

\begin{exam}
Let $W$ be the Craighero-Gattazo surface \cite{CG94} defined over $\mathbb{Q}(\zeta)$ where $\zeta^{3} + \zeta^{2} - 1 = 0$.  Note that $W$ is a surface of general type with $H^{1}(W,\mathcal{O}_{W}) = H^{2}(W,\mathcal{O}_{W}) = 0$.  By \cite[Theorem 6.2]{RTU17} $W$ is simply connected, hence $\Pic(W)_{\mathrm{tor}} = 0$.

Using the construction above we obtain a projective bundle $X$ over $W$ which is ``almost Fano'' in the sense that $-K_{X}$ is big, $H^{i}(X,\mathcal{O}_{X}) = 0$ for $i>0$, and $\Pic(X)$ is torsion free.  Equipping $X$ with the anticanonical polarization we have $a(X,-K_{X}) = 1$.  However, Lang's conjecture predicts that the set of rational points on $X$ should not be Zariski dense even after a finite extension of the base field.
\end{exam}

\begin{exam}
Suppose that $W$ is the self-product of an elliptic curve without complex multiplication and construct a projective bundle $X$ over $W$ as above.  The computations of \cite[Example 1.6]{Cutkosky86} show that if $L$ is a big divisor then $a(X,L)$ can be irrational.  (See \cite[Example 2.6]{HTT15} for details.)  In this situation the set of rational points on $X$ is thin.
\end{exam}

Even when $X$ is rationally connected, the $a,b$-invariants can exhibit pathological behavior when the polarization $L$ is big.

\begin{exam}
Choose a pencil of cubics on $\mathbb P^2$ defined over $\mathbb Q$ such that the generic fiber has a positive Mordell-Weil rank. (For the existence of such a pencil, see \cite{Shioda92} and \cite{Kurumadani}.) Let $W$ be the blow-up of $\mathbb{P}^{2}$ along the base locus of this pencil. 
Then $W$ admits infinitely many $(-1)$-curves which are sections of the elliptic fibration. Choose any ample divisor $H$ on $W$ and set $X = \mathbb{P}(\mathcal{O} \oplus \mathcal{O}(-H))$.  If $D$ denotes the rigid section of $\pi: X \to W$, then we have $K_{X} = -2D - H + K_{W}$.  Since $-K_{W}$ is effective, $-K_{X}$ is big.  We equip $X$ with the anticanonical polarization so that $a(X,-K_{X}) = 1$.

Let $\{ T_{i} \}$ denote the infinite set of $(-1)$-curves on $W$ and let $S_{i} = \pi^{-1}(T_{i})$.  Then we have $K_{S_{i}} - K_{X}|_{S_{i}} = \pi^{*}\mathcal{O}_{T_{i}}(-1)$, showing that $a(S_{i},-K_{X}) >1 =  a(X,-K_{X})$.  Since $S_{i}$ is toric, we know that the $a$-invariant actually does predict the asymptotic growth rate of rational points of bounded height on $S_{i}$.  Thus there is a countable union of subvarieties whose rational points grow faster than the expected rate.  In particular, there is no open subset $U \subset X$ such that for sufficiently small $\epsilon > 0$ the number of rational points on $U$ of height bounded above by $B$ is $O_{\epsilon}(B^{a(X,-K_{X})+\epsilon})$.  This is a counterexample to weak Manin's Conjecture for almost Fano varieties which are rationally connected. 
\end{exam}

\section{Twists}
\label{sec: twists}

In this section we work over an arbitrary field of characteristic $0$.  We start with a lemma we will use frequently throughout the paper.

\begin{lemm}[\cite{Cheltsov04}] \label{lemm: birandaut}
Let $f : Y \dashrightarrow X$ be a dominant generically finite rational map between normal projective varieties defined over a field $F$ of characteristic $0$. Then there exists a birational modification $f': Y' \rightarrow X$ of $f$ such that $Y'$ is smooth and projective and $\mathrm{Bir}(\overline{Y}'/\overline{X}) = \mathrm{Aut}(\overline{Y}'/\overline{X})$.

Furthermore, if we fix a big and nef $\mathbb{Q}$-divisor $L$ on $X$, then we may assume that the canonical model for $K_{Y'} + a(Y',f'^{*}L)f'^{*}L$ is a morphism.
\end{lemm}

In particular, any twist of $f: Y \dashrightarrow X$ is birational to a twist of $f': Y' \to X$.

\begin{proof} 
We first replace $Y$ by a normal birational model which admits a morphism $f  : Y \to X$.  We then consider its Stein factorization $Y \to \widetilde{Y} \to X$. Note that $\widetilde{Y}$ is the normalization of $X$ in $Y$. We replace $Y$ by $\widetilde{Y}$ so that we may assume $\mathrm{Bir}(\overline{Y}/\overline{X}) = \mathrm{Aut}(\overline{Y}/\overline{X})$. Let $F'/F$ be a finite Galois extension such that all automorphisms in $G = \mathrm{Aut}(\overline{Y}/\overline{X})$ are defined over $F'$. Then $G \rtimes \mathrm{Gal}(F'/F)$ acts on $Y_{F'}$.
We resolve singularities equivariantly (as in \cite[Theorem 0.1]{AW97}) and take the quotient by the Galois group $\mathrm{Gal}(F'/F)$ to obtain a smooth variety $Y'$ satisfying the desired condition on automorphism groups.

We still must prove the last statement.  Let $\pi: Y' \dashrightarrow T$ denote the canonical model for $(Y',a(Y',f'^{*}L)f'^{*}L)$.  Choose the same field extension $F'/F$ as before.  Then the rational map $\pi_{F'}: Y_{F'}' \dashrightarrow T_{F'}$ is equivariant for the group $G \rtimes \mathrm{Gal}(F'/F)$. Thus we may take another equivariant resolution and quotient by the Galois action to ensure that $\pi$ is a morphism.
\end{proof}

We next discuss a lemma encapsulating our application of Hilbert's Irreducibility Theorem.

\begin{lemm} \label{lemm: hilbirrapp}
Let $f: Y \to X$ be a surjective generically finite morphism of geometrically integral normal projective varieties defined over a field $F$ of characteristic $0$.  Suppose that the extension of geometric function fields $\overline{F}(\overline{Y})/\overline{F}(\overline{X})$ is Galois with Galois group $G$. Let $F'/F$ be a finite extension such that $G = \mathrm{Bir}(Y_{F'}/X_{F'})$. There is a thin set of points $Z \subset X(F')$ such that if $x \in X(F') \backslash Z$ then $f^{-1}(x)$ is irreducible and the corresponding extension of residue fields is Galois with Galois group $G$.
\end{lemm}

\begin{proof}
According to Lemma \ref{lemm: birandaut} there is a birational model $f': Y' \to X$ of $f$ such that $\mathrm{Aut}(\overline{Y'}/\overline{X}) = \mathrm{Bir}(\overline{Y'}/\overline{X}) = G$.  Since $Y$ and $Y'$ only differ in a closed set, it suffices to prove the statement for $Y'$.

Then our finite field extension $F'/F$ satisfies that $\mathrm{Aut}(\overline{Y}'/\overline{X}) = \mathrm{Aut}(Y'_{F'}/X_{F'})$.  Note that $f_{F'}: Y'_{F'} \to X_{F'}$ is a Galois covering over an open subset $X_{F'}^{\circ}$ of $X_{F'}$.  Applying the Hilbert Irreducibility Theorem (\cite[Proposition 3.3.1]{Serre}) to this open set and adding on $(X_{F'} \backslash X_{F'}^{\circ})(F')$, we obtain a thin subset $Z' \subset X_{F'}(F')$ satisfying the desired property with respect to $F'$-points. 
\end{proof}

Using Hilbert's Irreducibility Theorem, we show that if we fix a thin map $f: Y \to X$ then there is a thin set which contains all point contributions $f^{\sigma}(Y^{\sigma}(F))$ from twists $f^{\sigma}$ that are breaking thin maps.

\begin{theo}
\label{theo:twists}
Let $X$ be a geometrically uniruled smooth projective variety over a field $F$ of characteristic $0$  and let $L$ be a big and nef $\mathbb Q$-divisor on $X$.  Suppose that $f: Y \to X$ is a dominant generically finite morphism from a normal projective variety $Y$.  As $\sigma$ varies over all $\sigma \in H^1(F, \mathrm{Aut}(\overline{Y}/\overline{X}))$ such that $Y^{\sigma}$ is irreducible,
\begin{equation*}
(a(X, L), b(F, X, L)) \leq (a(Y^{\sigma}, (f^{\sigma})^*L), b(F, Y^\sigma, (f^\sigma)^*L)),
\end{equation*}
and $f^{\sigma}$ is face contracting the set 
\begin{equation*}
 Z= \bigcup_{\sigma} f^\sigma(Y^\sigma (F)) \subset X(F)
 \end{equation*}
is contained in a thin subset of $X(F)$.
\end{theo}

\begin{proof}
We start with several simplifications.  If $X$ is not geometrically integral, then $X(F)$ is empty since $X$ is smooth.  So we may suppose $X$ is geometrically integral. 

Suppose that $Y$ is not geometrically integral.  Then any twist $Y^{\sigma}$ of $Y$ which has a rational point not contained in $\Sing(Y^{\sigma})$ must be reducible.  Thus, the set $Z$ is contained in the thin set $f(\Sing(Y))(F)$.  So from now on we assume that $Y$ is geometrically integral.

If $f: Y \to X$ induces an extension of geometric function fields $\overline{F}(\overline{Y})/\overline{F}(\overline{X})$ that is not Galois, then we may conclude by \cite[Proposition 8.2]{LTDuke}. (Although \cite[Proposition 8.2]{LTDuke} is stated over a number field the proof applies to an arbitrary field of characteristic $0$.)  So we may assume that $\overline{F}(\overline{Y})/\overline{F}(\overline{X})$ is Galois with Galois group $G$.  

Suppose that $Y$ is not smooth.  Choose a birational model $f': Y' \to X$ as in Lemma \ref{lemm: birandaut}.  Note that the statement for $f'$ implies the statement for $f$.  Indeed, if $B$ denotes the closed subset of $Y$ where the rational map $\phi: Y \dashrightarrow Y'$ is not defined then 
\begin{equation*}
\bigcup_{\sigma} f^\sigma(Y^\sigma (F)) \subset \bigcup_{\tau} f'^\tau(Y'^\tau(F)) \cup f(B)(F)
 \end{equation*}
where $\tau$ varies over all twists of $f': Y' \to X$ as in the statement of the theorem.  So from now on we assume that $Y$ is smooth.  Similarly, if $Y$ is smooth but $G$ does not coincide with $\mathrm{Aut}(\overline{Y}/\overline{X})$, then we may apply the same construction to reduce to the case when $G =\mathrm{Aut}(\overline{Y}/\overline{X}) = \mathrm{Bir}(\overline{Y}/\overline{X})$.

Since $f$ is dominant Lemma \ref{lemm: genfinite} shows that the only case we need to consider is when $a(Y,f^{*}L) = a(X,L)$. 
Suppose that $F_1/F$ is a finite extension so that $N^1(\overline{Y}) = N^1(Y_{F_1})$ and $G =  \mathrm{Aut}(Y_{F_1}/X_{F_1})$.
By Lemma \ref{lemm: hilbirrapp} we obtain a thin set $Z'' \subset X(F_1)$ such that for any point $x \in X(F_1)\backslash Z''$ the fiber $f^{-1}(x)$ is irreducible over $F_1$ and the corresponding extension of residue fields is Galois with Galois group $G$.  We let $Z'= Z'' \cap X(F)$ which is contained in a thin set by \cite[Proposition 3.2.1]{Serre}.

We prove that if a twist $\sigma$ satisfies $f^{\sigma}(Y^{\sigma}(F)) \not \subset Z'$ then $b(F,Y^{\sigma},f^{\sigma *}L) \leq b(F,X,L)$ and if equality of $b$-invariants is achieved then $f^{\sigma}$ is not face contracting.

We first claim that $N^1(\overline{Y})^G$ is spanned by $N^1(\overline{X})$ and $\overline{f}$-exceptional divisors. To see this, denote the Stein factorization of $\overline{f}: \overline{Y} \to \overline{X}$ by $\overline{g}: \overline{Y} \to \overline{W}$ and $\overline{h}: \overline{W} \to \overline{X}$.  Let $D_{\overline{Y}}$ be a $G$-invariant divisor on $\overline{Y}$ and let $D_{\overline{X}} = \overline{f}_{*}D_{\overline{Y}}$.  We have $\overline{g}_{*}D_{\overline{Y}} = \frac{1}{|G|}\overline{h}^{*}D_{\overline{X}}$ as $\mathbb{Q}$-Weil divisors on $\overline{W}$.  Thus $D_{\overline{Y}} - \frac{1}{|G|}\overline{g}^{*}\overline{h}^{*}D_{\overline{X}}$ is exceptional for $\overline{g}$, and hence for $\overline{f}$.  This proves the claim.

 Let $\mathcal{F}^{\overline{X},\overline{L}}$ be the minimal face of $\Eff^1(\overline{X})$ containing $a(X, L)L + K_X$ and $\mathcal{F}^{\overline{Y},\overline{f}^{*}\overline{L}}$ be the minimal face of $\Eff^1(\overline{Y})$ containing $a(X, L)f^*L + K_Y$.  Since $\mathcal{F}^{\overline{Y},\overline{f}^{*}\overline{L}}$ contains all $\overline{f}$-exceptional effective divisors, we conclude that the natural map
\begin{align}
\label{surjection}
N^1(\overline{X})/\Span ( \mathcal{F}^{\overline{X},\overline{L}} ) \rightarrow N^1(\overline{Y})^G/\Span ( \mathcal{F}^{\overline{Y},\overline{f}^{*}\overline{L}})^G
\end{align}
is surjective.

Now suppose that $x \in X(F) \backslash Z'$ and that there is a point $y \in Y^{\sigma}(F)$ with $f^{\sigma}(y) = x$.  Then the twist $Y^{\sigma}$ corresponds to an element $\sigma \in H^{1}(F,\mathrm{Aut}(\overline{Y}/\overline{X}))$, i.e.~an equivalence class of $1$-cocycles $\Gal(\overline{F}/F) \to G$. By abuse of notation, we will also use $\sigma$ to denote a $1$-cocycle representing this equivalence class.  Note that the fiber $f^{-1}(x)$ is irreducible over $F_1$ and the corresponding extension of residue fields is Galois with Galois group $G$. Since the $G$-torsor $(f^{\sigma})^{-1}(x) \to x$ is also trivial we obtain a surjection $\sigma: \Gal(\overline{F}/F_1) \twoheadrightarrow G$.
We claim that the image of $\mathrm{Gal}(\overline{F}/F)$ in $\mathrm{GL}(N^{1}(\overline{Y}^\sigma))$ contains the image of $G$ in $\mathrm{GL}(N^{1}(\overline{Y}^\sigma))$.  Indeed, choose a finite extension $F'/F_{1}$ such that
\[
Y^{\sigma} \otimes_{F} F' \cong Y \otimes_F F'.
\]
Suppose that $\sigma$ is represented by the $1$-cocycle $\mathrm{Gal}(F'/F) \ni s \mapsto \sigma_s \in G = \mathrm{Aut}(Y_{F'}/X_{F'})$.  Then $Y^{\sigma}$ is the quotient of $Y \otimes_{F} F'$ by $\mathrm{Gal}(F'/F)$ where the Galois action on $Y \otimes_F F'$ is given by the composition
\[
\sigma_s \circ (\mathrm{id}_Y \otimes s).
\]
By our construction, $\sigma$ induces a surjective homomorphism from $\mathrm{Gal}(F'/F_1)$ to $G$ and $\mathrm{Gal}(F'/F_1)$ acts trivially on $N^1(Y_{F'})$ via $\mathrm{id}_Y \otimes s$.  Thus the image of $\mathrm{Gal}(\overline{F}/F)$ in $\mathrm{GL}(N^{1}(\overline{Y^{\sigma}}))$ contains the image of $G$.
Since the action of $\Gal(\overline{F}/F)$ on the N\'eron-Severi space factors through a finite group, by taking the Galois invariant part of (\ref{surjection}) we obtain a surjection
\[
N^1(X)/\Span ( \mathcal{F}^{X,L} ) \rightarrow (N^1(\overline{Y}^\sigma)^G/\Span ( \mathcal{F}^{\overline{Y}^\sigma,\overline{f}^{\sigma *}\overline{L}})^G)^{\Gal(\overline{F}/F)}.
\]
Since the image of $\mathrm{Gal}(\overline{F}/F)$ in $\mathrm{GL}(N^{1}(\overline{Y^{\sigma}}))$ contains the image of $G$ in $\mathrm{GL}(N^{1}(\overline{Y^{\sigma}}))$, we have
\begin{align*}
(N^1(\overline{Y}^\sigma)^G/\Span (\mathcal{F}^{\overline{Y}^\sigma,\overline{f}^{\sigma *}\overline{L}} )^G)^{\Gal(\overline{F}/F)} & = (N^1(\overline{Y}^\sigma)/\Span ( \mathcal{F}^{\overline{Y}^\sigma,\overline{f}^{\sigma *}\overline{L}} ))^{\Gal(\overline{F}/F)}\\
& = N^1(Y^\sigma)/\Span ( \mathcal{F}^{Y^\sigma,f^{\sigma *}L}).
\end{align*}
Thus we conclude that $b(F, Y^\sigma, f^{\sigma *}L) \leq b(F, X, L)$.  If the equality is achieved, then we have that $\Span (\mathcal{F}_{Y^{\sigma},f^{\sigma*}L})$ maps isomorphically to $\Span (\mathcal{F}_{X,L})$ and the cover is not face contracting.
\end{proof}

\begin{proof}[Proof of Theorem~\ref{theo: mainthinness}:]
Theorem~\ref{theo: mainthinness} is the special case of Theorem 6.3 when $Y$ is smooth and $f^\sigma : Y^\sigma \to X$ has invariants which satisfy the strict inequality so that it is face contracting.
\end{proof}

In the proof of Theorem \ref{theo: precisetheorem} we will often need to replace a thin map $f: Y \to X$ by a map $f': Y' \to X$ constructed using an Iitaka base change.  We must show that if $f$ is a breaking thin map then certain twists of $f'$ are also breaking thin maps -- this will ensure that we do not ``lose'' any rational point contributions when performing this replacement.

\begin{lemm} \label{lemm: iitakabasechangeandbvalue}
Let $Y$ be a geometrically uniruled smooth projective variety over a field $F$ of characteristic $0$ and let $L$ be a big and nef $\mathbb{Q}$-divisor on $Y$.  Suppose that $K_{Y}+a(Y,L)L$ has positive Iitaka dimension and let $\pi: Y \dashrightarrow B$ 
denote the canonical model.  Suppose that $h: T \to B$ is any dominant generically finite map from a projective variety $T$.  
Let $g: Y' \to Y$ be the Iitaka base change of $Y$ corresponding to $h$.
Then for every twist $g^{\sigma}: Y'^{\sigma} \to Y$ of $g$ with $Y'^{\sigma}$ irreducible, the induced map $g^{\sigma}_{*}: \mathcal{F}_{Y'^{\sigma},g^{\sigma *}L} \to  \mathcal{F}_{Y,L}$ is surjective. In particular we have $b(F,Y'^{\sigma},g^{\sigma*}L) \geq b(F,Y,L)$.
\end{lemm}

\begin{proof}
Note that in this situation we have $a(Y'^{\sigma},g^{\sigma*}L) = a(Y,L)$, 
since the inequality $\leq$ follows from Lemma \ref{lemm: genfinite} and the inequality $\geq$ follows from Lemmas \ref{lemm:ainvandcanonicalfibers} and \ref{lemm:ainvdominantfamily}.  Let $\phi: \widetilde{Y} \to Y$ be a birational morphism resolving the canonical model map $\pi$ such that $\widetilde{Y}$ is smooth and set $\widetilde{\pi} = \pi \circ \phi$. By applying Lemma \ref{lemm: birandaut} to $Y' \dashrightarrow \widetilde{Y}$, we obtain a smooth birational model $\widetilde{Y}'$ of $Y'$ and a morphism $\widetilde{g}: \widetilde{Y}' \to \widetilde{Y}$ such that $\mathrm{Bir}(\overline{\widetilde{Y}'}/\overline{\widetilde{Y}}) = \mathrm{Aut}(\overline{\widetilde{Y}'}/\overline{\widetilde{Y}})$ and 
the canonical map for $(\widetilde{Y}'^{\sigma}, a(\widetilde{Y}'^{\sigma},\widetilde{g}^{\sigma*}\phi^{*}L)\widetilde{g}^{\sigma*}\phi^{*}L)$ is a morphism.  We claim that the desired statement for $g: Y' \to Y$ follows from the corresponding statement for $\widetilde{g}: \widetilde{Y}' \to \widetilde{Y}$.  Indeed, every twist $g^{\sigma}$ of $g$ is birational to some twist $\widetilde{g}^{\tau}$ of $\widetilde{g}$, and by Lemma \ref{lemm:birfaceinv} we can deduce the desired surjection for $g^{\sigma}_{*}$ from the corresponding surjection for $\widetilde{g}^{\tau}_{*}$.

Let $T^\sigma \to B$ be the Stein factorization of $\widetilde{Y}'^\sigma \to \widetilde{Y} \to B$.  Since $\pi': \widetilde{Y} \to B$ is the canonical model for $(\widetilde{Y},a(\widetilde{Y},\phi^{*}L)\phi^{*}L)$, there is an ample divisor $H$ on $B$ such that $K_{\widetilde{Y}} + a(\widetilde{Y},\phi^{*}L)\phi^{*}L - \widetilde{\pi}^{*}H$ is $\mathbb{Q}$-linearly equivalent to an effective divisor.  Since $T^{\sigma} \to B$ is finite, the pullback of $H$ to $T^{\sigma}$ is still ample, and by the ramification formula we have that $K_{\widetilde{Y}'^{\sigma}} + a(\widetilde{Y}'^{\sigma},\widetilde{g}^{\sigma*}\phi^{*}L)\widetilde{g}^{\sigma*}\phi^{*}L - \widetilde{g}^{\sigma*}\widetilde{\pi}^{*}H$ is $\mathbb{Q}$-linearly equivalent to an effective divisor.
Thus Lemma \ref{lemm:birationaltocanonical} shows that the map $\widetilde{Y}'^{\sigma} \to T^{\sigma}$ agrees with the canonical map for $(\widetilde{Y}'^{\sigma}, a(\widetilde{Y}'^{\sigma},\widetilde{g}^{\sigma*}\phi^*L)\widetilde{g}^{\sigma*}\phi^*L)$ over an open subset of $T^{\sigma}$.  
For a general closed point in $t \in T^\sigma$ with $b = g^\sigma(t)$, consider the commuting diagram
\begin{equation*}
\xymatrix{
\mathcal{F}_{\widetilde{Y}'^\sigma_t, \widetilde{g}^{\sigma*}\phi^*L|_{\widetilde{Y}'^\sigma_t}} \ar@{>}[r] \ar@{>}[d]& \mathcal{F}_{\widetilde{Y}_b, \phi^*L|_{\widetilde{Y}_b}} \ar@{>}[d]\\
\mathcal{F}_{\widetilde{Y}'^\sigma, \widetilde{g}^{\sigma*}\phi^*L} \ar@{>}[r] & \mathcal{F}_{\widetilde{Y}, \phi^*L} 
}.
\end{equation*}
Lemma~\ref{lemm: monodromyandbvalue} (3) shows that the two vertical arrows are surjections.  We also know that $\mathcal{F}_{\widetilde{Y}'^\sigma_t, \widetilde{g}^{\sigma*}\phi^*L|_{\widetilde{Y}'^\sigma_t}} \to \mathcal{F}_{\widetilde{Y}_b, \phi^*L|_{\widetilde{Y}_b}}$ is an isomorphism by Lemma~\ref{lemm:birfaceinv}.  We deduce that the bottom arrow is also a surjection, proving the desired statement.
\end{proof}

\section{The boundedness of breaking thin maps} \label{sect: boundedness}
Our next goal is to prove a boundedness statement for the set of breaking thin maps.  
In this section we work over an algebraically closed field $F$ of characteristic $0$.

\begin{defi} \label{defi:goodfamily}
A good family of adjoint rigid varieties is a morphism $p: \mathcal U \to W$ of smooth quasi-projective varieties and a relatively big and nef $\mathbb{Q}$-divisor $L$ on $\mathcal U$ satisfying the following properties:
\begin{enumerate}
\item The map $p$ is projective, surjective, and smooth with irreducible fibers.
\item The $a$-invariant $a(\mathcal{U}_{w},L|_{\mathcal{U}_{w}})$ is constant and positive for the fibers $\mathcal{U}_{w}$ over closed points and $(\mathcal{U}_{w},L|_{\mathcal{U}_{w}})$ is adjoint rigid for each fiber.
\item The $b$-value $b(F,\mathcal{U}_{w},L|_{\mathcal{U}_{w}})$ is constant for the fibers $\mathcal{U}_{w}$ over closed points.
\item Let $Q$ denote the union of all divisors $D$ in fibers $\mathcal{U}_{w}$ such that $a(D,L|_{D}) > a(\mathcal{U}_{w},L|_{\mathcal{U}_{w}})$. Then $Q$ is closed and flat over $W$.  Furthermore, if we set $\mathcal{V} := \mathcal U \backslash Q$, there is a projective birational map $\phi: \mathcal{U}' \to \mathcal{U}$ that is an isomorphism over $\mathcal{V}$ such that $\mathcal{U}'$ is smooth over $W$ and $\mathcal{U}' \backslash \mathcal{V}$ is a strict normal crossings divisor relative to $W$.
\end{enumerate}
Note that the restriction of $L$ to every fiber of $p$ is nef and since by assumption the $a$-invariant is constant on the fibers the restriction of $L$ to every fiber of $p$ is also big.  A base change of a good family is defined to be the good family induced via base change by a map $g: T \to W$.  We say that $p$ has a good section if there is a section $W \to \mathcal U \backslash Q$, i.e.~if there is a section avoiding $Q$.

A good morphism of good families is a diagram
\begin{equation*}
\xymatrix{\mathcal Y \ar[r]^{f} \ar[d]_{q}&  \mathcal U \ar[d]^{p} \\
T \ar[r]_{g} & W}
\end{equation*}
and a relatively big and nef $\mathbb{Q}$-divisor $L$ on $\mathcal U$ such that $p$ and $q$ are good families of adjoint rigid varieties (with respect to $L$ and $f^{*}L$ respectively), the relative dimensions of $p$ and $q$ are the same, and $a(\mathcal{Y}_{t},f^{*}L|_{\mathcal{Y}_{t}}) = a(\mathcal{U}_{g(t)},L|_{\mathcal{U}_{g(t)}})$ for any point $t \in T$.  
\end{defi}

\begin{lemm} \label{lemm:avaluesandfibers}
Let $p: X \to W$ be a surjective morphism of projective varieties and let $L$ be a big and nef $\mathbb{Q}$-Cartier divisor on $X$.  Suppose that an irreducible component $Y$ of a general fiber of $p$ satisfies $a(Y,L) > 0$.  Fix an ample divisor $H$ on $W$.  Then there is an open subset $W^{\circ} \subset W$ and a positive integer $m$ such that
\begin{equation*}
a(X,L+mp^{*}H) = a(Y,L|_{Y})
\end{equation*}
for any irreducible component $Y$ of a fiber $X_{w}$ over $W^{\circ}$.
\end{lemm}

\begin{proof}
It suffices to prove the statement after replacing $X$ by a smooth birational model.  By applying Theorem \ref{theo: aconstancy} to the Stein factorization of $p$ 
we see there is an open subset $W^{\circ} \subset W$ over which $p$ is a smooth morphism and the $a$-invariant of the components of the fibers with respect to $L$ is constant.  For any positive integer $m$ Lemma \ref{lemm:ainvdominantfamily} shows that a component $Y$ of a general fiber $X_{w}$ satisfies
\begin{equation*}
a(X,L+mp^{*}H) \geq a(Y,(L + mp^{*}H)|_{Y}) = a(Y,L|_{Y}).
\end{equation*}
To show the reverse inequality, it suffices to consider the case when $a(Y,L|_{Y}) < a(X,L)$.  We first prove that $K_{X} + a(Y,L|_{Y})L + \ell p^{*}H$ is pseudo-effective for some sufficiently large $\ell$. By Lemma \ref{lemm: terminalpair} we can write $a(Y,L|_{Y}) L\sim_{\mathbb Q} \Delta + A$ where $\Delta$ is an effective $\mathbb{Q}$-divisor, $A$ is an ample $\mathbb{Q}$-divisor, and $(X,\Delta+A)$ is a terminal pair.  Lemma \ref{lemm:conetheorem} shows there is a finite set of numerical classes $\{ \alpha_{i} \}_{i=1}^{r}$ which generate the extremal rays of $\Eff_{1}(X)_{K_{X} + \Delta \geq 0} + \Nef_{1}(X)$ which have negative intersection against $K_{X} + a(Y,L|_{Y})L$.

The following argument of \cite{Peternell} shows that none of these classes $\alpha_{i}$ satisfies $p_{*}\alpha_{i} = 0$.  Fix an ample divisor $A'$ on $X$ and fix $\epsilon > 0$.  Since $K_{X} + a(Y,L|_{Y})L + \epsilon A'$ is $p$-relatively big, there is some positive $c_{\epsilon}$ such that $K_{X} + a(Y,L|_{Y})L + \epsilon A' + c_{\epsilon}p^{*}H$ is big.  In particular, for every $i$
\begin{equation*}
(K_{X} + a(Y,L|_{Y})L + \epsilon A' + c_{\epsilon}p^{*}H) \cdot \alpha_{i} > 0.
\end{equation*}
If $p_{*}\alpha_{i} = 0$ then we obtain $0 < (K_{X} + a(Y,L|_{Y})L + \epsilon A') \cdot \alpha_{i}$ for every $\epsilon > 0$, a contradiction.

Since $p_{*}\alpha \neq 0$ we may choose $\ell$ sufficiently large so that
\begin{equation*}
\ell p^{*}H \cdot \alpha_{i} > -(K_{X} + a(Y,L|_{Y})L) \cdot \alpha_{i}
\end{equation*}
for every $i$.  By construction $K_{X} + a(Y,L|_{Y})L + \ell p^{*}H$ has positive intersection against every $\alpha_{i}$.  It is also clear that $K_{X} + a(Y,L|_{Y})L + \ell p^{*}H$ has non-negative intersection against any $\beta \in \Eff_{1}(X)_{K_{X} + a(Y,L|_{Y})L \geq 0}$. 
Together these show that $K_{X} + a(Y,L|_{Y})L + \ell p^{*}H$ is pseudo-effective. Then for any $m \geq \ell/a(Y,L|_{Y})$, we have
\begin{equation*}
a(X,L+mp^{*}H) \leq a(Y,L|_{Y})
\end{equation*}
proving the reverse inequality.
\end{proof}

\begin{lemm} \label{lemm: opensetgoodfamily}
Suppose $p: X \to W$ is a surjective projective morphism of varieties such that $X$ is smooth and $p$ has connected fibers.  Let $L$ be a $p$-relatively big $\mathbb{Q}$-Cartier divisor which is the restriction to $X$ of a nef $\mathbb{Q}$-Cartier divisor on some projective compactification of $X$.  Suppose furthermore that the general fiber $X_{w}$ of $p$ is adjoint rigid with respect to $L$. Then there is a non-empty open subset $W^{\circ} \subset W$ with preimage $X^{\circ} = p^{-1}(W^{\circ})$ such that $p: X^{\circ} \to W^{\circ}$ is a good family of adjoint rigid varieties.
\end{lemm}

\begin{proof}
Let $W^{\circ}$ denote a smooth open subset of $W$ over which $p$ is smooth. We construct the family by repeatedly shrinking $W^{\circ}$ (and thus also shrinking its $p$-preimage $X^{\circ}$).  After shrinking $W^{\circ}$, by Theorem~\ref{theo: aconstancy} we may ensure that the $a$-invariant is constant and positive and that all fibers over $W^{\circ}$ are adjoint rigid with respect to $L$.   By \cite[Theorem 1.2]{Sengupta17} after shrinking $W^{\circ}$ again we may ensure that the $b$-invariant is constant.

Let $Q_{+}$ denote the union of all projective subvarieties $Y$ which are contained in some fiber $X_{w}$ over $W^{\circ}$ and which satisfy $a(Y,L) > a(X_{w},L)$.  We claim that after shrinking $W^{\circ}$ the set $Q_{+}$ is closed.  To verify this, let $p': X' \to W'$ denote a projective morphism of projective varieties such that $X'$ and $W'$ contain $X$ and $W$ as open subsets and $p'$ restricts to $p$ on $X$.  After replacing $X'$ by a birational model we may assume that $X'$ is smooth and there is a nef divisor $L'$ on $X'$ that restricts to $L$ on $X$.  Since $L'$ is $p'$-relatively big, there is an ample divisor $A'$ on $W'$ such that $L' + p'^{*}A'$ is big and nef.  Since the restrictions of $L'$ and $L'+p'^{*}A'$ to the fibers $X_{w}$ are the same divisor, it suffices to prove the theorem with $L' + p'^{*}A'$ in place of $L'$.  Thus from now on we assume that $L'$ is a big and nef $\mathbb{Q}$-Cartier divisor.

Fix a very ample divisor $H'$ on $W'$.  After possibly shrinking $W^{\circ}$ Lemma \ref{lemm:avaluesandfibers} shows that we may find a positive integer $m$ such that
\begin{equation*}
a(X_{w},L) = a(X',L'+mp^{*}H')
\end{equation*}
for every fiber $X_{w}$ over $W^{\circ}$.

By Theorem \ref{theo: HJ} the union of all subvarieties $Y$ of $X'$ satisfying $a(Y,L'+mp'^{*}H') > a(X',L' + mp'^{*}H')$ is closed.  We let $Q_{*}$ denote this closed set.  Note that $Q_{+} \subset Q_{*}$.  We claim that if we increase $m$ and shrink $W^{\circ}$ further then $Q_{*} \cap X^{\circ}$ coincides with $Q_{+}$.  First note that by Noetherian induction $Q_{*}$ must eventually stabilize as $m$ increases.  Thus after increasing $m$ we may suppose that $Q_{*}$ does not change upon further increasing the coefficient of $H'$.  After shrinking $W^{\circ}$, we may suppose that each component of $Q_{*}$ that intersects $X^{\circ}$ dominates $W'$.  By applying Lemma \ref{lemm:avaluesandfibers} to the finitely many components of $Q_{*}$ that surject onto $W'$ and are not contained in $\mathbf{B}_{+}(L')$, we see that if we increase $m$ and shrink $W^{\circ}$ further then $Q_{*} \cap X^{\circ} = Q_{+}$ as claimed.  In particular, this implies that $Q_{+}$ is a closed set.

Set $Q$ to be the codimension $1$ components of $Q_{+}$.  By shrinking $W^{\circ}$ we may ensure that $p: Q \to W^{\circ}$ is flat. After applying a resolution of singularities and shrinking $W^\circ$ further we may guarantee that the condition (4) of Definition~\ref{defi:goodfamily} is true.
Note that this set now coincides with $Q$ as defined in Definition \ref{defi:goodfamily} and satisfies all the necessary properties.
\end{proof}

Suppose that $p: \mathcal U \to W$ is a good family of adjoint rigid varieties.  Let $\mathcal V \subset \mathcal U$ be the complement of the set $Q$ as in Definition \ref{defi:goodfamily}.  Suppose that $p$ admits a good section $\zeta$.  
The hypotheses of \cite[Expos\'e XIII Exemples 4.4]{Gro} are verified by the existence of the good section along with Definition \ref{defi:goodfamily} (4), showing that 
\begin{equation*}
\pi_{1}^\et(\mathcal V,\zeta(w)) = \pi_{1}^\et(\mathcal V_{w},\zeta(w)) \rtimes \zeta_* \pi_{1}^\et(W,w)
\end{equation*}
for every fiber $\mathcal{V}_{w} = \mathcal{V} \cap \mathcal{U}_{w}$ over a closed point $w$.

By \cite[Th\'eor\`eme II.2.3.1]{Gro2}, the \'etale fundamental group of a smooth algebraic variety defined over $F$ is topologically finitely presented.  Thus there are only finitely many open subgroups of a given finite index.  Hence, for an open finite index subgroup $\Xi \subset \pi_{1}^\et(\mathcal{V}_{w},\zeta(w))$ its normalizer $N \subset \pi_{1}^\et(W,w)$ has finite index and the subgroup $\Xi \rtimes \zeta_*N$ has finite index in $\pi_{1}^\et(\mathcal{V},\zeta(w))$.

\begin{lemm} \label{lemm: finitelymanycoversinduction}
Let $p: \mathcal U \to W$ be a good family of adjoint rigid varieties with a good section $\zeta$.  Suppose furthermore that the divisor $L$ on $\mathcal{U}$ is the restriction of a nef $\mathbb{Q}$-Cartier divisor on a projective compactification of $\mathcal{U}$. There is a finite set of dominant generically finite good morphisms of good families $\{ f_{i}: \mathcal Y_{i} \to \mathcal U\}$ with structure maps $q_{i}: \mathcal Y_{i} \to T_{i}$ and a closed proper subset $D \subsetneq W$ such that the following holds.  Suppose that $q: \mathcal Y \to T$ is a good family of adjoint rigid varieties admitting a good morphism $f: \mathcal Y \to \mathcal U$.  Then either $f(\mathcal{Y})$ is contained in $p^{-1}D$, or there is a base change $\widetilde{q}: \widetilde{\mathcal{Y}} \to \widetilde{T}$ of $q$ by a generically finite surjective morphism $\widetilde{T} \to T$ such that the induced $\widetilde{f}: \widetilde{\mathcal{Y}} \to \mathcal{U}$ factors rationally through the map $f_{j}$ for some $j$.  

In the latter case there is an open subset $\widetilde{T}^{\circ} \subset \widetilde{T}$ such that every fiber of $\widetilde{\mathcal{Y}} \to \widetilde{T}$ lying over $\widetilde{T}^{\circ}$ is mapped birationally under the map $\widetilde{\mathcal{Y}} \dashrightarrow \mathcal{Y}_{j}$ to a fiber of $q_{j}$. 
\end{lemm}

\begin{proof}
Let $\mathcal V$ denote the open subset of $\mathcal U$ given by removing the set $Q$ as in Definition \ref{defi:goodfamily}.  By \cite[Expos\'e XIII Exemples 4.4]{Gro}, we know that $\pi_{1}^\et(\mathcal V,\zeta(w))= \pi_{1}^\et(\mathcal V_{w},\zeta(w)) \rtimes \zeta_{*}\pi_{1}^\et(W,w)$ where $w$ is a closed point on $W$.  Since $\mathcal{U}_{w}$ is adjoint rigid with respect to the divisor $L$, \cite[Corollary 2.20]{Sen17} shows that for any generically finite cover of $\mathcal{U}_{w}$ which has the same $a$-value and which is adjoint rigid with respect to the pullback of $L$ the  divisorial components of the branch locus are supported on the set $Q \cap \mathcal{U}_{w}$.  Furthermore, by \cite[Proposition 2.17]{Sen17} there is an upper bound on the degree of such covers depending only on $\dim(\mathcal{U}_{w})$, $a(\mathcal{U}_{w},L)$ and $L|_{\mathcal{U}_{w}}^{\dim (\mathcal{U}_{w})}$. 
Altogether there is a finite set of finite index subgroups $\Xi_{i} \subset \pi_{1}(\mathcal{V}_{w},\zeta(w))$ such that for some fiber of $p$ the corresponding \'etale cover has a projective closure which has the same $a$-value as $\mathcal{U}_{w}$ and is adjoint rigid. 
For each such $\Xi_{i}$ we use $N_{i}$ to denote its normalizer in $\pi_{1}^\et(W,w)$; we also denote $\Upsilon_{i} = \Xi_{i} \rtimes \zeta_*N_{i}$.  As remarked earlier $\Upsilon_{i}$ will always be a subgroup of $\pi_{1}^\et(\mathcal{V},\zeta(w))$ of finite index.

Let $\mathcal{E}_{i}$ denote the \'etale cover of $\mathcal{V}$ corresponding to $\Upsilon_{i}$.  Note that $\mathcal{E}_{i}$ admits a morphism to the \'etale cover $R_{i} \to W$ defined by $N_{i}$.  By construction we know that every fiber of the map $\mathcal{E}_{i} \to R_{i}$ is an \'etale cover of the corresponding fiber of $\mathcal V \to W$, and since these covers are induced by subgroups we deduce that every fiber of $\mathcal{E}_{i} \to R_{i}$ is irreducible.  By taking a projective closure of the fibers of $\mathcal{E}_{i}$ over $R_{i}$ and passing to a resolution we obtain a projective family $r_{i}: \widetilde{\mathcal{E}}_{i} \to R_{i}$.  By Theorem~\ref{theo: aconstancy} there is an open set $R_{i}^{\circ} \subset R_{i}$ over which $r_{i}$ has smooth irreducible fibers and such that the $a$-value and Iitaka dimension of the fibers over $R_{i}^{\circ}$ is constant.  We define a closed subset $D$ of $W$ as the union of the Zariski closures of the images of $R_{i} \backslash R_{i}^{\circ}$ for each $i$. We will enlarge $D$ later.  First suppose that the remaining fibers $\widetilde{\mathcal{E}}_{i}^{\circ} \to R_{i}^{\circ}$ are adjoint rigid with respect to the pullback of $L$ and have the same $a$-value as the fibers of $p$.  Then Lemma \ref{lemm: opensetgoodfamily} shows that $r_{i}$ is a good family of adjoint rigid varieties over an open subset $T_{i}$ of $R_{i}^{\circ}$.  (Note that the condition that $L$ be the restriction of a nef divisor from a projective compactification is preserved by arbitrary pullback so the hypotheses of Lemma \ref{lemm: opensetgoodfamily} are satisfied.)  In this case we call this good family $q_{i}: \mathcal{Y}_{i} \to T_{i}$ and further enlarge $D$ by adding the Zariski closure of the image of $R_{i}^{\circ} \backslash T_{i}$.  Otherwise, the remaining fibers either fail to be adjoint rigid with respect to the pullback of $L$ or fail to have the same $a$-value as the fibers of $p$.  Then we simply ignore the family $\widetilde{\mathcal{E}}_{i}^{\circ} \to R_{i}^{\circ}$. 

We have now constructed a finite set $\{ f_{i}: \mathcal Y_{i} \to \mathcal U \}$ of good morphisms of good families and a set $D$.   Set $W^\circ = W \setminus D$ and let $(T_i)^\circ$ and $\mathcal{V}^\circ$ denote the preimages of $W^{\circ}$ in  $T_i$ and $\mathcal V$.  Note that by the construction, the map $g_{i}: (T_i)^\circ \to W^\circ$ is proper \'etale so that for a point $t_{i} \in T_{i}^{\circ}$ we have that $\Xi_{i} \rtimes \zeta_*g_{i*}\pi_1^\et((T_{i})^\circ,t_{i})$
is a finite index subgroup of $\pi_{1}^\et(\mathcal{V}^{\circ},g_{i}(t_{i}))$ such that the corresponding \'etale cover is an open subset of $\mathcal{Y}_{i}$.  

We next show that the set $\{f_{i}: \mathcal{Y}_{i} \to \mathcal{U} \}$ satisfies the factoring property in the statement of the theorem. Suppose we have a morphism $f: \mathcal Y \to \mathcal U$ as in the statement of the theorem.  Let $T'$ be a general intersection of hyperplanes in $\mathcal{Y}$ that maps generically finitely onto $T$ and is not contained in the $f$-preimage of $Q$.  We let $q': \mathcal{Y}' \to T'$ denote the base change over $T' \to T$, and let $f': \mathcal{Y}' \to \mathcal{U}$ denote the induced map.  After shrinking $T'$ we may ensure that $q': \mathcal{Y}' \to T'$ is a good family and the map $T' \to \mathcal{Y}$ induces a good section $\eta'$ of $q'$ whose $f'$-image in $\mathcal{U}$ is disjoint from $Q$.

Let $\mathcal V_{\mathcal{Y}'}$ denote the open subset obtained by removing the closed subset $Q_{\mathcal{Y}'}$ as in Definition \ref{defi:goodfamily}.  Using the good section $\eta'$, we can identify
\begin{equation*}
\pi_{1}^\et(\mathcal V_{\mathcal{Y}'},\eta'(t'))= \pi_{1}^\et(\mathcal V_{\mathcal{Y}',t'},\eta'(t')) \rtimes \eta'_*\pi_{1}^\et(T',t')
\end{equation*}
for every fiber $\mathcal{V}_{\mathcal{Y}',t'}$ of $q'$.
Note however that this semidirect product structure need not be compatible with the semidirect product structure of $\pi_{1}^\et(\mathcal{V},\zeta(w))$ since there is no relationship between the two sections used in the constructions.

 Let $(T')^\circ$ denote the preimage of $W^{\circ}$ in $T'$.   Note that $(T')^{\circ}$ is empty if and only if the image of $\mathcal{Y}'$ is contained in $p^{-1}(D)$.  
If $(T')^{\circ}$ is not empty, for a general point $t'$ in $T'^{\circ}$ we set $w$ as the image of $t'$ in $W^{\circ}$ via $(T')^\circ \to W^\circ$. 
The map $f'|_{f'^{-1}(\mathcal V)_{t'}} : f'^{-1}(\mathcal V)_{t'} \rightarrow \mathcal V_w$ is proper. 
Furthermore, \cite[Corollary 2.20]{Sen17} shows the map $f'|_{f'^{-1}(\mathcal V)_{t'}}$ is the composition of a birational map with an \'etale map, so in particular it is \'etale outside of a codimension $\geq 2$ subset in $\mathcal V_w$.  We let $\Xi_{j}'$ denote the subgroup defined by the image of $\pi_{1}^\et(f'^{-1}(\mathcal{V})_{t'},\eta'(t'))$ in $\pi_{1}^\et(\mathcal{V}_{w},f'\circ \eta' (t'))$.

We claim there is an isomorphism $\Phi: \pi_{1}^\et(\mathcal{V},f'\circ\eta' (t')) \cong \pi_{1}^\et(\mathcal{V}, \zeta (w))$ such that $p_*^{f'\circ\eta' (t')} = p_*^{\zeta(w)} \circ \Phi$ where for any point $v \in \mathcal V$ above $w$, $p_*^v$ is the map $\pi_1^\et(\mathcal V, v) \to \pi_1^\et(W^\circ, w)$ induced by $p$.  Indeed, over $\mathbb{C}$ one can define such an isomorphism of topological fundamental groups $\pi_{1}(\mathcal{V},f'\circ\eta' (t')) \cong \pi_{1}(\mathcal{V}, \zeta (w))$ using a path in the fiber, and the result for \'etale fundamental groups follows from the comparison theorems for \'etale and topological fundamental groups. Over an arbitrary algebraically closed field, one can reduce to the case of $\mathbb{C}$ using \cite[Expos\'e XIII Proposition 4.6]{Gro} which shows that the \'etale fundamental group is unchanged under extensions of algebraically closed fields of characteristic $0$.  We fix such an isomorphism $\Phi$. 

Under the identification $\Phi$, $\Xi'_{j}$ corresponds to one of the subgroups $\Xi_{j}$ constructed earlier.  Furthermore, since the image of $\mathcal{Y}'$ is not contained in $p^{-1}(D)$, the map $f'$ constructs covers of some fibers $\mathcal{U}_{w}$ over $W^{\circ}$ which have the same $a$-value as $\mathcal{U}_{w}$ and which are adjoint rigid with respect to $L$.  This means that $\Xi_j$ defines a good family $q_j : \mathcal Y_j \to T_j$ via the construction at the beginning of the proof.  In particular, there is a point $t_{j} \in T_{j}^{\circ}$ mapping to $w$ under $g_{j}$.

Since the image of $f'$ is not contained in $p^{-1}(D)$, the group $\pi_{1}^\et((T')^\circ,t')$ maps into $\pi_{1}^\et(\mathcal{V}^\circ, f' \circ \eta'(t'))$ by composing $\eta'$ with the map $f'$ (since by construction the $f'$-image of $\eta'$ avoids $Q$).  Consider the finite index subgroup $M \subset \pi_{1}^\et((T')^\circ,t')$ which is the pullback of $\Phi^{-1}(\Xi_{j} \rtimes \zeta_*g_{j*}\pi_1^\et((T_{j})^\circ,t_{j}))$ by this map. 
Let $\widetilde{q}: \widetilde{\mathcal{Y}} \to \widetilde{T}$ be defined as a resolution of the base-change of $q'$ over the cover $\widetilde{T}$ of $(T')^\circ$ defined by $M$ with morphism $\widetilde{f}: \widetilde{\mathcal{Y}} \to \mathcal{U}$. After shrinking $\widetilde{T}$, if necessary, we may ensure that $\widetilde{q} : \widetilde{\mathcal{Y}} \to \widetilde{T}$ is a good family with a good section.  Thus by using the section $\widetilde{\eta}$ induced from $\eta'$ by base change and shrinking $\widetilde{T}$ to ensure that the hypotheses of \cite[Expos\'e XIII Exemples 4.4]{Gro} apply to $\widetilde{f}^{-1}(\mathcal{V})$, we have an identification
\begin{equation*}
\pi_{1}^\et(\widetilde{f}^{-1}(\mathcal{V}),\widetilde{\eta}(\widetilde{t})) = \pi_{1}^\et(\widetilde{f}^{-1}(\mathcal{V})_{\widetilde{t}},\widetilde{\eta}(\widetilde{t})) \rtimes \widetilde{\eta}_*\pi_{1}^\et(\widetilde{T},\widetilde{t}),
\end{equation*}
where $\widetilde{t}$ is any closed point in $\widetilde{T}$. 
Since $t'$ is general, we can find $\widetilde{t}$ mapping to $t'$.   
Then every element in $\widetilde{f}_{*}\pi_{1}^\et(\widetilde{f}^{-1}(\mathcal{V}),\widetilde{\eta}(\widetilde{t}))$ will be a product of an element in $\Xi_{j}' \rtimes \{ 1\} \subset \pi_{1}^\et(\mathcal{V}^\circ, f' \circ \eta'(t') )$ with an element in $\widetilde{f}_{*}\widetilde{\eta}_*\pi_{1}^\et(\widetilde{T},\widetilde{t})$, so by construction this set is contained in $\Phi^{-1}(\Xi_{j} \rtimes \zeta_*g_{j*}\pi_1^\et((T_{j})^\circ,t_{j}))$.  By the lifting property for \'etale fundamental groups as in \cite[Expos\'e V]{Gro} (see \cite{Chen18} for a careful explanation), the map $\widetilde{f}^{-1}(\mathcal{V}) \to \mathcal{V}$ factors through the cover defined by $\Phi^{-1}(\Xi_{j} \rtimes \zeta_*g_{j*}\pi_1^\et((T_{j})^\circ,t_{j}))$. Hence $\widetilde{\mathcal Y} \rightarrow \mathcal U$ rationally factors through $\mathcal Y_j$.  

It only remains to prove the last claim.  Let $\widetilde{\mathcal{Y}}'$ be a resolution of the map $\widetilde{\mathcal{Y}} \dashrightarrow \mathcal{Y}_{j}$.  Note that we have a morphism $\widetilde{T} \to T_{j}^{\circ}$ induced by the homotopy lifting property via the map $\pi_{1}^\et(\widetilde{T},\widetilde{t}) \to \pi_{1}^\et(W^{\circ},w)$ induced by $\Phi$. 
Altogether we have a commutative diagram 
\begin{equation*}
\xymatrix{ \widetilde{\mathcal{Y}}' \ar[r] \ar[d] &  \mathcal{Y}_{j} \ar[d] \\
\widetilde{T} \ar[r] & T_{j}}.
\end{equation*}
By applying Lemma \ref{lemm: opensetgoodfamily} we see there is an open subset $\widetilde{T}^{\circ} \subset \widetilde{T}$ over which $\widetilde{\mathcal{Y}}'$ is a good family.  By construction every fiber over $\widetilde{T}^{\circ}$ is mapped birationally under $\widetilde{\mathcal{Y}}' \to \mathcal{Y}_{j}$ to the corresponding fiber of $q_{j}$.  After shrinking $\widetilde{T}^{\circ}$, we can also ensure that every fiber of $\widetilde{\mathcal{Y}}'$ over $\widetilde{T}^{\circ}$ is birational to the corresponding fiber of $\widetilde{\mathcal{Y}}$. 
\end{proof}

For later use, we note a useful property of the construction of the previous lemma.

\begin{coro}
\label{coro: finitelymanycoversbasechange}
Let $p: \mathcal U \to W$ be a good family of adjoint rigid varieties with a good section.  Suppose furthermore that the divisor $L$ on $\mathcal{U}$ is the restriction of a nef $\mathbb{Q}$-Cartier divisor on a projective compactification of $\mathcal{U}$. Consider the good morphisms of good families $\{ f_{i}: \mathcal Y_{i} \to \mathcal U\}$ with family maps $q_{i}: \mathcal Y_{i} \to T_{i}$ and the closed proper subset $D \subsetneq W$ constructed by Lemma \ref{lemm: finitelymanycoversinduction}.  Suppose that for each $i$ we fix a projective generically finite dominant map $T_{i}' \to T_{i}$ from a smooth variety $T_{i}'$ and replace $q_{i}: \mathcal{Y}_{i} \to T_{i}$ by the base change $\mathcal Y_{i}' := T_{i}' \times_{T_{i}} \mathcal Y_{i}$ (with the natural induced maps $f_{i}'$ and $q_{i}'$).  Construct a closed subset $D' \subsetneq W$ by taking the union of $D$ with the branch locus of each map $T_{i}' \to W$.  Then the good families $\mathcal Y_{i}'$ (equipped with the morphisms $f_{i}'$ and $q_{i}'$) and the closed subset $D' \subsetneq W$ again satisfy the conclusion of Lemma \ref{lemm: finitelymanycoversinduction}. 
\end{coro}

\begin{proof}
Given a good family of adjoint rigid varieties $q: \mathcal{Y} \to T$ admitting a good morphism $f: \mathcal{Y} \to \mathcal{U}$, Lemma \ref{lemm: finitelymanycoversinduction} shows that there is a base change $\widetilde{q}: \widetilde{\mathcal{Y}} \to \widetilde{T}$ such that the induced $\widetilde{f}: \widetilde{\mathcal{Y}} \to \mathcal{U}$ either has image contained in $p^{-1}(D)$ or factors rationally through some $f_{i}$.  In the former case, the image is also contained in $p^{-1}(D')$.  In the latter case, let $T_{i}^{\circ}$ denote the complement of the image of the ramification divisor for the map $T_{i}' \to T_{i}$.  Let $T_{i}'^{\circ}$ and $\widetilde{T}^{\circ}$ denote the preimages of $T_{i}^{\circ}$.  If $\widetilde{T}^{\circ}$ is empty, then the image of $\widetilde{f}$ is contained in $p^{-1}(D')$.  
Otherwise, $T_{i}'^{\circ} \times_{T_{i}^{\circ}} \widetilde{T}^{\circ}$ is an \'etale cover of $\widetilde{T}^{\circ}$.  Let $R$ be any irreducible component of this product which dominates $\widetilde{T}^{\circ}$.  Then we can take a base change of $\widetilde{q}: \widetilde{\mathcal{Y}} \to \widetilde{T}$ over the map $R \to \widetilde{T}$ to obtain the desired rational factoring.
\end{proof}

Returning to the proof of boundedness, note that we can apply the argument of Lemma \ref{lemm: finitelymanycoversinduction} to the preimage of any component of the closed set $D$ constructed there.  Arguing by Noetherian induction, we conclude:

\begin{theo} \label{theo: finitelymanycovers}
Let $p: \mathcal U \to W$ be a good family of adjoint rigid varieties with a good section.  Suppose furthermore that the divisor $L$ on $\mathcal{U}$ is the restriction of a nef $\mathbb{Q}$-Cartier divisor on a projective compactification of $\mathcal{U}$. There is a finite set of generically finite good morphisms of good families $\{ f_{i}: \mathcal Y_{i} \to \mathcal U\}$ with family maps $q_{i}: \mathcal Y_{i} \to T_{i}$ such that the following holds.  Suppose that $q: \mathcal Y \to T$ is a good family of adjoint rigid varieties admitting a good morphism $f: \mathcal Y \to \mathcal U$.  Then there is a base change $\widetilde{q}: \widetilde{\mathcal{Y}} \to \widetilde{T}$ of $q$ such that the induced $\widetilde{f}: \widetilde{\mathcal{Y}} \to \mathcal{U}$ factors rationally through the map $f_{j}$ for some $j$ and a general fiber of $\widetilde{q}$ is birational to a fiber of $q_{j}$.
\end{theo}

As a consequence, we prove a finiteness statement for breaking thin maps.

\begin{theo} \label{theo: thinmapfactoring}
Let $X$ be a uniruled smooth projective variety and let $L$ be a big and nef $\mathbb{Q}$-divisor on $X$.  There is a finite set of thin maps $\{ f_{\ell}: Y_{\ell} \to X\}$ with $a(Y_{\ell},f_{\ell}^{*}L) \geq a(X,L)$ satisfying the following property.  For any thin map $f: Y \to X$ with $a(Y, f^*L) \geq a(X, L)$, after an Iitaka base change to obtain a variety $\widetilde{Y}$ the induced map $\widetilde{f}: \widetilde{Y} \to X$ will either factor rationally through some $f_{k}$ or will have image contained in $\mathbf{B}_{+}(L)$.  Furthermore, in the first case we have
\begin{equation*}
a(Y,f^{*}L) = a(\widetilde{Y}, \widetilde{f}^*L) \leq a(Y_k, f_k^*L)
\end{equation*}
and if equality of $a$-invariants is achieved then
\begin{equation*}
b(F,Y,f^{*}L) \leq b(F,\widetilde{Y},\widetilde{f}^{*}L) \leq b(F,Y_{k},f_{k}^{*}L).
\end{equation*}
\end{theo}

\begin{proof}
Let $p_i : \mathcal U_i \to W_i$ be the families from Theorem~\ref{theo: aconstruction}.  If the image of $s_{i}: \mathcal{U}_{i} \to X$ is contained in $\mathbf{B}_{+}(L)$ then we ignore the corresponding family $p_{i}$ from now on.  Otherwise, after a finite base change we may ensure each family has a rational section.  Since $L$ is a nef divisor on $X$, the hypotheses of Lemma \ref{lemm: opensetgoodfamily} hold and so we can shrink the base to obtain a good family with a good section.  By combining a Noetherian induction argument with repeated appeals to Lemma~\ref{lemm: opensetgoodfamily} (and repeated throwing away of families with image in $\mathbf{B}_{+}(L)$), we can repeat the argument along the complement of the open set constructed above to obtain a finite collection of good families with good sections.   To each such family we apply Theorem \ref{theo: finitelymanycovers}.  The result is a finite collection of good families $\{ q_{i,j}: \mathcal{Y}_{i,j} \to T_{i,j} \}$ with maps $g_{i,j}: \mathcal{Y}_{i,j} \to X$.

We next modify the families $q_{i,j}: \mathcal{Y}_{i,j} \to T_{i,j}$ by performing a couple of base changes over $T_{i,j}$.  For notational simplicity, we will use the same notation $q_{i,j}: \mathcal{Y}_{i,j} \to T_{i,j}$ and $g_{i,j}: \mathcal{Y}_{i,j} \to X$ for the results after base change.  First, by taking a base change we may ensure that $g_{i,j}$ is not birational.  Next, we make a base change to kill the monodromy action of $\pi_{1}^{\et}(T_{i,j},t_{i,j})$ on the N\'eron-Severi group of a general fiber of $q_{i,j}$. Finally, we take a base change over a cyclic cover of $T_{i,j}$ whose branch divisor is a very ample divisor.

We define the thin maps $\{ f_{\ell}: Y_{\ell} \to X \}$ as follows.  For each $\mathcal{Y}_{i,j}$ set $D_{i,j}$ to be the closure of $g_{i,j}(\mathcal{Y}_{i,j})$.  By construction $D_{i,j}$ is not contained in $\mathbf{B}_{+}(L)$.  If $a(D_{i,j},L|_{D_{i,j}})$ agrees with the $a$-value of the fibers of $q_{i,j}$, then \cite[Proposition 4.14]{LTDuke} shows that $g_{i,j}$ is generically finite.  
We differentiate the $\mathcal{Y}_{i,j}$ satisfying this property by calling them ``allowable families.''  If $\mathcal{Y}_{i,j}$ is allowable, we add a variety $Y_{\ell}$ which is a resolution of a projective closure of $\mathcal{Y}_{i,j}$ that comes equipped with a morphism $f_{\ell}: Y_{\ell} \to X$ extending $g_{i,j}$ and a morphism $r_{\ell}: Y_{\ell} \to R_{\ell}$ extending the family map $q_{i,j}$.  Since $g_{i,j}$ is not birational, $f_{\ell}$ is a thin map.  
We also add a finite collection of inclusions of subvarieties to the list of $f_{\ell}: Y_{\ell} \to X$ as follows.  Let $\{a_{c} \}$ denote the finite set of constants constructed by Theorem \ref{theo: HJ} (1).  For each $a_{c}$, Theorem \ref{theo: HJ} (2) constructs a proper closed subset $V^{a_{c}} \subsetneq X$.  For every component $V^{a_{c}}_{d}$ of $V^{a_{c}}$, we add the inclusion $V^{a_{c}}_{d} \hookrightarrow X$ as one of the $f_{\ell}: Y_{\ell} \to X$.  Note that when $\mathcal{Y}_{i,j}$ is not allowable then $D_{i,j}$ will be a subvariety of $V^{a(D_{i,j},L|_{D_{i,j}})}$ and so the map $\mathcal{Y}_{i,j} \to X$ will factor through the map $f_{\ell}$ corresponding to the inclusion of a suitable component of $V^{a(D_{i,j},L|_{D_{i,j}})}$.   

Before proving the factoring property of the $f_{\ell}$, we note one additional property that will be needed later.  Suppose we constructed $Y_{\ell}$ from an allowable family.  Recall that in the construction of $Y_{\ell}$ we took a cyclic cover of the corresponding $T_{i,j}$ branched over a very ample divisor.  This guarantees that there is an ample divisor $H$ on $R_{\ell}$ such that $K_{Y_{\ell}} + a(Y_{\ell},f_{\ell}^{*}L)f_{\ell}^{*}L - r_{\ell}^{*}H$ is $\mathbb{Q}$-linearly equivalent to an effective class.  Thus by Lemma \ref{lemm:birationaltocanonical} $r_{\ell}$ is birationally equivalent to the canonical map for $K_{Y_{\ell}} + a(Y_{\ell},f_{\ell}^{*}L)f_{i}^{*}L$.   

Now suppose $f: Y \to X$ is any thin map satisfying $a(Y,f^{*}L) \geq a(X,L)$ and whose image is not contained in $\mathbf{B}_{+}(L)$. Again we separate into two cases: just as before, we say that $Y$ is allowable if $a(f(Y),L|_{f(Y)}) = a(Y,f^{*}L)$.  First suppose that $Y$ is not allowable.  Then the map $f: Y \to X$ factors through the inclusion of $V^{a(f(Y),L|_{f(Y)})}$ into $X$, and hence also factors through some $f_{\ell}$.  If $Y$ is allowable, then we know that $a(f(Y),L|_{f(Y)}) = a(Y,f^{*}L)$.  After resolving we may suppose $Y$ is smooth and admits a morphism $q: Y \to T$ which is the canonical model for $K_{Y} + a(Y,f^{*}L)f^{*}L$.  By Lemma \ref{lemm: opensetgoodfamily} there is an non-empty open subset $T^{\circ} \subset T$ such that $q$ is a good family over $T^{\circ}$.  By Lemma \ref{lemm:dominantequalitycase} we know that possibly after shrinking $T^{\circ}$ the image of every fiber over $T^{\circ}$ in $X$ will have the same $a$-value as $Y$ does and will be adjoint rigid.  Thus, the map $f: Y \to X$ will yield a map $T^{\circ} \to \Hilb(X)$ whose image is contained in the locus parametrizing subvarieties which are adjoint rigid with respect to the pullback of $L$ and have $a$-value equal to $a(Y,f^{*}L)$.  Since $T$ is reduced, after taking a base change and shrinking $T^{\circ}$ we obtain a map $T^{\circ} \to W_{i}$ for some $i$.  After shrinking $T^{\circ}$ and replacing $Y$ by a birational model yet again, the preimage $Y^{\circ}$ will yield a good map of good families $Y^{\circ} \to \mathcal{U}_{i}$. Following through the construction of the $Y_{\ell}$ above, Theorem \ref{theo: finitelymanycovers} shows that if $\widetilde{Y}$ is the Iitaka base change defined by a suitable cover $\widetilde{T} \to T$, the induced map $\widetilde{Y} \to \mathcal{U}_{i}$ will factor rationally through $\mathcal{Y}_{i,j}$ for some $j$. 
We may also ensure that the cover $\widetilde{T} \to T$ is chosen so that the monodromy action on the N\'eron-Severi spaces of a general fiber of $\widetilde{Y} \to \widetilde{T}$ is trivial.  Finally, since we are only trying to prove the existence of a rational factoring, we are free to replace $\widetilde{Y}$ by a smooth resolution, which by abuse of notation we will continue to call $\widetilde{Y}$. 
Let $\widetilde{f}: \widetilde{Y} \to X$ denote the induced map. If $\mathcal{Y}_{i,j}$ is an allowable family, then $\widetilde{f}$ will factor rationally through the corresponding projective closure $Y_{\ell}$.  When $\mathcal{Y}_{i,j}$ is not allowable, then $\widetilde{f}$  
factors through the inclusion $D_{i,j} \hookrightarrow X$, and thus (as discussed above) also through some $f_{\ell}$.  In sum, in every case there is some index $k$ such that the map $\widetilde{f}$ factors rationally through $f_{k}: Y_{k} \to X$. 

We next prove the inequalities for $a$-values.  If $Y$ is not allowable, then $Y_{k}$ has $a$-value at least $a(f(Y),L|_{f(Y)}) > a(Y,L)$.  Otherwise, recall that by Lemma \ref{lemm:ainvandcanonicalfibers} the $a$-value of a pair is the same as the $a$-value of a general fiber of the canonical map.  Since we constructed $\widetilde{Y}$ by a base change over the canonical map, Lemma \ref{lemm:birationaltocanonical} shows that the map $\widetilde{Y} \to \widetilde{T}$ is birational to the canonical map for $K_{\widetilde{Y}} + a(\widetilde{Y},\widetilde{f}^{*}L)\widetilde{f}^{*}L$. Thus the general fibers of the canonical maps for $Y$ and $\widetilde{Y}$ are birational, so we have
\begin{equation*}
a(Y,f^{*}L) = a(\widetilde{Y},\widetilde{f}^{*}L).
\end{equation*}
Next, note that a general fiber of the canonical model for $(\widetilde{Y},a(\widetilde{Y},\widetilde{f}^{*}L)\widetilde{f}^{*}L)$ maps birationally onto a fiber $\mathcal{Y}_{i,j,t_{i,j}}$ of $q_{i,j}: \mathcal{Y}_{i,j} \to T_{i,j}$, and in particular
\begin{equation*}
a(\widetilde{Y},\widetilde{f}^{*}L) = a(\mathcal{Y}_{i,j,t_{i,j}},g_{i,j}^{*}L|_{\mathcal{Y}_{i,j,t_{i,j}}}).
\end{equation*}
Suppose first that $\mathcal{Y}_{i,j}$ is an allowable family. Since the $a$-invariant is constant for the fibers of $q_{i,j}$ (as it is a good family), $Y_{k}$ is dominated by subvarieties with the same $a$-value as $\widetilde{Y}$.  By Lemma \ref{lemm:ainvdominantfamily}  $a(\widetilde{Y}, \widetilde{f}^*L) \leq a(Y_k, f_k^*L)$.  If $\mathcal{Y}_{i,j}$ is not allowable, then as explained before the map $f$ factors through the inclusion $D_{i,j} \hookrightarrow X$ where $D_{i,j}$ has higher $a$-value than the members of the family $\mathcal{Y}_{i,j}$, and thus also higher $a$-value than $\widetilde{Y}$.  Thus in either case
\begin{equation*}
a(\widetilde{Y},\widetilde{f}^{*}L) \leq a(Y_{k},f_{k}^{*}L).
\end{equation*}

Finally, we prove the inequalities for $b$-values.  
Equality of $a$-values will only occur when $Y$ is allowable and the corresponding $\mathcal{Y}_{i,j}$ is an allowable family.  Let $Y_{t}$ denote a general fiber of $q$.  Then Lemma \ref{lemm: monodromyandbvalue} (3) shows that
\begin{equation*}
b(F,Y,f^{*}L) \leq b(F,Y_{t},f^{*}L|_{Y_{t}}).
\end{equation*}
Choose a point $\widetilde{t} \in \widetilde{T}$ mapping to $t$, so that $\widetilde{Y}_{\widetilde{t}}$ is birational to $Y_{t}$.  Thus $b(F,Y_{t},f^{*}L|_{Y_{t}}) = b(F,\widetilde{Y}_{\widetilde{t}},\widetilde{f}^{*}L|_{\widetilde{Y}_{\widetilde{t}}})$. Since the monodromy action on the N\'eron-Severi spaces of general fibers over $\widetilde{T}$ is trivial and the map $\widetilde{Y} \to \widetilde{T}$ is birational to the canonical map for $K_{\widetilde{Y}} + a(\widetilde{Y},\widetilde{f}^{*}L)\widetilde{f}^{*}L$, by Corollary \ref{coro: bvalequality} we have
\begin{equation*}
b(F,\widetilde{Y}_{\widetilde{t}},\widetilde{f}^{*}L|_{\widetilde{Y}_{\widetilde{t}}}) = b(F,\widetilde{Y},\widetilde{f}^{*}L). 
\end{equation*}
Next, note $\widetilde{Y}_{t}$ maps birationally onto a fiber of the map $q_{i,j}$.  By constancy of the $b$-invariant for the fibers of $q_{i,j}$  and the fact that $f_{k}$ extends $g_{i,j}$, we see that
\begin{equation*}
b(F,\widetilde{Y}_{t},\widetilde{f}^{*}L|_{\widetilde{Y}_{t}}) = b(F,Y_{k,r},f_{k}^{*}L|_{Y_{k,r}})
\end{equation*}
for a general fiber $Y_{k,r}$ of $r_{k}: Y_{k} \to R_{k}$. 
Recall that the geometric monodromy action on the N\'eron-Severi space of the fibers of $r_{k}$ over $T_{i,j}$ is trivial.  Since $r_{k}$ is birationally equivalent to the canonical map for $(Y_{k},a(Y_{k},f_{k}^{*}L)f_{k}^{*}L)$, we may apply Corollary \ref{coro: bvalequality} to conclude that
\begin{equation*}
b(F,Y_{k,r},f_{k}^{*}L|_{Y_{k,r}}) = b(F,Y_{k},f_{k}^{*}L).
\end{equation*}
Thus our assertion follows.
\end{proof}

\begin{proof}[Proof of Theorem~\ref{theo: mainfiniteness}:]
Let $\{ f_{i}: Y_{i} \to X \}$ be the set of morphisms from Theorem 7.7.  Suppose $f: Y \to X$ is a generically finite morphism of smooth varieties satisfying $$(a(Y,f^{*}L),b(F, Y,f^{*}L)) > (a(X,L),b(F, X,L))$$ in the lexicographic order.  Theorem 7.7 shows that there is some Iitaka base change $\widetilde{Y}$ of $Y$ such that the induced morphism $\widetilde{f}: \widetilde{Y} \to X$ factors rationally through one $f_{j}: Y_{j} \to X$.  Furthermore, we know that $Y_{j}$  satisfies $(a(Y,f^{*}L),b(F, Y,f^{*}L)) \leq (a(Y_{j},f_{j}^{*}L),b(F, Y_{j},f_{j}^{*}L))$ in the lexicographic order.  Thus $f_{j}: Y_{j} \to X$ is a breaking thin map.  In conclusion, in the setting of Theorem~\ref{theo: mainfiniteness} we may take the subset of the $f_{j}$ which are breaking thin maps.
\end{proof}

\section{Constructing a thin set}
\label{sec: thinset}

In this section we prove analogues of the results in the previous section over a field $F$ of characteristic $0$ and use them to prove the main theorem (Theorem~\ref{theo: precisetheorem}).  We first discuss some constructions involving the \'etale fundamental group.

\begin{defi}
Let $\overline{T}$ be a smooth variety over an algebraically closed field of characteristic $0$ and let $\overline{t}$ be a geometric point on $\overline{T}$.  Let $\Xi$ be a finite index open subgroup of $\pi_1^\et(\overline{T}, \overline{t})$.  We say that $\Xi$ is strongly Galois if it is invariant under $\mathrm{Aut}(\pi_1^\et(\overline{T}, \overline{t}))$ where $\mathrm{Aut}(\pi_1^\et(\overline{T}, \overline{t}))$ is the automorphism group of $\pi_1^\et(\overline{T}, \overline{t})$ as a topological group.

More generally, for any finite index open subgroup $\Xi$, we define the strong Galois closure $\Xi'$ of $\Xi$ to be the intersection
\begin{equation*}
\Xi' := \bigcap_{\Phi \in \mathrm{Aut}(\pi_1^\et(\overline{T}, \overline{t}))} \Phi (\Xi) \subset \pi_1^\et(\overline{T}, \overline{t})
\end{equation*}
Since the \'etale fundamental group is topologically finitely presented by \cite[Th\'eor\`eme II.2.3.1]{Gro2}, there are only finitely many open subgroups of a fixed finite index. Thus this is a finite intersection, and $\Xi'$ is also an open subgroup of finite index.

In particular, if $\overline{R} \to \overline{T}$ is the cover defined by a finite index subgroup $\Xi \subset \pi_1^\et(\overline{T}, \overline{t})$, then the strong Galois closure of $\Xi$ defines an \'etale cover $\overline{R}' \to \overline{T}$.  We will call $\overline{R}' \to \overline{T}$ the strong Galois closure of $\overline{R} \to \overline{T}$.  Note that this notion is independent of the choice of basepoint used to define the cover.
\end{defi}

This construction is particularly useful in the following situation.  Suppose there is a smooth variety $T$ carrying a rational point $t$ defined over the ground field $F$.  Then $\Gal(\overline{F}/F)$ acts on $\pi_1^\et(T, t)$ via the splitting induced by $t$, and hence also acts on $\pi_1^\et(\overline{T}, \overline{t})$ by conjugation.  Thus, for any finite \'etale cover $\overline{R} \to \overline{T}$ the strong Galois closure $\overline{R}' \to \overline{T}$ descends to a morphism $R' \to T$ defined over $F$.

We next define good families of adjoint rigid varieties over $F$.

\begin{defi} \label{defi: goodfamilyovernf}
Fix a field $F$ of characteristic $0$.   A good family of adjoint rigid varieties over $F$ is an $F$-morphism $p: \mathcal{U} \to W$ of smooth quasi-projective varieties and a relatively big and nef $\mathbb{Q}$-divisor $L$ on $\mathcal{U}$ such that the base-change to the algebraic closure is a good family of adjoint rigid varieties over each component of the base.

Let $\overline{Q}$ denote the subset of $\overline{\mathcal{U}}$ as in Definition \ref{defi:goodfamily}.  Note that $\overline{Q}$ descends to $F$ by Proposition \ref{prop: galinvofa}.  We denote this set by $Q$.  A good section of a good family over $F$ is a section avoiding $Q$.
\end{defi}

It is natural to wonder whether one can prove a version of Theorem \ref{theo: thinmapfactoring} over an arbitrary field of characteristic $0$ which takes twists into account.  
However, it seems quite difficult to prove an analogous factoring result since it is hard to decide whether a cover constructed using an \'etale fundamental group descends to the ground field.

Fortunately, in Theorem~\ref{theo: precisetheorem} we do not care about arbitrary breaking thin maps $f: Y \to X$ but only those which contribute a rational point.  By keeping careful track of the rational point we can work with \'etale fundamental groups to construct covers defined over the ground field.  In the end we still need a factoring result for twists, but we only prove it when we have a rational point to work with.  Thus, in the proof below we will mimic the argument of Lemma \ref{lemm: finitelymanycoversinduction} but with many subtle changes which help us keep track of the behavior of rational points and twists.  In particular, in contrast to Section \ref{sect: boundedness} we will need to work with projective varieties throughout.

\begin{lemm} \label{lemm: finitelymanycoversovernf}
Let $X$ be a geometrically uniruled geometrically integral smooth projective variety defined over $F$ and let $L$ be a big and nef $\mathbb Q$-divisor on $X$. Let $p : \mathcal U \rightarrow W$ be a surjective morphism between projective varieties where $\mathcal{U}$ is equipped with a morphism $s : \mathcal U \to X$ which is birational. 
Suppose that there exists an open subset $W^\circ \subset W$ such that $p: \mathcal{U}^\circ \to W^\circ$ is a good family of adjoint rigid varieties over $F$ (where $\mathcal{U}^\circ$ denotes the preimage of $W^\circ$) and that any fiber over $W^\circ$ has the same $a$-invariant with respect to $s^{*}L$ as $X$ does with respect to $L$. Then there is a proper closed subset $C \subsetneq X$ and a finite set of dominant generically finite morphisms  $\{ f_{j}: \mathcal{Y}_{j} \to \mathcal{U} \}$ defined over $F$ that fit into commutative diagrams
\begin{equation*}
\xymatrix{ \mathcal{Y}_{j} \ar[r]^{f_{j}} \ar[d]_{q_{j}} &  \mathcal{U} \ar[d]_{p} \\
T_{j} \ar[r] & W}
\end{equation*}
such that the following holds.
\begin{enumerate}
\item both $\mathcal Y_j$ and $T_j$ are projective and geometrically integral, $\mathcal Y_j$ is smooth, $T_j$ is normal, and $q_{j}: \mathcal{Y}_{j} \to T_{j}$ is generically a good family of adjoint rigid varieties;
\item the canonical model for $a(X, L)f_j^*s^*L + K_{\mathcal Y_j}$ is a morphism and over some open set of $T_j$ this map agrees with $q_j$;  
\item $T_j \rightarrow W$ is dominant, finite, and Galois;
\item $\mathrm{Bir}(\overline{\mathcal Y}_{j}/\overline{X}) = \mathrm{Aut}(\overline{\mathcal Y}_{j}/\overline{X})$;
\item there is a non-empty Zariski open subset $W' \subset W^\circ$ such that for any $j$, for any twist $\mathcal{Y}_{j}^{\sigma}$ over $X$ and for any closed point $b$ in the preimage $T'^{\sigma}_{j}$ of $W'$ we have an isomorphism $\iota^\sigma_{j, b*}: \mathcal{F}_{\mathcal Y^\sigma_{j, b},f_{j}^{\sigma *}s^*L|_{\mathcal Y^\sigma_{j, b}}} \to \mathcal{F}_{\mathcal Y^\sigma_j,f_{j}^{\sigma *}s^*L}$ where 
$\mathcal{Y}^{\sigma}_{j,b}$ denotes the fiber over a closed point $b$ on $T_j'^{\sigma}$, $\iota^\sigma_{j, b} : \mathcal{Y}^{\sigma}_{j,b}\hookrightarrow \mathcal{Y}^{\sigma}_{j}$ is the inclusion, and $\mathcal{F}_{\mathcal Y^\sigma_{j, b},f_{j}^{\sigma *}L|_{\mathcal Y^\sigma_{j, t}}}$ and $\mathcal{F}_{\mathcal Y^\sigma_j,f_{j}^{\sigma *}L}$ are faces of the nef cone of curves as in Definition~\ref{defi:facedefinition};
\item
Suppose that $q: \mathcal{Y} \to T$ is a projective surjective morphism of varieties over $F$ where $\mathcal{Y}$ is smooth and geometrically integral and that we have a diagram
\begin{equation*}
\xymatrix{ \mathcal{Y} \ar[r]^{f} \ar[d]_{q} &  \mathcal{U} \ar[d]_{p} \\
T \ar[r]^{g} & W}
\end{equation*}
 satisfying the following properties:
\begin{enumerate}
\item There is some open subset $T' \subset T$ such that $\mathcal{Y}$ is a good family of adjoint rigid varieties over $T'$ and the map $f: q^{-1}(T') \to \mathcal{U}$ has image in $\mathcal{U}^{\circ}$ and is a good morphism of good families.  
\item There is a rational point $y \in \mathcal{Y}(F)$ such that $s \circ f(y) \not \in C$.
\end{enumerate}
Then for some index $j$ there will be a twist $f_{j}^\sigma : \mathcal Y_j^\sigma \rightarrow \mathcal U$ such that $f(y) \in f_{j}^{\sigma}(\mathcal{Y}_{j}^{\sigma}(F))$.  Furthermore, there is a dominant generically finite map $\widetilde{T} \to T$ such that the main component $\widetilde{q}: \widetilde{\mathcal{Y}} \to \widetilde{T}$ of the base change of $q$ by $\widetilde{T} \to T$
satisfies that the induced map $\widetilde{f}: \widetilde{\mathcal{Y}} \to \mathcal{U}$ will factor rationally through $f_{j}^{\sigma}$ and a general geometric fiber of $\widetilde{q}$ will map birationally to a geometric fiber of the map $q_{j}^{\sigma}: \mathcal{Y}_{j}^{\sigma} \to T_{j}^{\sigma}$.
\end{enumerate}
\end{lemm}

The most important property is (6), which guarantees that such a morphism $f: \mathcal{Y} \to \mathcal{U}$ must factor rationally through a twist of one of the $\mathcal{Y}_{j}$. The other properties allow us to keep track of twists. 
Since the proof of Lemma~\ref{lemm: finitelymanycoversovernf} is long and complicated, we start with an outline of the proof:

\

\noindent
\underline{\bf Outline of the proof of Lemma~\ref{lemm: finitelymanycoversovernf}}:
Recall that $\mathcal{U}$ contains a good open subset $\mathcal{V}$ as in Definition~\ref{defi: goodfamilyovernf}.  We will concentrate only on the rational points contained in $\mathcal{V}$ (and add the complement of $\mathcal{V}$ to the ``exceptional set'').  We will frequently need to shrink $\mathcal{V}$ in the course of the proof as our ``bad'' locus increases.

\noindent
{\bf Step 1: Improve the properties of $p$.}

The first step of the proof is to take a base change of $p: \mathcal{U} \to W$ over a finite morphism $W^{\mu} \to W$ so that the base change $p^{\mu}: \mathcal{U}^{\mu} \to W^{\mu}$ admits a good section over an open subset of $W$.  (This is necessary to apply the constructions in Theorem~\ref{lemm: finitelymanycoversinduction} and Theorem~\ref{theo: finitelymanycovers} over $\overline{F}$.)

\noindent
{\bf Step 2: Construct the families $\mathcal{Y}_{j} \to T_{j}$.}

\begin{itemize}
\item Step 2a: Perform construction over $\overline{F}$.

Applying Lemma~\ref{lemm: finitelymanycoversinduction} over $\overline{F}$, we obtain a collection of generically finite covers of $\mathcal{U}^{\mu}$ which satisfy a factoring property.  
The resulting maps have three types which we differentiate in our notation:  
\begin{enumerate}
\item Maps $\overline{\mathcal P}_{k} \to \overline{R}_{k}$ which do not descend to the ground field.
\item Maps which descend to the ground field as $\mathcal{P}_{k} \to R_{k}$ such that no twist of $\mathcal{P}_{k}$ carries a rational point mapping to $\mathcal{V}$.
\item Maps which descend to the ground field as $q_{j}: \mathcal{Y}_{j} \to T_{j}$ such that some twist of $\mathcal{Y}_{j}$ carries a rational point mapping to $\mathcal{V}$.
\end{enumerate}
The first two types are not relevant to the discussion, since their rational point contributions lie in a closed set.  Our focus is on maps of the third type.

\item Step 2b: Improve the properties of $q_{j}$.

We perform a sequence of base changes of $q_{j}: \mathcal{Y}_{j} \to T_{j}$ over finite covers of $T_{j}$ and birational modifications of $\mathcal{Y}_{j}$.  The goal is to ensure three key properties:
\begin{enumerate}
\item the geometric monodromy action on the N\'eron-Severi space of the general geometric fibers of $q_j$ is trivial,
\item the birational automorphism group of $\overline{\mathcal{Y}}_{j}/\overline{X}$ agrees with the automorphism group of $\overline{\mathcal{Y}}_{j}/\overline{X}$, and
\item the map $T_{j} \to W_{j}^{\mu}$ is strongly Galois closed.
\end{enumerate}
We show that one can achieve these properties simultaneously, but unfortunately the construction is a little complicated.
\end{itemize}

\noindent
{\bf Step 3: Prove the universal property of the $q_{j}: \mathcal{Y}_{j} \to T_{j}$.}

Suppose we have a map $q: \mathcal{Y} \to T$ and a rational point $y$ on $\mathcal{Y}$ as in the statement of the theorem.

\begin{itemize}
\item Step 3a: Improve the properties of $q$.

Since we started with a base change $W^{\mu} \to W$ for the map $\mathcal{U} \to W$, we need to take a compatible base change of $q: \mathcal{Y} \to T$.  In fact we construct a base change
\begin{equation*}
\xymatrix{ \widetilde{\mathcal{Y}}_{\widetilde{T}} \ar[r] \ar[d]_{\widetilde{q}} &  \mathcal{Y} \ar[d]_{q} \\
\widetilde{T} \ar[r] & T }
\end{equation*}
over a generically finite map $\widetilde{T} \to T$ so that $\widetilde{\mathcal{Y}}$ admits a rational point $\widetilde{y}$ mapping to $y$, $\widetilde{\mathcal{Y}}$ admits a morphism to $\mathcal{U}^{\mu}$, and $\widetilde{q}$ admits a rational section through $\widetilde{y}$.  The key property is that there is a rational point mapping to $y$; this may necessitate replacing $W^{\mu}$ by a twist.

\item Step 3b: Construct a factoring over $\overline{F}$.

Using Lemma~\ref{lemm: finitelymanycoversinduction} over $\overline{F}$, we see that after perhaps taking a further base change over a generically finite morphism to $\overline{\widetilde{T}}$ the map $\overline{\widetilde{\mathcal{Y}}}_{\overline{\widetilde{T}}} \to \overline{\mathcal{U}^{\mu}}$ factors through one of the $\overline{\mathcal{Y}_{j}}$ or $\overline{\mathcal{P}_{k}}$.

\item Step 3c: Descend an open subset to $F$.

This is the most difficult part of the proof.  Suppose that the map $\overline{\widetilde{\mathcal{Y}}}_{\overline{\widetilde{T}}} \to \overline{\mathcal{U}^{\mu}}$ factors through $\overline{\mathcal{Y}_{j}}$.  The first step is to show that (a Stein factorization of) the fiber $Y'$ of the structural map $\overline{\mathcal{Y}_{j}} \to \overline{T_{j}}$ descends to $F$.  The next is to show that the preimage of $\overline{\mathcal{V}}$ descends to $F$.  We first leverage strong Galois closure to construct a subgroup of $\pi_{1}^{\text{\'et}}(\mathcal{V})$ by piecing together the subgroup defining the fiberwise cover $Y' \to \mathcal{U}^{\mu}_{p^{\mu}(f(y))}$ and the subgroup defining the cover of the base $\overline{T_{j}} \to \overline{W^{\mu}}$.  This subgroup defines an \'etale cover of $\mathcal{V}$ over $F$ which is the desired open set.  Finally, this construction shows the rational point $\widetilde{y} \in \widetilde{\mathcal Y}_{\widetilde{T}}$ has the same image in $\mathcal{U}^{\mu}$ as a rational point $y_{j}$ on some twist of $\mathcal{Y}_{j}$.

\item Step 3d: Construct a factoring over $F$.

Using the rational point $y_{j}$ in the previous step, we show that the factoring over $\overline{F}$ actually descends to $F$ for the corresponding twist of $\mathcal{Y}_{j}$.
\end{itemize}

We now turn to the proof of Lemma~\ref{lemm: finitelymanycoversovernf}:

\begin{proof}[Proof of Lemma~\ref{lemm: finitelymanycoversovernf}]
Let $Q$ denote the closed subset of  $\mathcal{U}^\circ$ as in Definition \ref{defi: goodfamilyovernf} and let $\mathcal{V}$ denote its complement.  During the construction we will shrink $W^{\circ}$ several times; when we do this operation, $\mathcal{U}^{\circ}$ will continue to denote its preimage and $\mathcal{V}$ will continue to denote $\mathcal{U}^{\circ} \backslash Q$. 
 
 \noindent
\underline{\bf Step 1: Improve the properties of $p$.}

We next construct a finite morphism $W^{\mu} \to W$; during this construction we let $W^{\mu \circ}$ denote the preimage of $W^{\circ}$ (and use the same notation after shrinking $W^{\circ}$).  By taking a general complete intersection of hyperplanes in $\mathcal U$ and shrinking $W^\circ$ we obtain a base change $W' \to W$ such that $W'^\circ \to W^\circ$ is proper and \'etale, $\mathcal U'^\circ = \mathcal U^\circ \times_{W^\circ} W'^\circ \to W'^\circ$ is a good family,  and it admits a good section $\zeta'$. 
After taking the Galois closure $W^\mu \rightarrow W'$, we replace $W^\mu \to W$ by its Stein factorization so that $W^\mu \to W$ is Galois and finite.  After shrinking $W^\circ$ again, we may assume that $p^\mu : \mathcal V^{\mu} \rightarrow W^{\mu \circ}$ is a good family and admits a good section $\zeta$ where $\mathcal V^{\mu}$ is $\mathcal V \times_{W^\circ}W^{\mu \circ}$ (and as usual we will use the same notation even after shrinking $W^{\circ}$).  We also set $\mathcal{U}^{\mu}$ to be the main component of $\mathcal U \times_{W}W^{\mu}$ with maps $s^{\mu}: \mathcal{U}^{\mu} \to X$ and $p^{\mu}: \mathcal{U}^{\mu} \to W^{\mu}$.
Let $\overline{D} \subset \overline{W}^{\mu \circ}$ be the proper closed subset obtained by applying Lemma~\ref{lemm: finitelymanycoversinduction} over $\overline{F}$. After including its Galois conjugates, we may assume that $D$ is defined over the ground field.  Initially we set $C$ to be the closure of $s^\mu((p^{\mu})^{-1}(D\cup R)) \cup s(Q)$ where $R$ is the ramification locus of $W^\mu \to W$; we will increase $C$ later. We also shrink $W^\circ$ so that $W^{\mu \circ} \to W^\circ$ is proper and \'etale.

\noindent
\underline{\bf Step 2: Construct the families $\mathcal{Y}_{j} \to T_{j}$.}

We next construct the families $\mathcal{Y}_{j}$.  We may suppose that $\mathcal{V}$ admits a rational point since otherwise condition (6) is vacuous. Since $W^\mu \to W$ is Galois we may ensure that $\mathcal V^\mu$ admits a rational point after replacing $W^\mu$ by its twist.  (Note that after this change the section $\overline{\zeta}$ may not be defined over the ground field but is still defined after base change to $\overline{F}$.  The other properties of $W^{\mu}$ are preserved by replacing by a twist.)
Let $w^{\mu}$ denote a rational point on $W^{\mu \circ}$ which is the image of a rational point in $\mathcal{V}^{\mu}$ and define the geometric point $\overline{v}^\mu = \overline{\zeta}(w^\mu)$.  

{\bf Step 2a: Perform construction over $\overline{F}$}:
Consider the set of subgroups $$\Xi_{j} \subset \pi_1^\et(\overline{\mathcal V}^\mu \cap \overline{\mathcal U}^{\mu}_{w^\mu}, \overline{v}^\mu),$$ constructed as in Lemma \ref{lemm: finitelymanycoversinduction} applied to $\overline{\mathcal{U}}^{\mu \circ} \to \overline{W}^{\mu \circ}$.  Each $\Xi_{j}$ yields a normalizer $N_{j} \subset \pi_1^\et(\overline{W}^{\mu \circ}, w^\mu)$. 
Let $(\Xi_{j} \rtimes \overline{\zeta}_*N_j)'$ be the strong Galois closure of $\Xi_{j} \rtimes \overline{\zeta}_*N_j$ in $\pi_1^\et(\overline{\mathcal V}^\mu, \overline{v}^\mu)$.
We define $\widetilde{N}_{j} \subset N_{j}$ as the preimage of $(\Xi_{j} \rtimes \overline{\zeta}_*N_j)'$ via $\overline{\zeta}_*$. Let $\widetilde{N}_j'$ be the strong Galois closure of $\widetilde{N}_j$ in $\pi_1^\et(\overline{W}^{\mu \circ}, w^\mu)$.
We define $\widetilde{\Upsilon}_{j} = \Xi_{j} \rtimes \overline{\zeta}_*\widetilde{N}_{j}'$.  Each $\widetilde{\Upsilon}_{j}$ defines an \'etale cover $\overline{\mathcal{E}}_{j} \to \overline{\mathcal{V}}^{\mu}$, and by composing with the \'etale map $\overline{\mathcal{V}}^{\mu} \to \overline{\mathcal{V}}$ we obtain an \'etale cover $\overline{\mathcal{E}}_{j} \to \overline{\mathcal{V}}$.  

By taking a projective closure of the fibers of $\overline{\mathcal{E}}_{j}$ over $\overline{W}^{\mu \circ}$ and passing to a resolution we obtain a projective family $\overline{\widetilde{\mathcal{E}}}_{j} \to \overline{W}^{\mu \circ}$. Let $\overline{R}^\circ_j \to \overline{W}^{\mu \circ}$ be the Stein factorization of this map; it is defined by $\widetilde{N}_j'$.
Note that by our construction, $\overline{R}_j^\circ \to \overline{W}^{\mu \circ}$ is strongly Galois. 
After shrinking $W^\circ$ we may assume that for every $j$ the $a$-values and the Iitaka dimension of any fiber of $\overline{\widetilde{\mathcal{E}}}_{j} \to \overline{R}_j^\circ$ is constant.
If fibers of this family do not have the same $a$-value as $X$ or they are not adjoint rigid, then henceforth we disregard these $j$. 
After possibly shrinking $W^\circ$ again we may assume that for every $j$ the family $\overline{\widetilde{\mathcal{E}}}_{j} \to \overline{R}_j^\circ$ is a good family.
We enlarge $C$ by adding $s(p^{-1}(W\setminus W^\circ))$.

We now consider two cases.  First suppose that $\overline{\mathcal{E}}_{j} \to \overline{\mathcal{V}}$ fails to descend to $F$ in such a way that $\overline{\mathcal{E}}_{j}$ admits a rational point.  At least over $\overline{F}$ we can compactify $\overline{\mathcal{E}}_{j} \to \overline{\mathcal{V}}$ to obtain a morphism of smooth projective varieties $\overline{\mathcal{P}}_{k} \to \overline{\mathcal{U}}$ over $\overline{F}$.  We denote by $s_{k}$ the induced map $\overline{\mathcal{P}}_{k} \to \overline{X}$.  
We let $\overline{\mathcal{P}}_{k} \to \overline{R}_{k}$ denote the Stein factorization of the map $\overline{\mathcal{P}}_{k} \to \overline{W}$. 
Note that by the construction in the previous paragraph $\overline{\mathcal P}_k^\circ \rightarrow \overline{R}^\circ_k$ is a good family where $\overline{\mathcal P}_k^\circ$ is the preimage of $\overline{W}^\circ$.  
We enlarge $C$ by adding $s_{k}(\overline{\mathcal P}_{k}\setminus \overline{\mathcal P}_{k}^\circ) \cup s(p^{-1}(\overline{B}_{k})) \cup \overline{E}_{k}$ where $\overline{E}_{k}$ is the branch locus of $s_{k}: \overline{\mathcal P}_k \to \overline{X}$ and $\overline{B}_{k}$ is the branch locus of $\overline{R}_k \to \overline{W}$.
By taking the union with Galois conjugates we may assume that $C$ is defined over the ground field.

For the second case, suppose the map $\overline{\mathcal{E}}_{j} \rightarrow \overline{\mathcal V}$ descends to a morphism $\mathcal{E}_{j} \to \mathcal{V}$ over $F$ in such a way that $\mathcal{E}_{j}$ admits a rational point.  Choose one such $F$-model $\mathcal{E}_{j}$ with a rational point.  
We then define $\mathcal{Y}_{j}$ over $F$ as a smooth projective compactification of $\mathcal E_j$ with a morphism $\mathcal Y_j \to \mathcal U$ extending $\mathcal E_j \to \mathcal V$ and let $T_{j}$ denote the Stein factorization of the map $\mathcal{Y}_{j} \to W$.  
The structure map $q_{j}: \mathcal{Y}_{j} \to T_{j}$ is generically a good family of adjoint rigid varieties. 
We let $T_{j}^{\circ}$, $\mathcal{Y}_{j}^{\circ}$ denote the preimages of $W^{\circ}$ and let $\mathcal E_j$ denote the preimage of $\mathcal V$. We will continue to use this notation after shrinking $W^{\circ}$.

{\bf Step 2b: Improve the properties of $q_{j}$}:
We make a few additional changes to the family.  
First we apply Lemma~\ref{lemm: birandaut} to $\mathcal{Y}_j/\mathcal U$ and replace $\mathcal{Y}_{j}$ by a birational model to ensure that $\mathrm{Bir}(\overline{\mathcal Y}_{j}/\overline{X}) = \mathrm{Aut}(\overline{\mathcal Y}_{j}/\overline{X})$.  
We claim that after shrinking $W^\circ$, $q_{j}$ is a good family of adjoint rigid varieties over $T_{j}^{\circ}$.  Indeed, Lemma \ref{lemm: opensetgoodfamily} shows that this is true after base change to $\overline{F}$, and after possibly shrinking further this open set will descend to $F$.   
Note that by our convention shrinking $W^{\circ}$ also causes $\mathcal{E}_{j}$ to shrink.  After this change if no twist of the $\mathcal{E}_{j}$ contains a rational point, then we add $s_{j}(\mathcal Y_{j}\setminus \mathcal E_{ j}) \cup s(p^{-1}(B_{j})) \cup E_{j}$ to $C$, where $E_{ j}$ is the branch locus of $s_{j}: \mathcal Y_{j}\rightarrow X$ and $B_{j}$ is the branch locus of $T_{j}\rightarrow W$.  To distinguish such types of families, we will henceforth relabel them and add them to the list of $q_k : \mathcal P_k \rightarrow R_k$ with the evaluation map $s_k : \mathcal P_k \rightarrow X$.  If some twist of $\mathcal{E}_{j}$ does contain a rational point, then we replace our families with this twist and continue the construction. In this situation we must also replace $T_j, W^\mu$ by twists in the following way. First, when we replace $\mathcal{Y}_{j}$ by a twist which carries a rational point in $\mathcal{E}_{j}$, we redefine $T_j$ to be the Stein factorization of $\mathcal Y_j \to W$. Since $W^\mu \to W$ is Galois and finite, for some twist $W^{\mu \sigma}$ the image of our rational point on $\mathcal E_j$ is a rational point on $W^{\mu \sigma}$. Then $T_j$ admits a morphism to $W^{\mu \sigma}$.

Since we are in a situation where $T_{j}^{\circ}$ admits a rational point $b_{j}$, there is a fiber of $q_{j}$ over $T_{j}^{\circ}$ defined over the ground field.  This implies that we can find a base change defined over $F$ which kills the geometric monodromy action on the N\'eron-Severi space of a general fiber; indeed one can consider a projective closure of the \'etale cover of $T_{j}^{\circ}$ defined by $G \rtimes \Gal(\overline{F}/F) \subset \pi_1^\et(T_{j}^{\circ}, b_{j})$ where $G$ is the kernel of the geometric monodromy action on the geometric N\'eron-Severi space of the fiber defined over the ground field.
After taking this finite base change (which we continue to represent by the notation $T_{j}$, $\mathcal{Y}_{j}$, etc.~for simplicity), we may assume that the geometric monodromy of $\pi_{1}^{\et}(\overline{T^{\circ}_{j}},b_{j})$ on the N\'eron-Severi space of a general fiber of $q_{j}$ is trivial. While doing so we keep that $T_j$ is normal and $T_j \to W$ is finite. We take another cyclic cover so that the ramification locus contains the pullback of an ample divisor on $T_j$ which does not contain $b_j$.   After shrinking $W^\circ$ we may guarantee that $T_j^\circ \to W^{\mu \circ}$ is proper and \'etale. After taking a strong Galois closure and descending to $F$, 
we may assume that (i) $T_j$ is normal, (ii) $T_{j}/W^\mu$ is finite, and (iii) $T_{j}/W^\mu$ satisfies that the subgroup of $\pi_{1}^{\et}(\overline{W}^{\mu \circ},w^{\mu})$ defined by $\overline{T}_{j}^\circ$ is strongly Galois. We replace $\mathcal Y_j$ by the main component of the base change.
We then apply Lemma~\ref{lemm: birandaut} to $\mathcal Y_j/\mathcal U$ again and replace $\mathcal{Y}_{j}$ by a birational model to ensure that $\mathrm{Bir}(\overline{\mathcal Y}_{j}/\overline{X}) = \mathrm{Aut}(\overline{\mathcal Y}_{j}/\overline{X})$ and the canonical map for $K_{\mathcal{Y}_{j}} + a(X,L)s_{j}^{*}L$ is a morphism. 
We may need to shrink $W^{\circ}$ to preserve the good family structure over its preimage  $T_{j}^{\circ}$.  
By the construction we have $T_{j}^{\circ}$ is proper and \'etale over the open set $W^{\circ}$. 
After possibly shrinking $W^\circ$ we may guarantee that $f_{j} : \mathcal{E}_{j} = f_{j}^{-1}(\mathcal V) \rightarrow \mathcal V$ is \'etale. 
If no twist of $\mathcal{E}_{j}$ contains a rational point again, then we add $s_{j}(\mathcal Y_{j}\setminus \mathcal E_{ j}) \cup s(p^{-1}(B_{j})) \cup E_{j}$ to $C$ and then we relabel this family as one of the $q_k : \mathcal P_k \rightarrow R_k$ with the evaluation map $s_k : \mathcal P_k \rightarrow X$. If some twist of $\mathcal{E}_{j}$ does contain a rational point, then we replace our families with this twist.

After all these changes we have a commutative diagram
\begin{equation*}
\xymatrix{
\mathcal Y_{j} \ar@{>}[r]^{f_{j}} \ar@{>}[d]&  \mathcal U \ar@{>}[d]\\
T_{j} \ar@{>}[r] & W}
\end{equation*}
We enlarge $C$ by adding $s_{j}(\mathcal Y_{j}\setminus \mathcal E_{ j}) \cup s(p^{-1}(B_{j})) \cup E_{j}$ where $E_{ j}$ is the branch locus of $s_{j}: \mathcal Y_{j}\rightarrow X$ and $B_{j}$ is the branch locus of $T_{j}\rightarrow W$.  Note that by the construction we have now verified Lemma \ref{lemm: finitelymanycoversovernf} (1),(3),(4).  Recall that during the construction we took a cyclic cover 
so that the ramification divisor of $T_j \to W$ contains an ample divisor. Hence there is an ample $\mathbb{Q}$-divisor $H$ on $T_{j}$ such that $K_{\mathcal{Y}_{j}} + a(\mathcal{Y}_{j},s_{j}^{*}L)s_{j}^{*}L - q_{j}^{*}H$ is $\mathbb{Q}$-linearly equivalent to an effective divisor.  Thus  Lemma \ref{lemm: finitelymanycoversovernf} (2) follows from Lemma \ref{lemm:birationaltocanonical}.

Before continuing, we prove that there is a Zariski open subset $W' \subset W^{\circ}$ 
such that for any twist $\mathcal{Y}_{j}^{\sigma}$ over $X$ and for any closed point $b \in T'^{\sigma}_{j}$ we have an isomorphism $\iota^\sigma_{j, b*}: \mathcal{F}_{\mathcal Y^\sigma_{j, b},f_{j}^{\sigma *}L|_{\mathcal Y^\sigma_{j, b}}} \to \mathcal{F}_{\mathcal Y^\sigma_j,f_{j}^{\sigma *}L}$ where $\iota^\sigma_{j, b}: \mathcal Y^\sigma_{j, b} \hookrightarrow \mathcal Y^\sigma_{j}$ is the inclusion.  By construction $\overline{\mathcal{Y}}^{\sigma}_{j}$ has a birational model with a structure map to $\overline{T}_{j}^{\circ}$ which has a trivial geometric monodromy action.  Furthermore, the structure map to $\overline{T}_{j}^{\circ}$ is birational to the canonical map by property (2).  Thus Lemma \ref{lemm:birationaltocanonical} verifies the hypotheses of Corollary \ref{coro: bvalequality} on this birational model.  By applying Corollary \ref{coro: bvalequality} to this model and using Lemma \ref{lemm:birfaceinv} to transfer the result to $\overline{\mathcal{Y}}_{j}^{\sigma}$, we deduce that for each $j$ we have an isomorphism $\overline{\iota}^\sigma_{j, b*}: \mathcal{F}_{\overline{\mathcal Y}^\sigma_{j, \overline{b}}, \overline{f}_{j}^{\sigma *} \overline{L}|_{\overline{\mathcal Y}^\sigma_{j, \overline{b}}}}\to \mathcal{F}_{\overline{\mathcal Y}^\sigma_j,\overline{f}_{j}^{\sigma *}\overline{L}}$ for every $\overline{b}$ lying above some open subset of $\overline{W}^{\circ}$. 
The desired equality follows by taking the Galois invariant part.  Since the index set of $j$ is a finite set we find an open subset $W'$ which works for all $j$ simultaneously.  This verifies Lemma \ref{lemm: finitelymanycoversovernf} (5).

Note that the families constructed here are geometrically independent of the initial choice of $\overline{v}^\mu$.  Indeed, we only used $\overline{v}^\mu$ to define geometric covers over $\overline{F}$; all the other choices in the construction were obtained intrinsically from the geometry of this finite set of  covers.  Thus, we can at a later stage choose a (possibly different) basepoint $w^{\mu}$ in a twist of $W^{\mu \circ}$ and pretend that all our constructions were made with respect to this choice all along.

\noindent
\underline{\bf Step 3: Prove the universal property of the $q_{j}: \mathcal{Y}_{j} \to T_{j}$.}

Now we prove the universal property (6) for these families.  Assume that $f: \mathcal Y \rightarrow \mathcal U$ is a morphism and $y \in \mathcal{Y}(F)$ is a rational point as in the statement.
Our goal is to show that $f(y) \in f_j^\sigma(\mathcal Y_j^\sigma(F))$ for some twist $\sigma$.  We set $t = q(y)$. After resolving $T$ and $\mathcal Y$, we may assume that $t$ is a smooth point of $T$.

{\bf Step 3a: Improve the properties of $q$}:
First of all, since $W^{\mu}$ is Galois over $W$, after replacing $W^\mu$ by its twist we may assume that it carries a rational point whose image in $W$ is the same as the image of $t \in T$ in $W$.  Note that since $s \circ f(y) \not \in C$, the map $T \times_W W^\mu \to T$ is \'etale over an open neighborhood of $t$.  Thus there is some component of $T \times_W W^\mu$ which maps dominantly to $T$ and which admits a rational point $t^\mu$ in its smooth locus mapping to $t$.  Let $T^\mu$ be the normalization of this dominant component.  Let $\mathcal Y^\mu$ be a smooth resolution of the main component of $\mathcal Y \times_T T^\mu$. Since the map $T^{\mu} \to T$ is \'etale in a neighborhood of the rational point $t^{\mu}$ we know that the resolution process is an isomorphism on a neighborhood of the fiber of $\mathcal{Y}^{\mu} \to T^{\mu}$ over $t^{\mu}$.  Thus this fiber admits a rational point mapping to $y$ which we denote by $y^{\mu}$. We denote the induced morphism by $q^\mu : \mathcal Y^\mu \to T^\mu$ and let $t^\mu = q^\mu(y^\mu)$. Note that $y^\mu$ is a smooth point of $\mathcal Y^\mu$.

Next, if we let $\widetilde{T}$ be a general intersection of hyperplanes through $y^\mu$, then for the generically finite surjective base change $\widetilde{T} \rightarrow T^\mu$ there is a rational point $\widetilde{t} \in \widetilde{T}(F)$ mapping to $t^\mu$ such that $\widetilde{T}$ is smooth at $\widetilde{t}$ and the main component $\mathcal Y_{\widetilde{T}}$ admits a rational section $\tau$ such that $(y^\mu, \widetilde{t}) = \tau(\widetilde{t})$.
Let $\widetilde{\mathcal Y}_{\widetilde{T}}$ be a resolution of $\mathcal Y_{\widetilde{T}}$ chosen in such a way that $\widetilde{\mathcal Y}_{\widetilde{T}}$ still admits a rational section $\tau$ which is well-defined at $\widetilde{t}$.  
Thus we obtain a rational point $y' \in \widetilde{\mathcal Y}_{\widetilde{T}}$ mapping to $y$ and $\widetilde{t}$. Let $v$ be the image of $y'$  in $\mathcal V^\mu$ and set $w^\mu = p^\mu(v)$ and $\overline{v}^\mu = \overline{\zeta}(w^\mu)$.

{\bf Step 3b: Construct a factoring over $\overline{F}$}:
Recall that the morphism $\mathcal U^{\mu} \to W^{\mu}$ has a geometric section $\overline{\zeta}$.  Working over $\overline{F}$, Lemma \ref{lemm: finitelymanycoversinduction} and Corollary \ref{coro: finitelymanycoversbasechange} show that for some Iitaka base change of $\overline{\widetilde{\mathcal Y}}_{\widetilde{T}}$ the induced map to $\overline{\mathcal{U}}^{\mu}$ factors rationally through $\overline{\mathcal Y_{j}}$ for some $j$ or $\overline{\mathcal P}_{k}$ for some $k$.  Assume for a contradiction that the map factors rationally through $\overline{\mathcal P}_{k}$.  
 We claim that if we take the Stein factorization $Y'$ of the map of fibers $\widetilde{\mathcal Y}_{\widetilde{t}} \rightarrow \mathcal U^{\mu}_{w^\mu}$ and then base change to $\overline{F}$ the result is birational to the adjoint rigid variety $\overline{\mathcal{P}}_{k,\overline{r}}$ where $\overline{r}$ is some preimage of $w^\mu$ 
 which we specify later. Note that due to our construction $\overline{\mathcal{P}}_{k,\overline{r}}$ is irreducible with $a$-invariant equal to $a(X, L)$ and it is adjoint rigid.
To see the claim, first choose an open subset $\overline{\widetilde{T}}^{\circ}$ of $\overline{\widetilde{T}}$ that contains $\widetilde{t}$ 
such that the image of this set in $\overline{W}$ is contained in $\overline{W}^{\circ}$ as defined above and the $\overline{\tau}$-image of this set lies in the preimage of $\overline{\mathcal{V}}$.  
Let $\overline{T}^{\nu}$ denote the \'etale cover of $\overline{\widetilde{T}}^\circ$ defined by the finite index subgroup of $\pi_{1}^{\et}(\overline{\widetilde{T}}^{\circ}, \overline{\widetilde{t}})$ constructed by pulling back under $\overline{\tau}$ the subgroup of $\pi_{1}(\overline{\mathcal{V}}, \overline{v}^\mu)$ corresponding to the \'etale cover defined by $\overline{\mathcal{E}}_{k}$.  For the open subset of $\overline{\widetilde{T}}^{\circ}$ over which we have a good family, just as in Lemma \ref{lemm: finitelymanycoversinduction} we know that the main component of the base change $\overline{\widetilde{\mathcal{Y}}}^{\nu}$ over $\overline{T}^{\nu}$ admits a rational map to $\overline{\mathcal{P}}_{k}$.  Since the map $\overline{T}^{\nu}$ to $\overline{\widetilde{T}}^\circ$ is \'etale, $\overline{\widetilde{\mathcal Y}}^\nu$ is smooth in a neighborhood of the fiber $\overline{\widetilde{\mathcal Y}}^{\nu}_{\overline{t}^{\nu}}$ where $\overline{t}^\nu$ is a geometric point mapping to $\overline{\widetilde{t}}$.
Let $\overline{\mathcal{Y}}^{*}$ denote a smooth resolution of the rational map to $\overline{\mathcal{P}}_{k}$. 
The fiber $\overline{\mathcal{Y}}^{*}_{\overline{t}^{\nu}}$ maps to some fiber $\overline{\mathcal P}_{k,\overline{r}}$.  
We claim that we have a morphism $\overline{T}^\nu \to \overline{R}_k^\circ$ and the resulting morphism $\overline{\mathcal{Y}}^{*} \to \overline{\mathcal P}_k^\circ\times_{\overline{R}_k^\circ} \overline{T}^{\nu }$ is birational. Indeed, we have a morphism $\overline{\mathcal Y}^* \to \overline{\mathcal P}_k^\circ \to \overline{R}_k^\circ$. Taking the Stein factorization induces $\overline{\mathcal Y}^* \to \overline{T}^\nu \to \overline{R}_k^\circ$. The resulting map $\overline{\mathcal Y}^* \to \overline{\mathcal P}_k^\circ \times_{\overline{R}_k^\circ}\overline{T}^\nu$ is a morphism of $\overline{T}^\nu$-schemes which is birational on a general fiber, and thus birational. Since the induced map $\overline{\mathcal{Y}}^{*} \to \overline{\mathcal P}_k^\circ\times_{\overline{R}_k^\circ} \overline{T}^{\nu }$ is birational and the target is smooth, this map has connected fibers.  In particular the map $\overline{\mathcal{Y}}^{*}_{\overline{t}^{\nu}} \to \overline{\mathcal P}_{k,\overline{r}}$ has connected fibers and we deduce that the Stein factorization of $\overline{\mathcal{Y}}^{*}_{\overline{t}^{\nu}} \to \overline{\mathcal{U}}^{\mu}_{\overline{w^\mu}}$ is birational to ${\overline{\mathcal P}}_{k,\overline{r}}$. 
Note that the map $\overline{\mathcal{Y}}^{*}_{\overline{t}^{\nu}} \to \overline{\mathcal{U}}^{\mu}_{\overline{w^\mu}}$ also factors through our original fiber $\overline{\widetilde{\mathcal Y}}_{\overline{\widetilde{t}}}$ and that the first step of this factoring has connected fibers.  Thus the Stein factorization of our original fiber over $\overline{F}$ is also birational to $\overline{\mathcal P}_{k,\overline{r}}$.  Since Stein factorization commutes with base change  to the algebraic closure our assertion follows.
This implies that the subgroup $\Xi_{k}$ which corresponds to the cover $\overline{\mathcal{P}}_{k, \overline{r}} \to \overline{\mathcal U}^{\mu}_{\overline{w^\mu}}$ admits an extension $\widetilde{\Xi}_k \subset \pi_1^\et(\mathcal V^{\mu}_{w}, \overline{v}^\mu)$ corresponding to the cover $Y' \to \mathcal U^{\mu}_{w^\mu}$ defined over the ground field.

We next show that $f_{k}^{-1}(\overline{\mathcal V})$ must descend to the ground field.  
First note that $\pi_1^\et(\overline{R}_k^\circ, \overline{r}_k)$ is preserved as a subset of $\pi_1^\et(W^{\mu \circ}, w^{\mu})$ under conjugation by any element in the image of the splitting $\epsilon: \Gal(\overline{F}/F) \to \pi_1^\et(W^{\mu \circ}, w^{\mu})$ induced by $w^\mu$.
Indeed, by construction there exists some open subset $W^{\mu a}$ containing $W^{\mu \circ}$ 
such that for the preimage $\overline{R}_k^{a}$ of $\overline{W}^{\mu a}$ in $\overline{R}_k$, $\overline{R}_k^{a} \to \overline{W}^{\mu a}$ is proper, \'etale and strongly Galois.  In particular, the strong Galois property implies that $\pi^{\text{\'et}}_1(\overline{R}_k^a, \overline{r}_k)$ is invariant in $\pi^{\text{\'et}}_1(W^{\mu a}, w^\mu)$ under conjugation. Since $\pi^{\text{\'et}}_1(\overline{R}_k^\circ, \overline{r}_k)$ is the preimage of $\pi^{\text{\'et}}_1(\overline{R}_k^a, \overline{r}_k)$ under the surjection $\pi^{\text{\'et}}_1(W^{\mu \circ}, w^\mu) \to \pi^{\text{\'et}}_1(W^{\mu a}, w^\mu)$, it is also preserved under conjugation by the image of $\epsilon$.

Now the homomorphism $\widetilde{\Xi}_k \to  \pi_1^\et(W^{\mu \circ}, w^{\mu})$ factors through $\epsilon : \mathrm{Gal}(\overline{F}/F) \to \pi_1^\et(W^{\mu \circ}, w^{\mu})$.  Using the homotopy exact sequence of \'etale fundamental groups we see that the induced map $\widetilde{\Xi}_{k} \to \mathrm{Gal}(\overline{F}/F)$ is surjective, and we let $\delta : \mathrm{Gal}(\overline{F}/F) \to \widetilde{\Xi}_k$ denote any set-theoretic section.  
Note that every element in $\widetilde{\Xi}_{k}$ is a product of an element of $\Xi_{k}$ and an element in the image of $\delta$.

We claim that $\overline{\zeta}_*(\pi_1^\et(\overline{R}_k^\circ, \overline{r}_k)) \cdot\widetilde{\Xi}_k = \widetilde{\Xi}_k \cdot \overline{\zeta}_*(\pi_1^\et(\overline{R}_k^\circ, \overline{r}_k))$ as subsets of $\pi_{1}^{\et}(\mathcal{V}^{\mu},\overline{v}^{\mu})$.  
 Pick $\overline{\zeta}_*(\gamma) \in \overline{\zeta}_*(\pi_1^\et(\overline{R}_k^\circ, \overline{r}_k))$ and $g \in \Xi_k$ and $\sigma \in \mathrm{Gal}(\overline{F}/F)$.  Then we have
\[
\overline{\zeta}_*(\gamma) g \delta(\sigma) = \overline{\zeta}_*(\gamma) g \overline{\zeta}_*(\gamma)^{-1} \cdot \overline{\zeta}_*(\gamma) \delta(\sigma)\overline{\zeta}_*(\gamma')^{-1} \cdot \overline{\zeta}_*(\gamma'),
\]
where $\gamma' = \epsilon(\sigma)^{-1}\gamma \epsilon(\sigma)$.
By construction $\overline{\zeta}_{*}(\gamma)$ is contained in the normalizer of $\Xi_{k}$ so we have $\overline{\zeta}_*(\gamma) g \overline{\zeta}_*(\gamma)^{-1} \in \Xi_k$.  Thus it suffices to show that $ \overline{\zeta}_*(\gamma) \delta(\sigma)\overline{\zeta}_*(\gamma')^{-1} \in \widetilde{\Xi}_k$. Since $\delta(\sigma) \in \widetilde{\Xi}_k$ we may instead show
\[
 \overline{\zeta}_*(\gamma) \delta(\sigma)\overline{\zeta}_*(\gamma')^{-1}\delta(\sigma)^{-1} \in \Xi_k.
\]
Observe that $\overline{\zeta}_*(\gamma) \delta(\sigma)\overline{\zeta}_*(\gamma')^{-1}\delta(\sigma)^{-1}$ maps to the identity under $\pi_1^\et(\mathcal {V}^{\mu}, \overline{v}^\mu) \to \Gal(\overline{F}/F)$, so that it can be identified as an element in $\pi_1^\et(\overline{\mathcal V}^{\mu}, \overline{v}^\mu)$.  Furthermore it maps to the identity under $\pi_1^\et(\overline{\mathcal V}^{\mu}, \overline{v}^\mu) \to \pi_1^\et(\overline{W}^{\mu \circ}, w^\mu)$, so we can identify $\overline{\zeta}_*(\gamma) \delta(\sigma)\overline{\zeta}_*(\gamma')^{-1}\delta(\sigma)^{-1}$ with an element of $\pi_1^\et(\overline{\mathcal V}^\mu \cap \overline{\mathcal U}^{\mu}_{w^\mu}, \overline{v}^\mu)$.  
Since the image of $\pi_1^\et(\overline{R}_k^\circ, \overline{r}_k)$ via $\overline{\zeta}_*$ is contained in $(\Xi_{k} \rtimes \overline{\zeta}_*N_k)'$, we have $\overline{\zeta}_*(\gamma), \overline{\zeta}_*(\gamma')^{-1} \in (\Xi_{k} \rtimes \overline{\zeta}_*N_k)'$. Then the strong Galois property implies that $\delta(\sigma)\overline{\zeta}_*(\gamma')^{-1}\delta(\sigma)^{-1} \in (\Xi_{k} \rtimes \overline{\zeta}_*N_k)'$.
Thus we have 
\[
 \overline{\zeta}_*(\gamma) \delta(\sigma)\overline{\zeta}_*(\gamma')^{-1}\delta(\sigma)^{-1} \in \Xi_k\rtimes \overline{\zeta}_*N_k.
\]
However since $ \overline{\zeta}_*(\gamma) \delta(\sigma)\overline{\zeta}_*(\gamma')^{-1}\delta(\sigma)^{-1} \in \pi_1^\et(\overline{\mathcal V}^\mu \cap \overline{\mathcal U}^{\mu}_{w^\mu}, \overline{v}^\mu)$ our assertion follows.

Thus $\widetilde{\Xi}_k \cdot \overline{\zeta}_*(\pi_1^\et(\overline{R}_k^\circ, \overline{r}_k))$ is a subgroup of $\pi_{1}^{\et}(\mathcal{V}^{\mu},\overline{v}^{\mu})$.  Since the image of this group in $\Gal(\overline{F}/F)$ is the full group, one may use this group to define an \'etale cover of $\mathcal V^\mu$ that is defined over $F$.  Since $\widetilde{\Xi}_k \cdot \overline{\zeta}_*(\pi_1^\et(\overline{R}_k^\circ, \overline{r}_k))$ is an extension of $\widetilde{\Upsilon}_{k}$, this \'etale cover coincides with $f_{k}^{-1}(\overline{\mathcal V})$ after base change to $\overline{F}$, and thus we will write this cover as $f_{k}^{-1}(\mathcal V) \to \mathcal{V}^{\mu}$.
Moreover we claim that $f_{k}^{-1}(\mathcal{V})$ admits a fiber birational to $Y'$ and this birational map is an isomorphism on an open neighborhood of the image of $y'$.  Indeed, let $R_k^\circ$ be the Stein factorization of a compactification of $f_{k}^{-1}(\mathcal V) \rightarrow \mathcal V^\mu \to W^{\mu \circ}$.  Then the cover $R_{k}^{\circ} \to W^{\mu \circ}$ corresponds to an extension of $\pi_1^\et(\overline{R}_k^\circ, \overline{r}_k)$ by $\Gal(\overline{F}/F)$.
Using the splitting $\delta$ constructed above and pushing forward to $R_{k}^{\circ}$, we obtain a  
group theoretic splitting $\Gal(\overline{F}/F) \to \pi_1^\et(R_k^\circ, \overline{r}_k)$ compatible with the splitting $\epsilon: \Gal(\overline{F}/F) \to \pi^\et_1(W^{\mu \circ}, w^\mu)$. This section is a homomorphism because $\epsilon$ is. 
On the other hand since $R_k$ is Galois over $W^\mu$, $R_k^\circ$ admits a twist $R_k^{\sigma \circ}$ with a rational point mapping to $w^\mu$. Moreover the fundamental group of $R_k^{\sigma \circ}$ also has a splitting compatible with the splitting of $\pi^\et_1(W^{\mu \circ}, w^\mu)$ coming from $w^{\mu}$.  Altogether we conclude that $R_k^\circ$ and $R_k^{\sigma \circ}$ must be isomorphic to each other, or in other words, that $R_{k}^{\circ}$ comes with a rational point $r_k$ mapping to $w^\mu$.  
By comparing fundamental groups, we see that the fiber over $r_{k}$ is birational to the variety defined by $\widetilde{\Xi}_k$ as claimed.  Furthermore, this birational map is an isomorphism on a neighborhood of $y'$ because $y'$ maps to $\mathcal{V}$.  We conclude that $f_{k}^{-1}(\mathcal{V})$ admits a rational point coming from $y$. 
However, the fact that the geometric model descends to the ground field with a rational point contradicts our definition of the $\overline{\mathcal{P}}_{k}$.  We deduce that this case cannot happen; in other words, some base change of $\overline{\widetilde{\mathcal{Y}}}_{\widetilde{T}}$ admits a rational map to $\overline{\mathcal Y}_{j}$ for some $j$.  (We do not yet know that this rational factoring can be achieved over the ground field, but we will verify this soon.) 

{\bf Step 3c: Descend an open subset to $F$}:
Next we would like to show that some twist of $\mathcal Y_{j}$ contains a rational point $y_j$ mapping to $v$. 
We repeat exactly the same construction that we made above to rule out the case $\overline{\mathcal P}_{k}$.  The result is an \'etale cover of $\mathcal{V}^{\mu}$ that admits a rational point mapping to $v$ and after base change to $\overline{F}$ is an open subset of $\overline{\mathcal{Y}}_{j}$.  We deduce that there is a twist $\mathcal{Y}_{j}^{\sigma}$ admitting a rational point whose image is $v$.  Furthermore, the argument shows that the corresponding fiber $\mathcal{Y}^{\sigma}_{j,t_{j}}$ of $\mathcal{Y}_{j}^{\sigma} \to T_{j}^{\sigma}$ is geometrically birational to $Y'$.  Since $\mathcal{Y}^{\sigma}_{j,t_{j}}$ is induced by $\widetilde{\Xi}_j$ corresponding to $Y'$, they are actually birational over the ground field.

{\bf Step 3d: Construct a factoring over $F$}:
To finish the proof of Lemma  \ref{lemm: finitelymanycoversovernf} (6) we must prove a factoring property over $F$.  Recall that we have constructed a rational point $\widetilde{t} \in \widetilde{T}$ and a rational point $y_{j} \in \mathcal{Y}_{j}^{\sigma}$ such that the image of $y_{j}$ in $\mathcal{U}$ is the same as the image of $\widetilde{t}$ under the rational map $\widetilde{T} \dashrightarrow \mathcal{U}$ given by the composition of the rational section to $\widetilde{\mathcal{Y}}_{\widetilde{T}}$ and the map to $\mathcal{U}$.  Let $\widetilde{T}^{\dagger}$ be the open subset where the rational map to $\mathcal{U}$ is defined.  Consider the base change
\begin{equation*}
\xymatrix{
\widetilde{T}^{\dagger} \times_{\mathcal U} \mathcal Y_{j}^\sigma \ar@{>}[r]\ar@{>}[d]&  \mathcal Y_{ j}^\sigma \ar@{>}[d]\\
\widetilde{T}^{\dagger} \ar@{>}[r] & \mathcal U}.
\end{equation*}
Since $\mathcal{Y}_{j}^{\sigma} \to \mathcal{U}$ is \'etale on a neighborhood of the image of $\widetilde{t}$ and admits a rational point mapping to the image of $\widetilde{t}$, there is a component $T^{\nu}$ of $\widetilde{T} \times_{\mathcal U} \mathcal Y_{j}^\sigma$ which maps dominantly to $\widetilde{T}^{\dagger}$ and admits a rational point $t^{\nu}$ mapping to $\widetilde{t}$.  Furthermore, the base change $\widetilde{\mathcal Y}_{T^\nu}$ is smooth at any point of the fiber over $t^\nu$.  This construction of a base change is as same as the construction in Lemma~\ref{lemm: finitelymanycoversinduction} and Corollary \ref{coro: finitelymanycoversbasechange}, hence after base changing to $\overline{F}$ the map $\widetilde{\mathcal Y}_{T^\nu}\to X$ factors rationally through the twist $\mathcal Y_{j}^\sigma$.

We claim that the map $\widetilde{\mathcal Y}_{T^\nu}\to X$ factors rationally through $\mathcal Y_{j}^\sigma$ over the ground field. Indeed, by the lifting property over $\overline{F}$ one may find a rational map $\overline{h} : \overline{\widetilde{\mathcal Y}_{T^\nu}} \dashrightarrow \overline{\mathcal Y_{j}^\sigma}$ mapping $(y', t^\nu)$ to the point $y_j$ constructed above.  
 Let $s$ be an element of the Galois group. Then both $\overline{h}$ and $\overline{h}^s$ are lifts of the same map to $\mathcal U$ and they  both map $(y', t^\nu)$ to $y_j$. Thus $\overline{h} = \overline{h}^s$ by the uniqueness of the lift.  Thus our assertion follows.
\end{proof}

\begin{proof}[Proof of Lemma~\ref{lemm: finitelymanycoversovernf_intro}]
Lemma~\ref{lemm: finitelymanycoversovernf_intro} is a simplified version of Lemma~\ref{lemm: finitelymanycoversovernf} where $\mathcal U$ is taken to be $X$.
\end{proof}

Finally we prove our main theorem.

\begin{proof}[Proof of Theorem \ref{theo: precisetheorem}:]
As mentioned before $Z_{0}$ and $Z_{3}$ are contained in proper closed subsets of $X$, so it suffices to consider only $Z_{1}$ and $Z_{2}$.

\textbf{Construction of a closed set:}
Let $V$ be the proper closed subset and $p_{i}: \mathcal{U}_{i} \to W_{i}$ be the projective families from Theorem \ref{theo: rigidfamilies} equipped with surjective evaluation maps $s_{i}: \mathcal{U}_{i} \to X$.  

Suppose that $\mathcal U_i$ is not geometrically irreducible.  
Then the Zariski closure $\overline{s_i(\mathcal U_i(F))}$ is a proper closed subset of $X$ where $s_i : \mathcal U_i \rightarrow X$ is the evaluation map. We enlarge $V$ by adding this proper closed subset to $V$.

Suppose that $\mathcal U_i$ is geometrically irreducible. 
Let us further suppose that the evaluation map $s_i : \mathcal U_i \rightarrow X$ is birational.
After applying a resolution, we may assume that $\mathcal U_i$ is smooth.
Let $W_i^\circ$ be a Zariski open locus so that $p_i : p_i^{-1}(W_i^\circ) \rightarrow W_i^\circ$ is a good family of adjoint rigid varieties. Let $Q_i$ be the closed subset associated to this family and  define $\mathcal V_i = p_i^{-1}(W_i^\circ) \setminus Q_i$. We enlarge $V$ by adding the proper closed subset $s_i(\mathcal U_i \setminus \mathcal V_i) \cup s_i(E_i)$ where $E_i$ is the ramification divisor of $s_i$. By applying Lemma~\ref{lemm: finitelymanycoversovernf} to $p_i : \mathcal U_i \rightarrow W_i$, we obtain families $q_{i, j} : \mathcal Y_{i, j} \rightarrow T_{i, j}$ with morphisms $s_{i, j} : \mathcal Y_{i, j} \rightarrow X$. 
We shrink $W_i^\circ$ if necessary so that Lemma~\ref{lemm: finitelymanycoversovernf} (5) is valid for any closed point on $T_{i, j}^{\circ \sigma}$.
We enlarge $V$ by taking the union with $s_i(p_i^{-1}(W_i\setminus W_i^\circ))$ and $C_{i}$ from Lemma~\ref{lemm: finitelymanycoversovernf} for every $i$.

\textbf{Construction of a thin set:}
We now construct a thin set $Z' \subset X(F)$.  The construction involves several steps.  First set $Z' = V(F)$.  
If $\mathcal{U}_{i}$ is geometrically integral and the evaluation map $s_i : \mathcal U_i  \to X$ has degree $> 1$,
then we add $s_i(\mathcal U_i(F))$ to $Z'$.   

Suppose that $\mathcal{U}_{i}$ is geometrically integral and $s_i$ is birational.
As $\sigma$ varies over all $\sigma \in H^1(F, \mathrm{Aut}(\overline{\mathcal Y}_{i, j}/\overline{X}))$ such that
\begin{equation*}
(a(X, L), b(X, L)) \leq (a(\mathcal Y_{i, j}^\sigma, (s_{i,j}^\sigma)^*L), b(F, \mathcal Y_{i, j}^\sigma, (s_{i, j}^\sigma)^*L))
\end{equation*}
and the map is face contracting we add the set
\begin{equation*}
 \bigcup_{\sigma} s_{i,j}^\sigma(\mathcal Y_{i,j}^\sigma (F)) \subset X(F)
 \end{equation*}
 to $Z'$.  We repeat this process for each of the finitely many $\mathcal{Y}_{i,j}$.  Since the $\mathcal{Y}_{i,j}$ are geometrically integral by Lemma~\ref{lemm: finitelymanycoversovernf} (1), $Z'$ is contained in a thin set of $X(F)$ by Theorem~\ref{theo:twists}.  We show that $Z_{1}$ and $Z_{2}$ are contained in $Z'$.

\textbf{The set $Z_{1}$:}  Assume that $f: Y \rightarrow X$ is a thin map such that $Y$ is smooth and geometrically integral, $d(Y,f^{*}L) < d(X,L)$, and
\begin{equation*}
(a(X,L),b(F,X,L)) \leq (a(Y,f^{*}L),b(F,Y,f^{*}L)).
\end{equation*}
We would like to show that for any rational point $y \in Y(F)$ the image $f(y) \in Z'$. We may assume that $f(y) \not\in V$ since otherwise the statement is clear.
This condition implies that $a(Y,f^{*}L) =  a(X,L)$. 

We will next perform several constructions improving the properties of $Y$ so that we may apply Lemma \ref{lemm: finitelymanycoversovernf} (6).  Let $\phi : Y \dashrightarrow B$ be the canonical map for $K_{Y} + a(Y, f^*L) f^*L$. After replacing $Y$ by a birational model (and taking any preimage of $y$), we may assume that the canonical map is a morphism.  If the $a$-values of the images of the general fibers of $\phi$ were larger than $a(X,L)$ then we would have $f(Y) \subset V$, so by assumption we must have an equality instead.  Similarly, since the fibers of the canonical map for $Y$ are adjoint rigid with respect to $f^{*}L$, their images also must be adjoint rigid with respect to $L$ by Lemma \ref{lemm:dominantequalitycase}. Thus $B$ admits a rational map $g : B \dashrightarrow W_i$ for some $i$. After some birational modification (and again taking a preimage of $y$), we may assume that this rational map is a morphism. Then $f:Y \rightarrow X$ rationally factors through $f' : Y \dashrightarrow B \times_{W_i} \mathcal U_i \rightarrow \mathcal U_i$.
After again replacing $Y$ by a birational model and replacing $y$ by any preimage, we may suppose that $f'$ is a morphism and is generically a map of good families of adjoint rigid varieties. 
We may assume that $\mathcal U_i$ is geometrically irreducible and $s_i : \mathcal U_i \rightarrow X$ is birational as otherwise the desired containment of rational points is clear.

We may now apply Lemma~\ref{lemm: finitelymanycoversovernf} (6) to see that there exists $j$ and a twist $\sigma$ such that
\[
f(y) \in s_{i, j}^\sigma(\mathcal Y_{i, j}^\sigma(F)),
\]
and $f : Y \rightarrow X$ factors through $s_{i,j}^{\sigma}: \mathcal Y_{i, j}^\sigma \to X$ after an Iitaka base change.
 It only remains to verify
 \begin{equation*}
(a(X, L), b(F,X, L)) \leq (a(\mathcal Y_{i,j}^{\sigma}, (s_{i,j}^{\sigma})^*L), b(F, \mathcal Y_{i,j}^{\sigma}, (s_{i,j}^{\sigma })^*L)),
\end{equation*}
and if equality is achieved then $s_{i,j}^{\sigma}$ is face contracting. By the construction we know that the $a$-values are the same.

We first show that $b(F,Y,f^{*}L) \leq b(F, \mathcal Y_{i,j}^{\sigma}, (s_{i,j}^{\sigma })^*L)$.  
Let $c \in B$ be a general closed point.  By applying Lemma \ref{lemm: monodromyandbvalue} (3) we obtain
\[
b(F, Y, f^*L) \leq b(F, Y_{c}, f^{*}L).
\]
Since by assumption $f(y) \not \in V$ and $Y_{c}$ is general, $Y_{c}$ will map to a fiber of $\mathcal{Y}_{i,j}^{\sigma}$ lying over $T_{i,j}^{\sigma \circ}$.  Let $t \in T_{i, j}^{\sigma \circ}$ denote the image of $c$ and let $\mathcal{Y}^{\sigma}_{i,j,t}$ denote the corresponding fiber of $q_{i,j}$.  By construction every geometric component of $Y_{c}$ is birational to a geometric component of $\mathcal{Y}^{\sigma}_{i,j,t}$.  Thus Lemma~\ref{lemm:birfaceinv} shows that the $b$-values of these two varieties with respect to $L$ are the same.  Finally, by Lemma~\ref{lemm: finitelymanycoversovernf} (5) we have an equality $b(F, \mathcal Y^{\sigma}_{i,j}, s^{\sigma*}_{i, j}L) = b(F, \mathcal Y^{\sigma}_{i,j,t}, s^{\sigma*}_{i, j}L|_{\mathcal{Y}^{\sigma}_{i,j,t}})$.  Together these inequalities show the desired statement.

Finally suppose that
\[
(a(X, L), b(F, X, L)) = (a(\mathcal Y_{i,j}^{\sigma}, (s_{i,j}^{\sigma})^*L), b(F, \mathcal Y_{i,j}^{\sigma}, (s_{i,j}^{\sigma })^*L))
\]
holds. Note that we assume $d(\mathcal{Y}^{\sigma}_{i,j},s_{i,j}^{\sigma*}L) = d(Y,f^{*}L) < d(X,L)$. Thus
\begin{align*}
\dim(W_{i}) = \dim(T^{\sigma}_{i,j}) & = \dim(\mathcal{Y}^{\sigma}_{i,j}) - d(\mathcal{Y}^{\sigma}_{i,j},s_{i,j}^{\sigma*}L) \\
& >   \dim (X) - d(X,L) = \kappa(X,K_{X} + a(X,L)L).
\end{align*}
By applying Lemma~\ref{lemm: facecontractingcondition} to $s_{i,j}^{\sigma}$ and using the birational equivalence of $X$ and $\mathcal{U}_{i}$ we see that $s_{i,j}^{\sigma}$ is face contracting. 

\textbf{The set $Z_{2}$:}  Assume that $f: Y \rightarrow X$ is a thin map such that $Y$ is smooth and geometrically integral, $d(Y,f^{*}L) = d(X,L)$, and either
\begin{equation*}
(a(X,L),b(F,X,L)) < (a(Y,f^{*}L),b(F,Y,f^{*}L))
\end{equation*}
or equality is achieved and $f$ is face contracting.  We would like to show that $f(Y(F)) \subset Z'$.

The argument is essentially the same as for the set $Z_{1}$.  The only difference is the case when the $a$ and $b$ values are equal.  In this situation, we must show that if $f: Y \to X$ is face contracting then the map $s_{i,j}^{\sigma}: \mathcal Y_{i,j}^{\sigma} \to X$ is also face contracting. Let $\widetilde{Y}$ be a smooth birational model of an Iitaka base change of $Y$ such that $\widetilde{f} : \widetilde{Y} \rightarrow X$ factors  through $s_{i, j}^\sigma : \mathcal Y_{i, j}^\sigma \rightarrow X$.  Lemma \ref{lemm: iitakabasechangeandbvalue} gives a surjection $\mathcal{F}_{\widetilde{Y},\widetilde{f}^{*}L} \to \mathcal{F}_{Y,f^{*}L}$.  Thus it suffices to show that the map $\mathcal{F}_{\widetilde{Y},\widetilde{f}^{*}L} \to \mathcal{F}_{\mathcal Y_{i,j}^{\sigma},s_{i,j}^{\sigma *}L}$ is injective.
As before choose a general fiber $\widetilde{Y}_{c}$ of the canonical map for $\widetilde{Y}$ and let $\mathcal Y_{i,j,t}^{\sigma}$ denote the corresponding fiber of the family map for $\mathcal Y_{i,j}^{\sigma}$.  Recall that by assumption $\widetilde{f}(\widetilde{Y})$ is not contained in $V$.
Since $\widetilde{Y}_{c}$ is general, every geometric component of $\widetilde{Y}_{c}$ is birational to a geometric component of $\mathcal{Y}^{\sigma}_{i,j,t}$ and the image $t$ of $c$ is contained in $T_{i,j}^{\sigma \circ}$.  
Consider the maps
\begin{equation*}
\xymatrix{
 \mathcal{F}_{\widetilde{Y}_{c},\widetilde{f}^{*}L|_{\widetilde{Y}_{c}}} \ar@{>}[r] \ar@{>}[d]& \mathcal{F}_{\widetilde{Y},\widetilde{f}^{*}L} \ar@{>}[d]\\
 \mathcal{F}_{\mathcal Y_{i,j,t}^{\sigma},s_{i,j}^{\sigma *}L|_{\mathcal{Y}_{i,j,t}^{\sigma}}}  \ar@{>}[r] & \mathcal{F}_{\mathcal Y_{i,j}^{\sigma},s_{i,j}^{\sigma *}L} 
}
\end{equation*}
The arrow on the left is an isomorphism by Lemma~\ref{lemm:birfaceinv}.  By Lemma~\ref{lemm: finitelymanycoversovernf} (5) the arrow on the bottom is an isomorphism and by Lemma \ref{lemm: monodromyandbvalue} (3) the arrow on the top is surjective.  This implies that all the arrows are isomorphisms. Thus our assertion follows.
\end{proof}

\bibliographystyle{alpha}
\bibliography{balancedVI}

\end{document}